\documentclass[a4paper,fleqn]{cas-sc}

\usepackage[numbers]{natbib}
\usepackage{amsthm}
\usepackage{amsmath}
\usepackage{amssymb}
\usepackage{verbatim}
\usepackage{cleveref}
\usepackage{subfig}
\usepackage{float}
\usepackage{color}
\usepackage[printonlyused]{acronym}

\usepackage{booktabs}
\usepackage{tabularx}
\usepackage{makecell}
\usepackage{array}
\usepackage{rotating}

\crefalias{subequation}{equation}
\crefalias{eqnarray}{equation}
\crefformat{pluraleq}{(#2#1#3)}

\usepackage{bm}
\usepackage[ruled]{algorithm2e}

\newcommand{\bb}{\boldsymbol{b}}

\newcommand{\bh}{\boldsymbol{h}}

\newcommand{\bn}{\boldsymbol{n}}

\newcommand{\bu}{\boldsymbol{u}}
\newcommand{\bv}{\boldsymbol{v}}

\newcommand{\bx}{\boldsymbol{x}}

\newcommand{\bU}{\boldsymbol{U}}

\newtheorem{theorem}{Theorem}
\newtheorem{assumption}{Assumption}

\newtheorem{definition}{Definition}
\newtheorem{corollary}{Corollary}
\newtheorem{lemma}{Lemma}

\def\tsc#1{\csdef{#1}{\textsc{\lowercase{#1}}\xspace}}
\tsc{WGM}
\tsc{QE}
\tsc{EP}
\tsc{PMS}
\tsc{BEC}
\tsc{DE}

\begin{document}
\let\WriteBookmarks\relax
\def\floatpagepagefraction{1}
\def\textpagefraction{.001}
\shorttitle{MULTILEVEL NEURAL CORRECTIONS}
\shortauthors{Yu Yang et~al.}

\title [mode = title]{ Multilevel Neural Corrections with Adaptive Sampling for Partial Differential Equations}                      
\tnotemark[1]   

\tnotetext[1]{This research is partially sponsored by the National Natural Science Foundation of China (No.12371434, No. U25A20200), the National Key R \& D Program of China (No.2022YFE03040002), Sichuan University Research Project on Teaching Reform in Graduate Education (No. SCYZTD2025003).}


        \author[1]{Yu Yang}
        [type=editor,
        auid=000,bioid=1,
        orcid=0009-0005-1428-6745
        ]
        \ead{yangyu1@stu.scu.edu.cn}

        \address[1]{School of Mathematics, Sichuan University, 610065, Chengdu, China.}

        \author[1]{Qiaolin He}[style=chinese]
                                
        \cormark[1] 
        
        \ead{qlhejenny@scu.edu.cn}
        
        

        \cortext[cor1]{Corresponding author}

\begin{abstract}
We propose a multi-level framework for deep neural PDE solvers that integrates adaptive sampling with additive neural corrections. At each level, an adaptive sampling distribution is constructed from normalized residual-gradient indicators inherited from the preceding levels, thereby concentrating collocation points in regions where the current approximation exhibits large local errors or reduced regularity. The neural approximations obtained at previous levels are incorporated into a successive correction problem, leading to a hierarchical refinement of the numerical solution. Under suitable assumptions, we establish rigorous approximation results and derive error bounds for the proposed multi-level framework, which quantify the cumulative effect of the successive neural corrections across levels. 
Numerical experiments on representative PDEs with steep gradients and low-regularity solutions demonstrate the effectiveness of the adaptive multi-level strategy.
\end{abstract}




\maketitle

\section{Introduction}

Deep neural networks have emerged as a flexible computational paradigm for solving partial differential equations (PDEs). Existing approaches can be broadly categorized according to how the underlying physical laws are incorporated into the learning process. Physics-Informed Neural Networks (PINNs) \cite{PINN} and the Deep Galerkin Method \cite{DGM} enforce the governing equations through strong-form residuals, whereas the Deep Ritz method \cite{DeepRitz} and Weak Adversarial Networks \cite{WAN} are based on variational or weak formulations. In a different direction, operator-learning approaches, such as Fourier Neural Operators (FNOs) \cite{FNO} and DeepONet \cite{DeepONet}, aim to learn mappings between function spaces and thereby approximate solution operators over families of PDEs.

Despite their considerable flexibility, neural-network-based PDE solvers still face fundamental challenges arising from the approximation and optimization properties of deep networks. These issues are particularly relevant for PINNs, which typically employ fully connected neural networks to approximate the solution while minimizing residuals of the governing equations and associated constraints. A well-documented phenomenon is the frequency principle, or spectral bias, according to which neural networks tend to capture low-frequency components more readily than high-frequency ones \cite{xu2018frequency,xu2024overview,xu2020frequency}. Moreover, the approximation of solutions containing steep gradients, localized structures, or limited regularity can be substantially more difficult. For strong-form PINNs, this difficulty may be further amplified by the use of pointwise differential residuals, leading to degraded local accuracy near sharp transitions or low-regularity regions \cite{lu2021deepxde,yu2022gradient}. In the case of conservation laws with discontinuous solutions, pointwise residuals evaluated on smooth neural approximations may even become unbounded as the approximation approaches a shock \cite{de2022weak}.

Classical relaxation schemes exhibit a frequency-selective error-reduction property: Jacobi and Gauss--Seidel iterations efficiently damp high-frequency errors but are less effective for low-frequency components. Multigrid methods overcome this limitation through coarse-grid correction. This complementary behavior has motivated hybrid iterative--PINN solvers \cite{dong2026pinn} and multilevel neural architectures such as MgNet \cite{he2019mgnet} and PDE-MgNet \cite{chen2022meta}. Our framework is inspired by the multilevel correction principle of multigrid, but operates directly on collocation points and therefore does not require an underlying mesh hierarchy.

In this mesh-free setting, the distribution of collocation points becomes central to resolving steep gradients and low-regularity features. Adaptive sampling strategies therefore seek to concentrate training points in poorly resolved regions. Representative approaches include residual-based adaptive distributions (RAD) \cite{wu2023comprehensive}, failure-informed sampling \cite{FIPINN}, KRNet-based deep adaptive sampling \cite{das-pinns,tang2020deep}, moving-mesh redistribution \cite{MSPINN}, and gradient-based point adaptation \cite{deng2026adaptive}. Complementary multistage strategies progressively enrich or modify the approximation, for example by adapting network size \cite{xu2025multi}, rescaling amplitudes and spatial scales \cite{wang2024multi}, or correcting deficiencies in amplitude and frequency representation \cite{aldirany2024multi}. Related work also constructs enriched approximation spaces using newly generated neural basis functions followed by collocation least-squares recomputation of their coefficients \cite{weng2026deep}. These developments suggest that adaptive sampling and successive correction provide two complementary mechanisms for improving neural approximations of multiscale and low-regularity PDE solutions.

Existing multi-level neural solvers have demonstrated the effectiveness of additive correction, whereas adaptive sampling strategies typically rely only on information from the current approximation. We combine these two ideas by constructing the collocation distribution at each level from both the current correction problem and normalized residual-gradient information accumulated from preceding levels. The main contributions of this work are summarized as follows:

		\begin{itemize}
			\item  
            We develop a multilevel adaptive collocation strategy that combines the current PDE residual with a solution-feature indicator while incorporating normalized information inherited from previous correction levels.
			\item 
            We integrate this history-dependent sampling strategy with an additive multilevel neural correction procedure. Unlike existing approaches that treat successive correction and adaptive sampling largely separately, the proposed framework couples the evolving correction equations with collocation distributions informed by preceding levels.
			\item 
            We establish approximation and error estimates under explicit assumptions and validate the resulting framework through numerical experiments, with particular emphasis on the effect of history-dependent sampling on successive correction problems.
		\end{itemize}
        
The remainder of the paper is organized as follows. Section~2 introduces the PDE setting and the PINN formulation. Section~3 presents the proposed multilevel adaptive framework together with its approximation and error analysis. Section~4 reports the numerical experiments, and Section~5 concludes the paper. Appendices~A provide supplementary proofs.

\section{Preliminaries}
\label{sec:Preliminary Work}
\subsection{Problem formulation}
\label{sec:Problem formulation}
Let $\Omega\subset\mathbb{R}^d$ be a bounded Lipschitz domain with boundary $\partial\Omega$. We consider the boundary-value problem
\begin{equation}\label{eq:gen_PDE}
\mathcal{A}[\bu(\bx)]=f(\bx),\quad \bx\in\Omega,
\qquad
\mathcal{B}[\bu(\bx)]=g(\bx),\quad \bx\in\partial\Omega,
\end{equation}
where $\mathcal{A}$ and $\mathcal{B}$ denote the differential and boundary operators, respectively. We assume that \eqref{eq:gen_PDE} is well posed in an appropriate function space $X $ and admits a unique solution $\bu\in X$.

\subsection{PINN formulation}
\label{sec:Brief introduction to PINNs}
Let $\bu(\bx; \theta)$ denote a neural network approximation of the solution $\bu$, with trainable parameters $\theta$. We define the interior and boundary residuals by $r(\bx; \theta) : =\mathcal{A}[\bu(\bx;\theta)]-f(\bx)$ and $b(\bx; \theta):=\mathcal{B}[\bu(\bx;\theta)]-g (\bx)$.
The continuous loss is defined as 
\begin{equation}\label{eq:L2_loss}
\mathcal{L}(\bu(\bx; \theta))
=\alpha_1\int_\Omega |r(\bx;\theta)|^2\,d\bx
+\alpha_2\int_{\partial\Omega}|b(\bx;\theta)|^2\,d\bx,
\end{equation}
where $\alpha_1,\alpha_2>0$ are weighting parameters. 

Let $\{\bx_i^r\}_{i=1}^{N_r}\subset\Omega$ and
$\{\bx_i^b\}_{i=1}^{N_b}\subset\partial\Omega$ be the interior and boundary collocation points, respectively. The corresponding empirical loss is
\begin{equation}\label{eq:L2_empiricalloss}
\begin{aligned}
\mathcal{L}_{N_r,N_b}(\bu(\bx; \theta))
& =\alpha_1\Vert r(\bx;\theta)\Vert^2_{2,N_r}
 +\alpha_2\Vert b(\bx;\theta)\Vert^2_{2,N_b}\\
&=\frac{\alpha_1}{N_r}\sum_{i=1}^{N_r}|r(\bx_i^r;\theta)|^2
+\frac{\alpha_2}{N_b}\sum_{i=1}^{N_b}|b(\bx_i^b;\theta)|^2\\
&=\alpha_1\mathcal{L}_{r,N_r}(\theta)+\alpha_2\mathcal{L}_{b,N_b}(\theta).
\end{aligned}
\end{equation}
The PINN approximation is obtained by minimizing
\eqref{eq:L2_empiricalloss} with respect to the network parameters.

\section{Main results}
\label{sec:MainlyWork}

\subsection{Multi-Level Sampling}
\label{sec:MLS}
Multi-level sampling (MLS) adaptively concentrates interior collocation points in regions where the current approximation exhibits large PDE residuals or pronounced solution variations, while retaining information identified at preceding levels. Boundary collocation points are sampled independently; alternatively, the boundary conditions may be imposed exactly by an appropriate trial-function construction \cite{lyu2020enforcing,lyu2022mim}.
Let
$\bu_j(\bx) := \bu(\bx;\theta^j)$ and $\bU_j(\bx) :=\sum_{i=1}^j\bu_i(\bx)$ 
denote the neural correction at level $j$ and the corresponding accumulated approximation, respectively. At the first level, $\bu_1$ is trained on an initially uniform set of interior collocation points $\mathcal X^1_{N_r}\subset\Omega$.
Based on the resulting approximation $\bU_1 =\bu_1$, we define the residual and solution-feature indicators

\begin{align}
\mathcal R_1(\bx)&=|\mathcal A[\bU_1(\bx)]-f(\bx)|, \quad 
\mathcal H_1(\bx)=\gamma|\nabla\bU_1(\bx)|,\label{eq:1st_Hu}
\end{align}
where $\gamma>0$ balances the two indicators and the gradient is evaluated by automatic differentiation (AD). The corresponding normalized indicator is
\begin{align}
q_1(\bx)\propto [\mathcal R_1(\bx)+\mathcal H_1(\bx)]^s, \quad \text{i.e.,} \quad q_1(\bx)=\frac{[\mathcal R_1(\bx)+\mathcal H_1(\bx)]^s}
{\int_\Omega[\mathcal R_1(\bx)+\mathcal H_1(\bx)]^s\,d\bx},
\label{eq:1st_Wu}
\end{align}

where $s>0$ controls the concentration of the sampling distribution.

For the second level, the history monitor and sampling density are defined by
\begin{align}
\mathcal W_1(\bx)&=q_1(\bx),\label{eq:2nd_Wu}\\
\rho_1^{\mathrm{MLS}}(\bx)
&\propto\mathcal W_1(\bx)+c,\label{eq:2nd_pu}
\end{align}
 where $c > 0$ providing a baseline sampling probability in regions where the monitor is small. Interior samples drawn according to $\rho_1^{\mathrm{MLS}}$ form the level-2 collocation set $\mathcal X_{N_r}^{2}$. At subsequent levels, the monitor incorporates normalized information inherited from previous correction stages, thereby preserving previously identified difficult regions while adapting to the current correction problem. The parameter settings are summarized in Table~\ref{tab:parameter-MLS}, and the general MLS procedure for levels $k>2$ is presented in Algorithm~\ref{alg:MLS}.
\begin{algorithm}[htbp]
\caption{Multi-level sampling for level $k > 2$}\label{alg:MLS}
\textbf{Input:}
Trained corrections $\{\bu_j\}_{j=1}^{k-1}$, stored indicators $\{q_j\}_{j=1}^{k-2}$,
sampling parameters $\gamma,s,\mu,c$, and sample size $N_r$.\\
\textbf{Output:}
Interior collocation set $\mathcal X_{N_r}^{k}$.\\
\textbf{1. Current-level monitor.}
Set $\bU_{k-1}:=\sum_{j=1}^{k-1}\bu_j$,  and compute
\[ \mathcal R_{k-1}(\bx) = \left| \mathcal A[\bU_{k-1}(\bx)]-f(\bx)\right|,
\quad
\mathcal H_{k-1}(\bx)
=
\gamma|\nabla\bU_{k-1}(\bx)|.\]
Define the normalized indicator 
\[q_{k-1}(\bx) \propto
[\mathcal R_{k-1}(\bx)+\mathcal H_{k-1}(\bx)]^s.
\]\\
\textbf{2. History-dependent sampling density.}
Form
\[
\mathcal W_{k-1}(\bx)
=
q_{k-1}(\bx)
+
\mu\sum_{j=1}^{k-2}(k-j-1)q_j(\bx),
\qquad \mu>0,
\]
and define
\[
\rho_{k-1}^{\mathrm{MLS}}(\bx)
\propto\mathcal W_{k-1}(\bx)+c, 
\qquad c>0.
\]
\\
\textbf{3. Adaptive sampling.}
Draw $N_r$ interior collocation points independently from
$\rho_{k-1}^{\mathrm{MLS}}$ to obtain
$\mathcal X_{N_r}^{k}$.
\end{algorithm}

\subsection{Multi-Level Training}
\label{sec:MLT}

Inspired by the successive correction principle of multigrid methods, our framework progressively resolves components of the solution that remain insufficiently captured by the current neural approximation. After each level is trained, the residual of the accumulated approximation is evaluated over the domain, and the MLS strategy concentrates new collocation points in regions that remain difficult to resolve. A new neural network is then trained to approximate a correction to the current solution. 
The multi-level training (MLT) procedure constructs the approximation through successive additive corrections. At level 1, a neural solver is trained for \eqref{eq:gen_PDE} using initially sampled interior and boundary collocation points, yielding the approximation $\bu_1$. For each subsequent level ($k\geq 2$), MLS generates the interior collocation set $\mathcal X_{N_r}^{k}$, which is combined with the boundary set $\mathcal X_{N_b}$. Let $\bu^*$ denote the exact solution and define the remaining error by $\bv_{k-1} :=\bu^*-\bU_{k-1}$. With the accumulated approximation $\bU_{k-1}$ fixed, the correction satisfies
\begin{equation}\label{eq:error_PDEk}
\left\{\begin{aligned}
\mathcal A[\bU_{k-1}+\bv_{k-1}]&=f,&&\bx\in\Omega,\\
\mathcal B[\bU_{k-1}+\bv_{k-1}]&=g,&&\bx\in\partial\Omega.
\end{aligned}\right.
\end{equation}
A new predicted solution $\bu_k$ is trained to approximate $\bv_{k-1}$, and the accumulated approximation is updated according to $\bU_k=\bU_{k-1}+\bu_k$. At each level, previously trained networks are kept fixed and evaluated, together with the derivatives required by $\mathcal A$ and $\mathcal B$, directly at the newly sampled collocation points using AD. Hence, the multilevel correction is performed entirely in a mesh-free setting and requires neither grid-transfer operators nor interpolation between discretization levels. The overall MLS--MLT framework is illustrated in Figure~\ref{fig:flowchart}.

\begin{algorithm}[htbp]
\caption{Multi-level training for level $k \geq 2$}\label{alg:MLT}
\textbf{1.} Train $\bu_1$ for \eqref{eq:gen_PDE} on uniform collocation points and set
$\bU_1=\bu_1$. \\ 
\textbf{2.} For $k=2,3,\ldots$: Generate $\mathcal X_{N_r}^{k}$ using Algorithm~\ref{alg:MLS}. \\
\textbf{3.} Evaluate $\bU_{k-1}$ and the required derivatives at the new points. \\
\textbf{4.} Train $\bu_k$ for \eqref{eq:error_PDEk} and update
\[
\bU_k=\bU_{k-1}+\bu_k.
\] \\
\textbf{5.} Stop if
\[
\|\bu_k\|_{L^2(\Omega)}
/
\|\bU_k\|_{L^2(\Omega)}
\leq\eta.
\]
\end{algorithm}

\begin{figure}[htbp]
\centering
\includegraphics[width=0.95\textwidth]{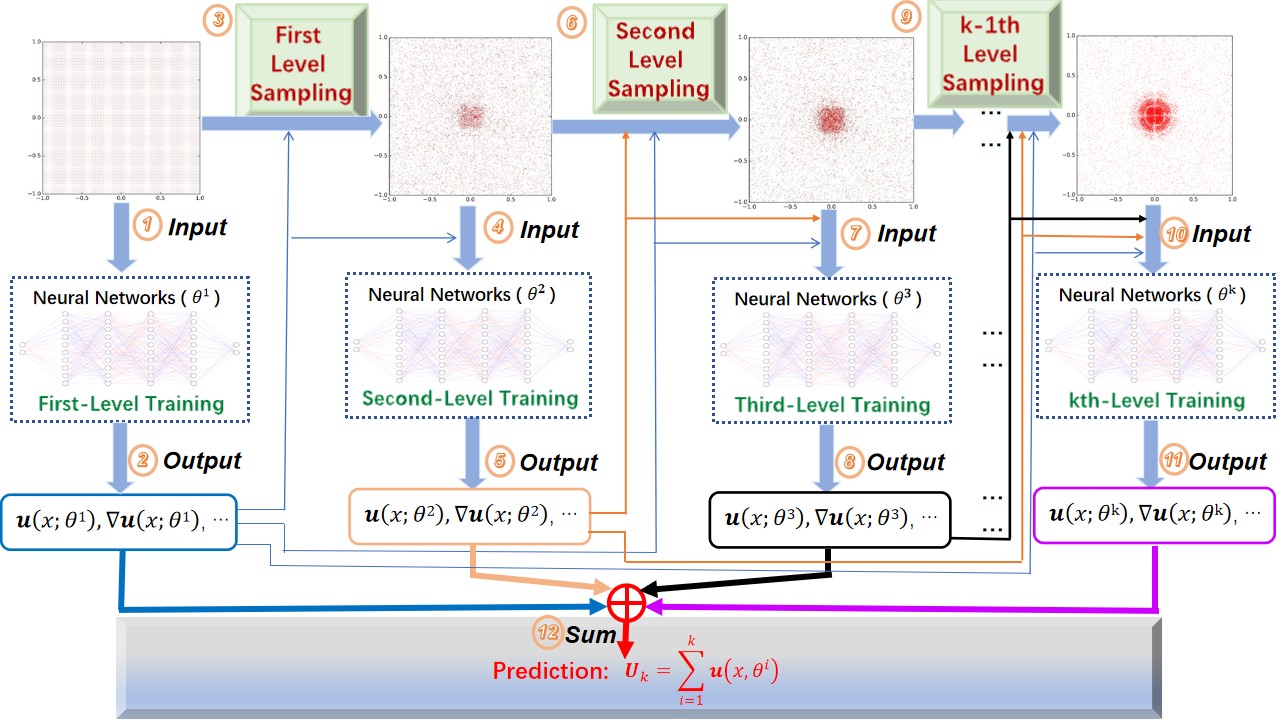}
\caption{Flowchart of our method.}
\label{fig:flowchart}
\end{figure}

\subsection{Error Estimation}
\label{sec:Error estimation}

We first study the approximation error, followed by the sampling-induced estimation error associated with \(\rho_{k-1}^{\mathrm{MLS}}\) \cite{shin2020convergence}. Under suitable stability, continuity, and empirical minimization assumptions, we then derive sampling-dependent error estimates for linear PDEs. Proofs are deferred to Appendix A.

Let $\mathbb H$ be a Hilbert space, and let $\mathcal D_1^{NN}\subset\mathbb H$ denote the admissible class of
fully connected neural networks with input dimension $d$, output dimension $p$, and $l_s$ layers of prescribed widths $[m_1,\ldots,m_{l_s}]$. The activation function is assumed to be sufficiently smooth and nonlinear, such as  $\tanh$ or the sigmoid function. For each layer $ 1\leq l\leq l_s$, the weight matrix $W_l$ and bias vector
$b_l$ satisfy $\|W_l\|_{L^\infty(\Omega)} \leq M_l$, $\|b_l\|_{L^\infty(\Omega)} \leq M_l$, where $M_l > 0$ is a prescribed constant.

\begin{lemma}\label{lem:1}
The admissible network class $\mathcal D_1^{NN}$ is star-shaped with respect to the
origin; that is,
\begin{align*}
0\in \mathcal D_1^{NN}, \quad \xi \bu\in \mathcal D_1^{NN}, 
\quad
\forall\,\bu\in \mathcal D_1^{NN} \  \mbox{and} \ \xi\in[0,1].
\end{align*}
\end{lemma}

\begin{definition}\label{de:distance}
For a nonempty set $D\subset\mathbb H$ and $\bu^*\in\mathbb H$, define
\[
\operatorname{dist}_{\mathbb H}(\bu^*,D)
:=
\inf_{\bu\in D}\|\bu-\bu^*\|_{\mathbb H}.
\]
Any $\widehat{\bu}\in D$ attaining this infimum is called a best
approximation of $\bu^*$ from $D$.
\end{definition}

For fixed corrections $\bu_1,\ldots,\bu_{k-1}$, let
$\bU_{k-1}:=\sum_{i=1}^{k-1}\bu_i$. The admissible approximation class at level $k\geq2$ is defined by
\begin{equation}\label{eq:level_k_class}
\mathcal D_k^{\mathrm{NN}}
(\bu_1,\ldots,\bu_{k-1})
:=
\left\{
\bU_{k-1}+\bu_k:
\bu_k\in\mathcal D_1^{\mathrm{NN}}
\right\}.
\end{equation}
When the previously computed corrections are fixed and clear from the
context, we simply write $\mathcal D_k^{\mathrm{NN}}$. For $k=1$, we set
$\mathcal D_1^{\mathrm{NN}}$ as defined above.

\begin{assumption}[Existence of best approximations]\label{asu:u_opt}
For each $k\geq 1$, there exists
$\widehat{\bU_k}\in\mathcal D_k^{\mathrm{NN}}$ such that
\[
\|\widehat{\bU_k}-\bu^\ast\|_{\mathbb H}
=
\operatorname{dist}_{\mathbb H}
\bigl(\bu^\ast,\mathcal D_k^{\mathrm{NN}}\bigr).
\]
\end{assumption}

\begin{theorem}\label{thm:2}
For any $\bu_1\in D_1^{NN}$ and $k\geq2$, Assumption~\ref{asu:u_opt} guarantees the existence of a sequence $\{\bu_i\}_{i=1}^k\subset D_1^{NN}$ such that, if $\bu_i\ne0$ for $2\leq i\leq k$, then $\{\|\bU_i-\bu^*\|_{\mathbb H}\}\}_{i=1}^k$ is strictly decreasing.
\end{theorem}

For the remainder of the analysis, we take $\alpha_1=\alpha_2=1$ in \eqref{eq:L2_loss}. Let $X$ be a Hilbert space continuously embedded in $L^2(\Omega;\mathbb R^p)$, and let $\mathcal A:X\to L^2(\Omega;\mathbb R^{p_r})$ and $\mathcal B:X\to L^2(\partial\Omega;\mathbb R^{p_b})$ be linear operators.

For a probability density $\rho$ on $\Omega$, define
\[
\|v\|_{L^2_\rho(\Omega)}
:=
\left(
\int_\Omega |v(\bx)|^2\rho(\bx)\,d\bx
\right)^{1/2}.
\]

\begin{assumption}[Stability and continuity]\label{asu:Normrelations}
There exist constants $C_1,C_2>0$, independent of $k$, such that for every $\bu\in X$, 
\begin{align}
C_1\|\bu\|_{L^2(\Omega)}
&\leq
\|\mathcal A\bu\|_{L^2_{\rho_{k-1}^{\mathrm{MLS}}}(\Omega)}
+\|\mathcal B\bu\|_{L^2(\partial\Omega)},
\label{eq:pde-stability}\\
\|\mathcal A\bu\|_{L^2_{\rho_{k-1}^{\mathrm{MLS}}}(\Omega)}
+\|\mathcal B\bu\|_{L^2(\partial\Omega)}
&\leq
C_2\|\bu\|_X,
\label{eq:operator-continuity}
\end{align}
where $\rho_{k-1}^{\mathrm{MLS}}$ denotes the interior sampling density used
at level $k$.
\end{assumption}

\begin{definition}\label{de:Rademacher}
For i.i.d. samples $\mathcal X=\{x_i\}_{i=1}^n$ and independent Rademacher signs $\epsilon_i$ with a probability $P(\epsilon_i=\pm1)=1/2$, the Rademacher complexity of $\mathcal F$ is
\[
\mathcal R_n(\mathcal F)=\mathbb E_{\mathcal X,\epsilon}
\left[\sup_{f\in\mathcal F}\left|\frac1n\sum_{i=1}^n\epsilon_i f(x_i)\right|\right].
\]
\end{definition}

At level $k$, the interior collocation points are sampled according to
$\rho_{k-1}^{\mathrm{MLS}}$, whereas the boundary points are sampled
uniformly with respect to the boundary measure. Accordingly, the residuals
defined in Section~\ref{sec:Brief introduction to PINNs} induce the classes
\[
\mathcal F_{k,r}^{\mathrm{NN}}
:=
\{\mathcal A[\bu]-f:\bu\in\mathcal D_k^{\mathrm{NN}}\},
\qquad
\mathcal F_{k,b}^{\mathrm{NN}}
:=
\{\mathcal B[\bu]-g:\bu\in\mathcal D_k^{\mathrm{NN}}\}.
\]

\begin{assumption}[Uniform boundedness of residuals]\label{asu:r-boundary}
For each fixed level $k$, there exist constants
$B_{k,r},B_{k,b}>0$ such that 
\begin{equation}\label{eq:r-boundary}
\begin{aligned}
\sup_{\bu\in\mathcal D_k^{NN}}\|\mathcal A[\bu]-f\|_{L^\infty(\Omega)}&\leq B_{k,r}, \qquad 
\sup_{\bu\in\mathcal D_k^{NN}}\|\mathcal B[\bu]-g\|_{L^\infty(\partial \Omega)}&\leq B_{k,b}.
\end{aligned}
\end{equation}
\end{assumption}

\begin{lemma}[\cite{wainwright2019high}]
\label{lem:Uniform laws of large numbers-Rademacher}
For a $\beta$-uniformly bounded class $\mathcal F$ of integrable real-valued functions, $n\geq1$, and $\delta\ge 0$,
\[
\sup_{f\in\mathcal F}\left|\frac1n\sum_{i=1}^n f(x_i)-\mathbb E[f(x)]\right|
\leq2\mathcal R_n(\mathcal F)+\delta,
\]
with probability at least $1-\exp\left(-n\delta^2/(2\beta^2)\right)$.
\end{lemma}

\begin{theorem}\label{thm:3}
Let Assumptions~\ref{asu:Normrelations} and~\ref{asu:r-boundary} hold, and let $\bu^*$ be the solution of  \eqref{eq:gen_PDE}. Define the squared
residual classes
\[
\mathcal G_{k,r}^{\mathrm{NN}}
:=
\left\{
|r|^2:\,
r\in\mathcal F_{k,r}^{\mathrm{NN}}
\right\},
\qquad
\mathcal G_{k,b}^{\mathrm{NN}}
:=
\left\{
|b|^2:\,
b\in\mathcal F_{k,b}^{\mathrm{NN}}
\right\}.
\]
Let the interior and boundary collocation sets consist of $N_r$ and $N_b$
independent i.i.d.\ samples, respectively. Then, for any $\delta>0$, with
probability at least \[
\prod_{a\in\{r,b\}}
\left(
1-\exp\left\{
-\frac{N_a\delta^2}{8B_{k,a}^4}
\right\}
\right),
\]

the estimate
\begin{equation}\label{eq:errorestimate}
\begin{aligned}
\|\bU_k-\bu^\ast\|_{L^{2}(\Omega)}
\leq
\frac{\sqrt{2}}{C_1}
\left[
\|r\|_{2,N_r}^2
+\|b\|_{2,N_b}^2
+2\mathcal R_{N_r}(\mathcal G_{k,r}^{\mathrm{NN}})
+2\mathcal R_{N_b}(\mathcal G_{k,b}^{\mathrm{NN}})
+\delta
\right]^{1/2},
\end{aligned}
\end{equation}
holds uniformly for all $\bU_k \in\mathcal D_k^{\mathrm{NN}}$.
\end{theorem}

\begin{assumption}[Best-approximation decay]\label{asu:opt-*}
For each $k\geq1$, there exists
$\widetilde{\bU_k} \in\mathcal D_k^{\mathrm{NN}}$ satisfying
\[
\|\widetilde{\bU_k}-\bu^*\|_X
=
\operatorname{dist}_X
(\bu^*,\mathcal D_k^{\mathrm{NN}}).
\]
Moreover, the sequence of best-approximation errors is nonincreasing and
\[
\operatorname{dist}_X
(\bu^*,\mathcal D_k^{\mathrm{NN}})
\longrightarrow \delta_d\geq0
\qquad\text{as }k\to\infty.\]
\end{assumption}

\begin{theorem}[Asymptotic error estimate]\label{thm:4}

Suppose that the assumptions of Theorem~\ref{thm:3} and
Assumption~\ref{asu:opt-*} hold. Assume that, for each fixed $k$,
\[
\mathcal R_{N_r}(\mathcal G_{k,r}^{\mathrm{NN}})\to 0,
\qquad
\mathcal R_{N_b}(\mathcal G_{k,b}^{\mathrm{NN}})\to 0,
\]
as $N_r,N_b\to\infty$. Let
${\bU}^{*}_{k} \in\mathcal D_k^{\mathrm{NN}}$
be an empirical minimizer of empirical loss function $\mathcal{L}_{N_r,N_b}$ at level $k$. Then, for every $\varepsilon>0$,
\begin{equation}\label{eq:errorlimit1}
\lim_{k\to\infty}
\limsup_{N_r,N_b\to\infty}
\mathbb P\left(
\|{\bU}^{*}_{k}-\bu^\ast\|_{L^{2}(\Omega)}
>
\frac{\sqrt{2}C_2}{C_1}\delta_d+\varepsilon
\right)
=0.
\end{equation}
\end{theorem}

The argument of Theorem~\ref{thm:4} also yields the following level-wise comparison. 

\begin{corollary}\label{cor:1}
Under the assumptions of Theorem~\ref{thm:4}, for each fixed $k$,
\begin{equation}\label{eq:error_cor}
\|\bU_k^*-\bu^*\|_{L^{2}(\Omega)}
\leq
\frac{\sqrt{2}C_2}{C_1}
\left(
\|\widetilde{\bU_k}-\bu^*\|_X
\right)
+o_{\mathbb P}(1),
\qquad N_r,N_b\to\infty.
\end{equation}
Also we have, 
\[
\frac{\sqrt{2}C_2}{C_1}
\left(
\|\widetilde{\bU_k}-\bu^*\|_X
\right)
\leq
\frac{\sqrt{2}C_2}{C_1}
\left(
\|\widetilde{\bU_1}-\bu^*\|_X
\right).
\]
\end{corollary}
Thus, in probability, the level-$k$ upper bound is no larger than the first-level (PINNs under uniform sampling) bound under the stated assumptions. The above analysis is restricted to linear, well-posed PDEs satisfying the stated stability and continuity assumptions; the nonlinear problems considered below are used only for numerical assessment of the proposed method.

\section{Numerical Experiments}
\label{sec:Numericalsurvey}

\subsection{Settings}
\label{sec:Symbols and parameter settings}
In all numerical experiments, we employ SOAP \cite{vyas2024soap}, which
combines Adam \cite{Adam} with the eigenbasis of the Shampoo preconditioner, and the self-scaled Broyden (SSB) method
\cite{al1998numerical,al2005wide,urban2025unveiling}. 
Let $\bu^*$ denote the exact solution and $\bu(\bx;\theta)$ the numerical
approximation. For a test set
$\{\bx_i\}_{i=1}^{N_t}$, we evaluate the relative discrete
$L^\infty$ error and $L^2$ error. 
\begin{equation*}
\begin{aligned}
e_\infty(\bu(\bx;\theta))
& :=
\frac{
\displaystyle\max_{1\le i\le N_t}
|\bu^*(\bx_i)-\bu(\bx_i;\theta)|
}{
\displaystyle\max_{1\le i\le N_t}
|\bu^*(\bx_i)|
}, \\ 
e_2(\bu(\bx; \theta))
& :=
\frac{
\left(
\displaystyle\sum_{i=1}^{N_t}
|\bu^*(\bx_i)-\bu (\bx_i; \theta)|^2
\right)^{1/2}
}{
\left(
\displaystyle\sum_{i=1}^{N_t}
|\bu^*(\bx_i)|^2
\right)^{1/2}
}.
\end{aligned}
\end{equation*}

Unless otherwise specified, all experiments use random seed \(100\), with five independent runs additionally considered in Section~\ref{sec:Helmholtz_2D} to assess robustness to random initialization. Tables~\ref{tab:parameter-MLS} and~\ref{tab:parameter-PINN} summarize the default MLS and PINN settings, respectively, whereas Table~\ref{tab:all_hyperparameters} reports the case-specific sampling points, hyperparameters and corresponding errors. Only at the final level (level $k$), we adopt the loss-balancing concept from \cite{xiang2022self} and utilize the terminal residual loss and boundary loss from level $k-1$ to initialize the boundary weight $\alpha_2$ for level $k$.
Specifically, with $\alpha_1$ fixed, $\alpha_2$ is chosen so that $\alpha_1\mathcal L_{r,N_r} =\alpha_2\mathcal L_{b,N_b}.$

All computations are performed in double precision (FLOAT64) on an NVIDIA
A800 GPU with 80~GB memory. The average times per epoch for
SOAP/SSB are $0.063/0.474$~s at the first level and $0.072/0.806$~s at
subsequent levels, respectively. The MLS procedure requires, on average,
$0.545$~s per sampling update.

\begin{table}[htbp]
	\centering
	\caption{Default MLS and MLT parameters.}
	\label{tab:parameter-MLS}
	\setlength{\tabcolsep}{8mm}{
		\begin{tabular}{|c|c|c|c|c|}
			\hline
			 $\gamma$  & s &  c  &  $\mu$  &  $\eta$ \\
			\hline
			1 & 0.5 & 1 & 4 & $1 \times 10^{-3}$\\
            \hline
		\end{tabular}
	}
\end{table}

\begin{table}[htbp]
	\centering
	\caption{Default PINN parameters.}
	\label{tab:parameter-PINN}
	\setlength{\tabcolsep}{0.3mm}{
		\begin{tabular}{|c|c|c|c|c|c|c|c|}
			\hline
			 Torch  & Cuda &Activation &  Initialization  &  Loss Weights&Net &Random &Learning Rate\\
			 Version&  Version& Function  &            Method  &($\alpha_1$ / $\alpha_2$)& Size &Seed & (SOAP)\\
			\hline
			2.4.0 & 12.8 &Tanh & Xavier  & 1/1& $40 \times 4$ &100& $1\times 10^{-3} $  \\
            \hline
		\end{tabular}
	}
\end{table}

\subsection{Two-dimensional Poisson equation on the L-shaped domain}
\label{sec:Lshape_2D}
We consider the Poisson problem
\begin{equation}\label{eq:2d_Lshape}
\begin{cases}
-\Delta u=f, & \text{in } \Omega,\\
u=g, & \text{on } \partial\Omega,
\end{cases}
\end{equation}
on the L-shaped domain
\[
\Omega=(-1,1)^2\setminus\bigl((0,1)\times(-1,0)\bigr).
\]
The source term $f$ and boundary data $g$ are chosen according to the exact
solution
\begin{equation}\label{eq:2d_Lshapesolution}
u(x,y)
=
r^{2/3}\sin\left(\frac{2\theta}{3}\right)
-\frac{x^2+y^2-2}{4},
\end{equation}
where $r=(x^2+y^2)^{1/2}$ and $\theta\in[0,3\pi/2]$ denotes the polar angle.
The singular term $r^{2/3}\sin(2\theta/3)$ induces reduced regularity at the
re-entrant corner $(0,0)$. For MLS, we set $\gamma=10^{-3}$ and $c=0.1$ to balance the gradient
indicator and maintain baseline sampling away from highly weighted regions.
Figure~\ref{fig:2dLshape_points} shows that the adaptive samples concentrate
near the re-entrant corner, accompanied by a progressive reduction in the
approximation error across levels.

\begin{figure}[htbp]
\centering
\subfloat[Points distribution, level 1]
{\includegraphics[width=0.32\textwidth]{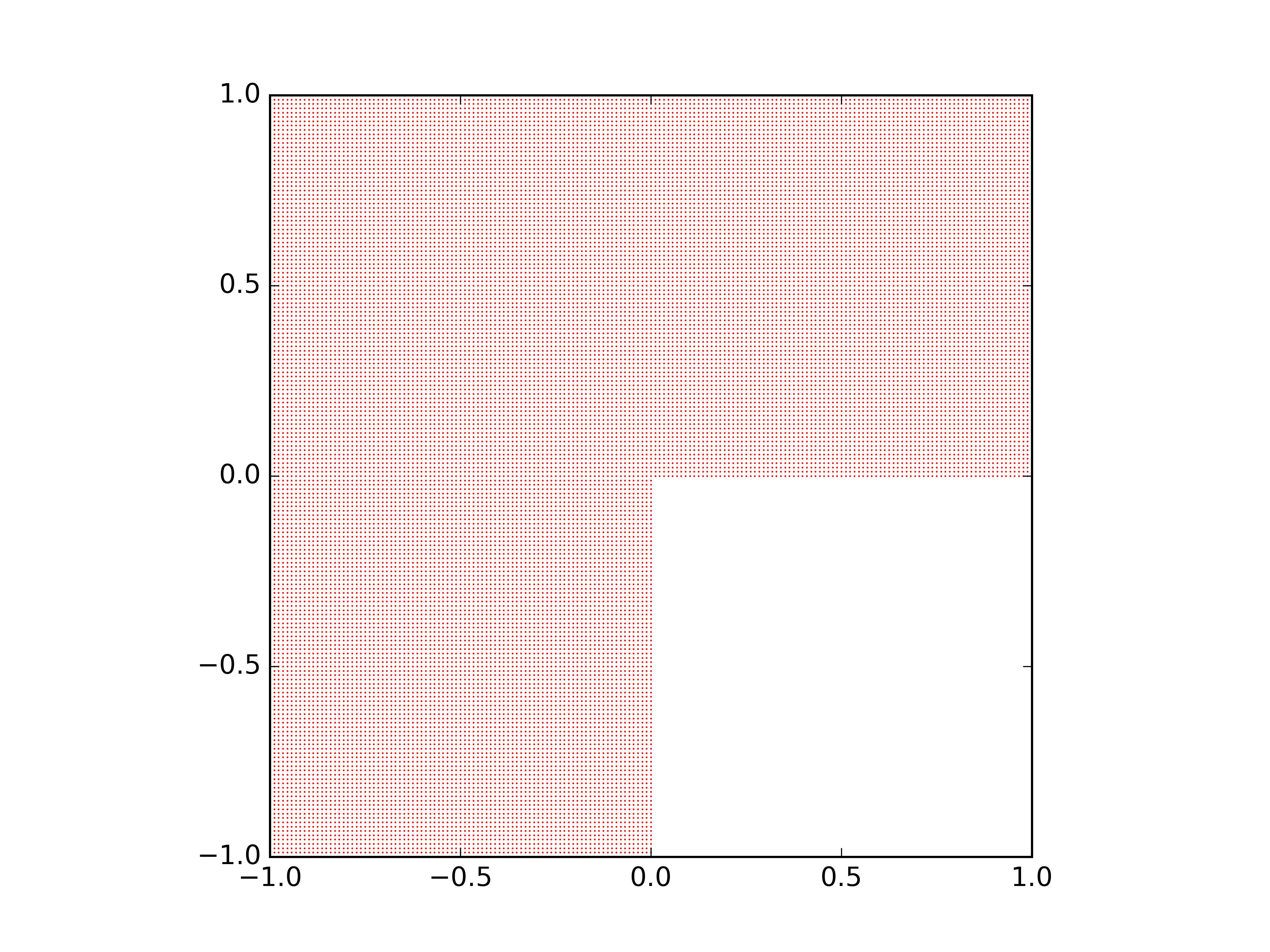}}\hfill
\subfloat[Points distribution, level 2]{\includegraphics[width=0.32\textwidth]{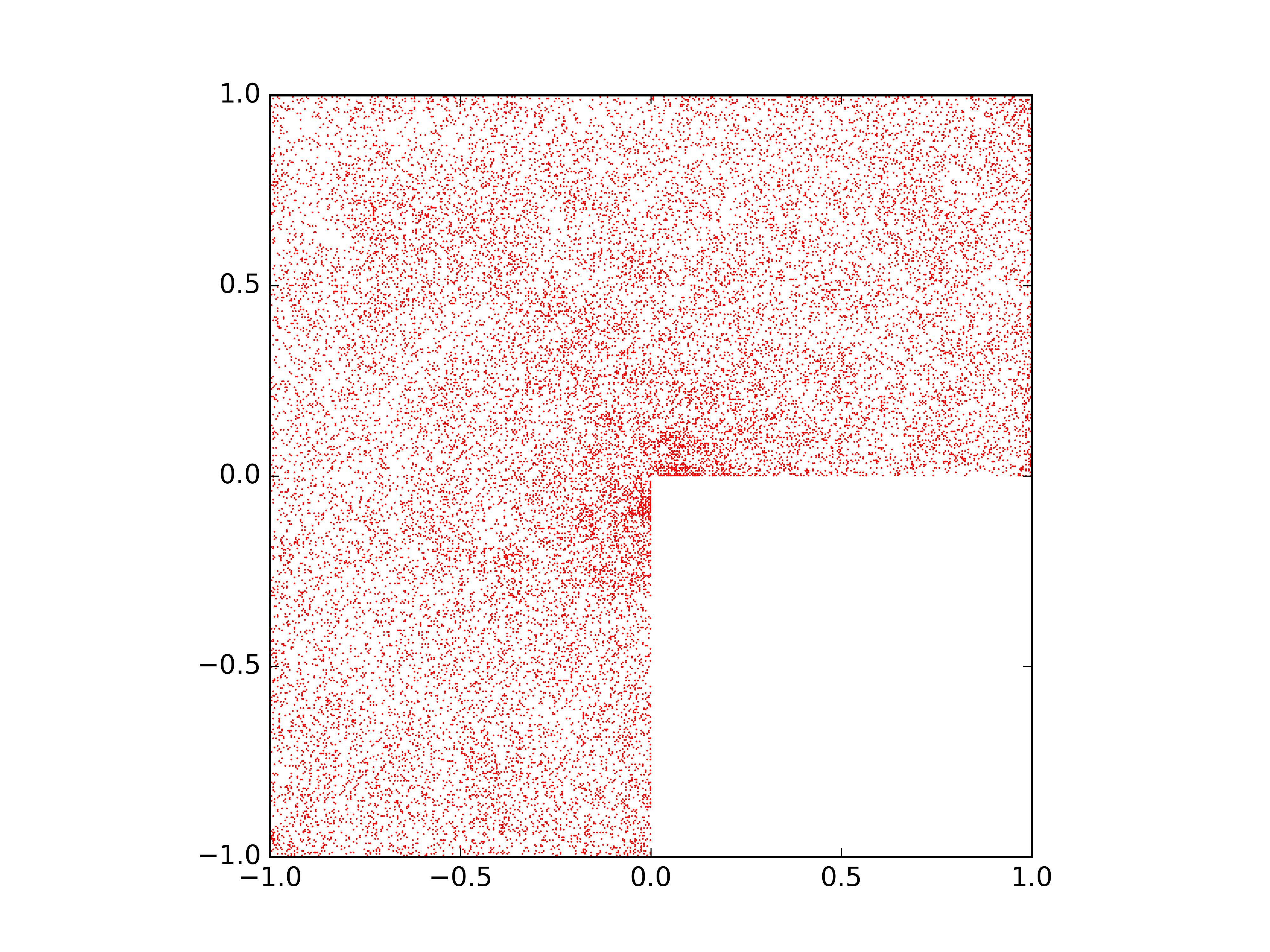}}\hfill
\subfloat[Points distribution, level 3]{\includegraphics[width=0.32\textwidth]{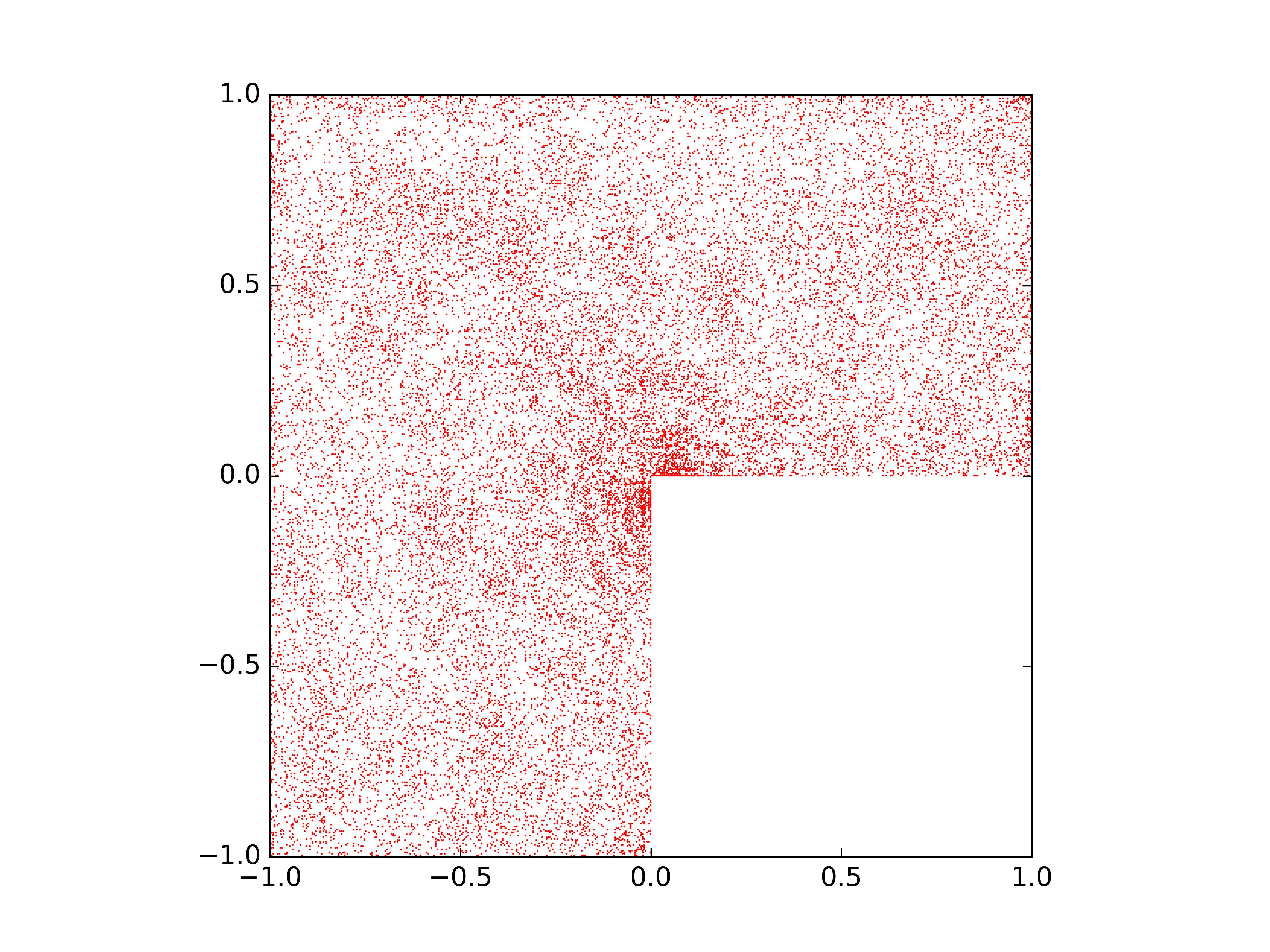}}\\
\subfloat[Error, level 1]{\includegraphics[width=0.32\textwidth]{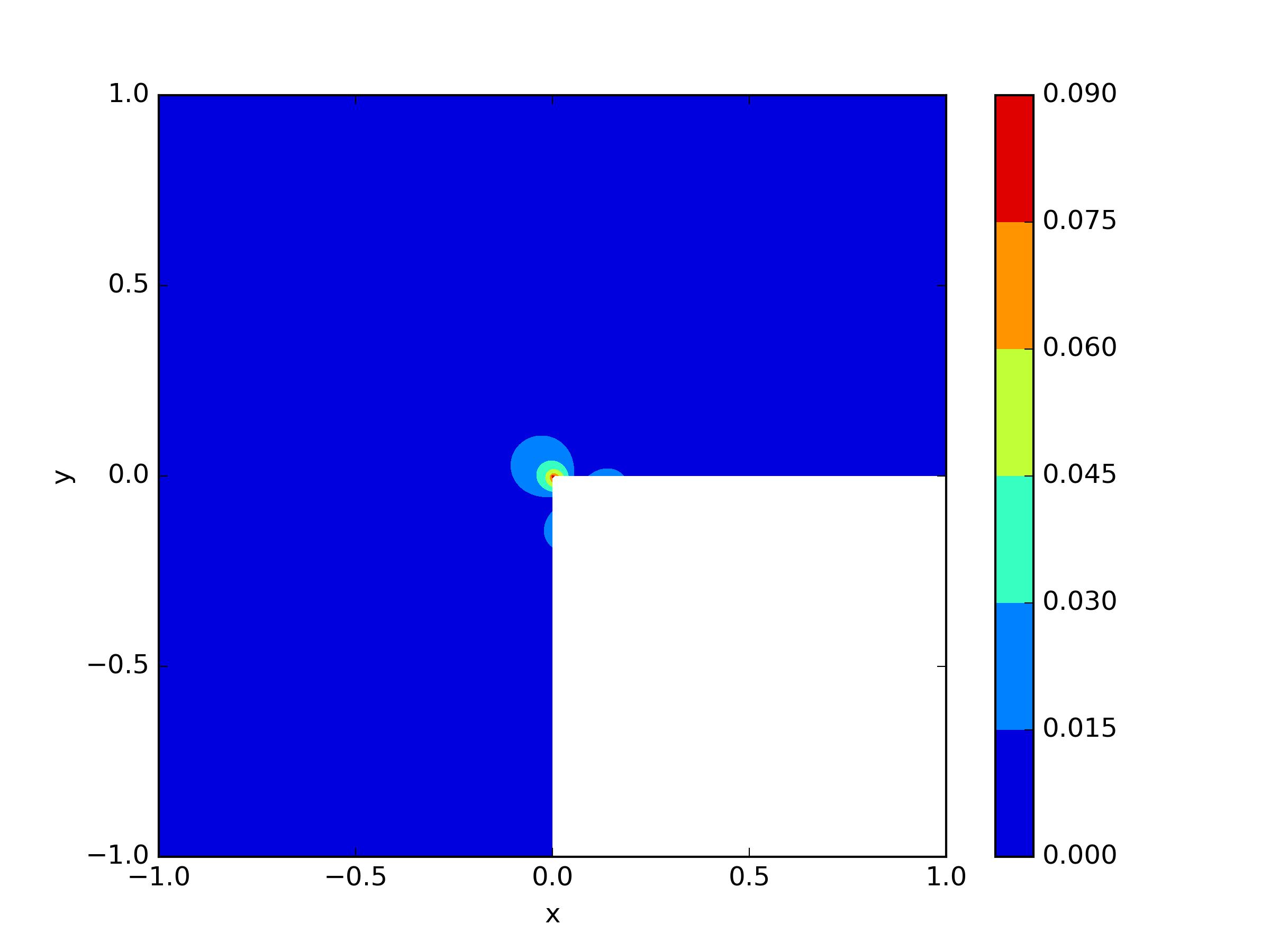}}\hfill
\subfloat[Error, level 2]{\includegraphics[width=0.32\textwidth]{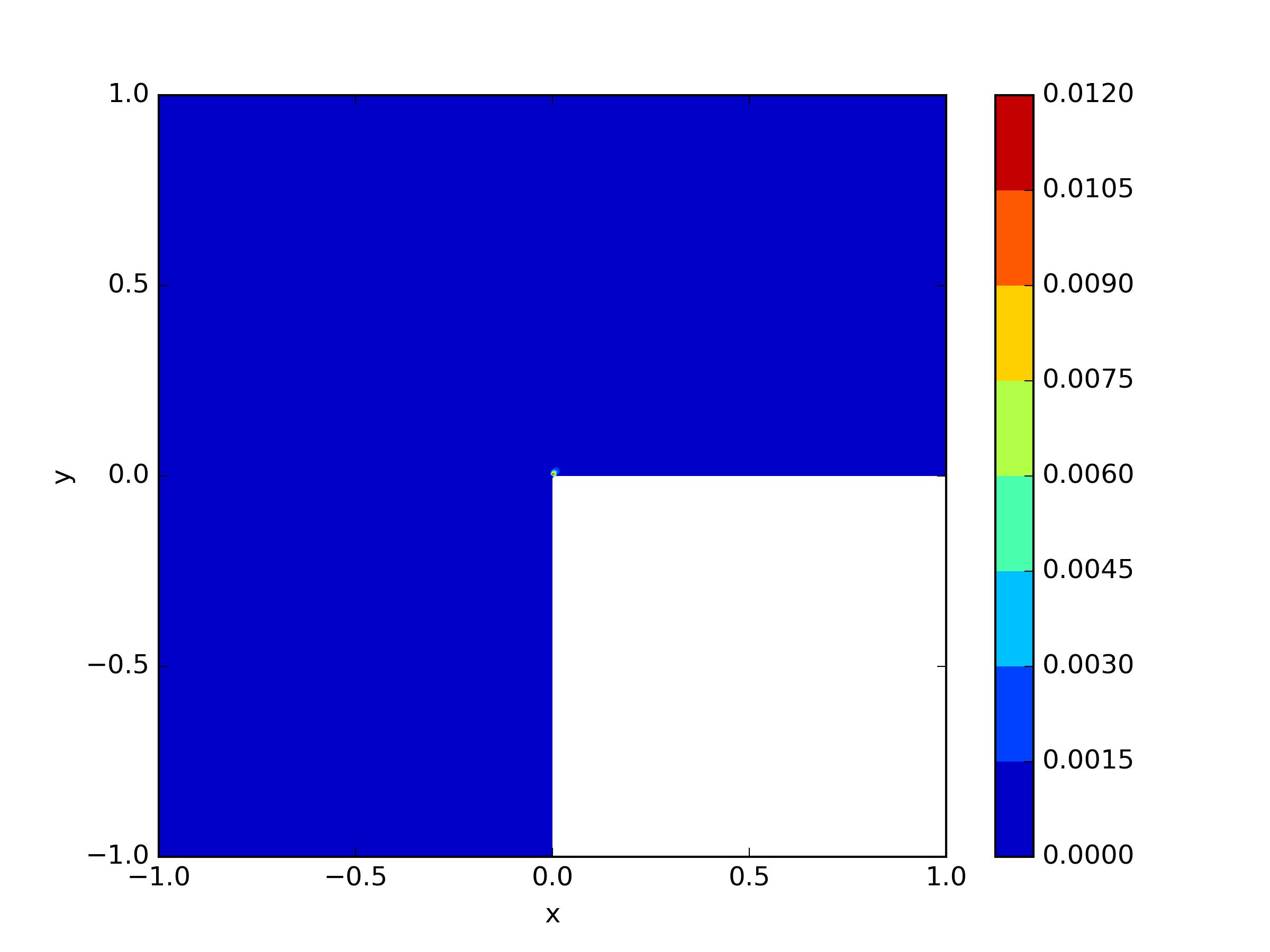}}\hfill
\subfloat[Error, level 3]{\includegraphics[width=0.32\textwidth]{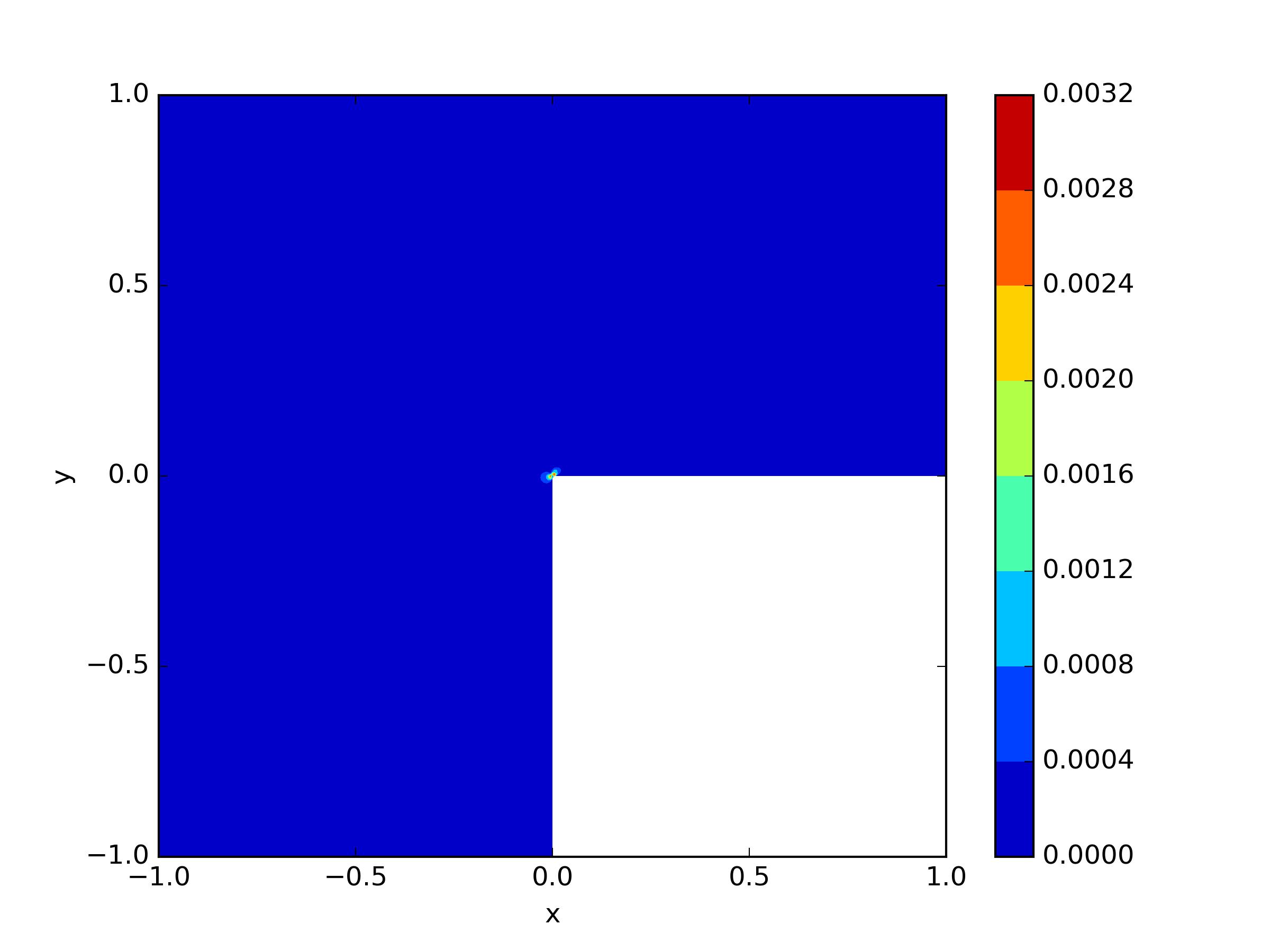}}
\caption{L-shaped Poisson problem at levels 1--3: collocation points (top) and pointwise absolute errors (bottom).} 
\label{fig:2dLshape_points}
\label{fig:2dLshape_results}
\end{figure}

We compare nonadaptive sampling (NaS), residual-based adaptive distribution (RAD), and the proposed MLS strategy, each combined with either PINNs or MLT. To ensure a fair comparison, all methods use the same total training budget of \(50{,}000\) optimization epochs. For NaS+PINNs, the network is trained for \(32{,}500\) SOAP epochs followed by \(17{,}500\) SSB epochs. RAD+PINNs and MLS+PINNs first undergo \(20{,}000\) SOAP pretraining epochs, followed by a single adaptive resampling step and subsequent training with \(12{,}500\) SOAP and \(17{,}500\) SSB epochs. For the MLT variants, the training budget is distributed across three levels: \(20{,}000\) SOAP epochs at level 1, \(10{,}000\) SOAP and \(5{,}000\) SSB epochs at level 2, and \(2{,}500\) SOAP and \(12{,}500\) SSB epochs at level 3. RAD+MLT and MLS+MLT perform adaptive resampling at each level from level 2 onward. As shown in Figure~\ref{fig:2dLshape_ErrorEpochs} and Table~\ref{tab:2dLshape_Results}, MLS+MLT achieves the lowest errors among the compared methods.

\begin{figure}[htbp]
\centering
\subfloat[$e_\infty(u)$]{\includegraphics[width=0.48\textwidth]{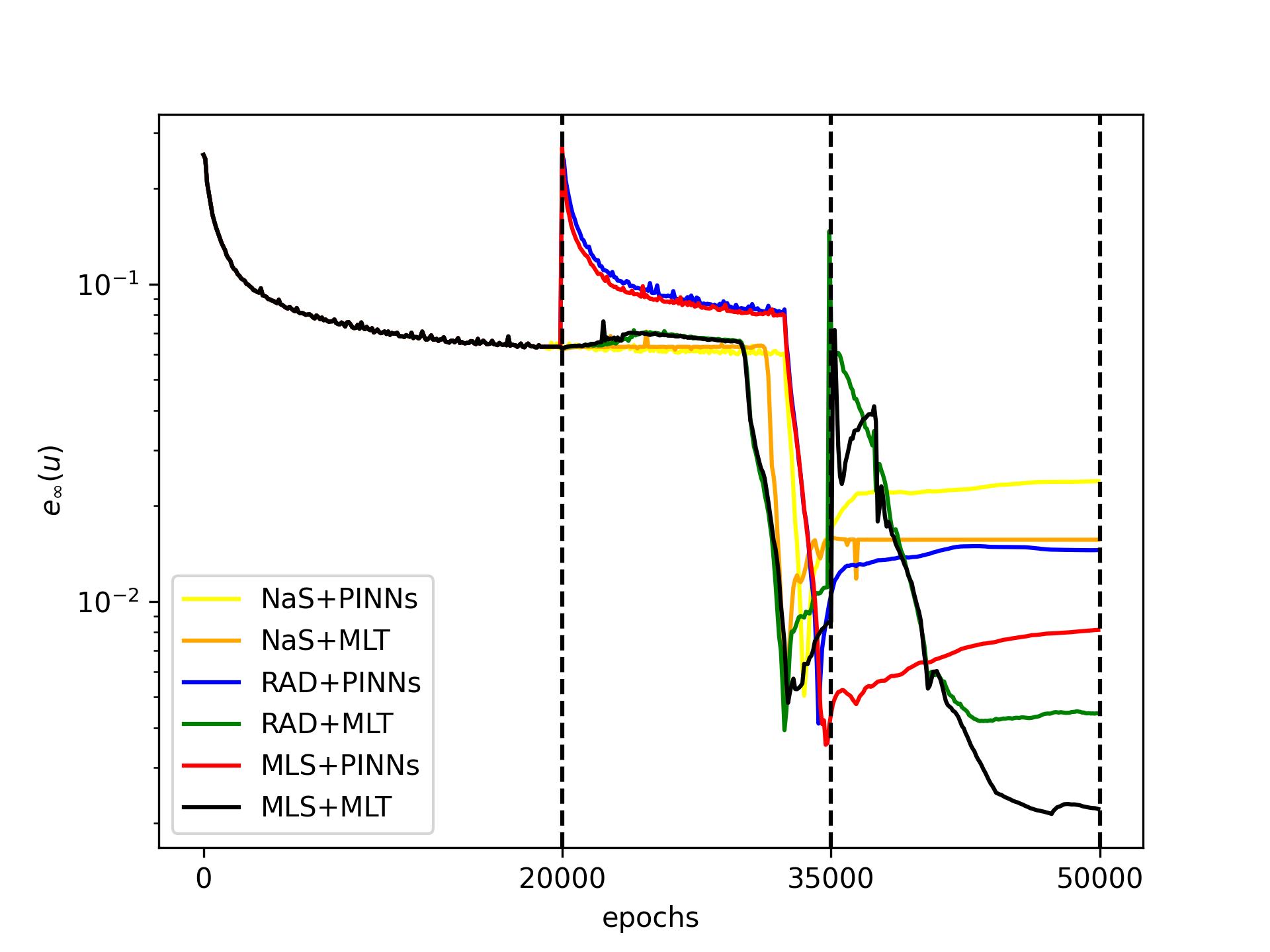}}\hfill
\subfloat[$e_2(u)$]{\includegraphics[width=0.48\textwidth]{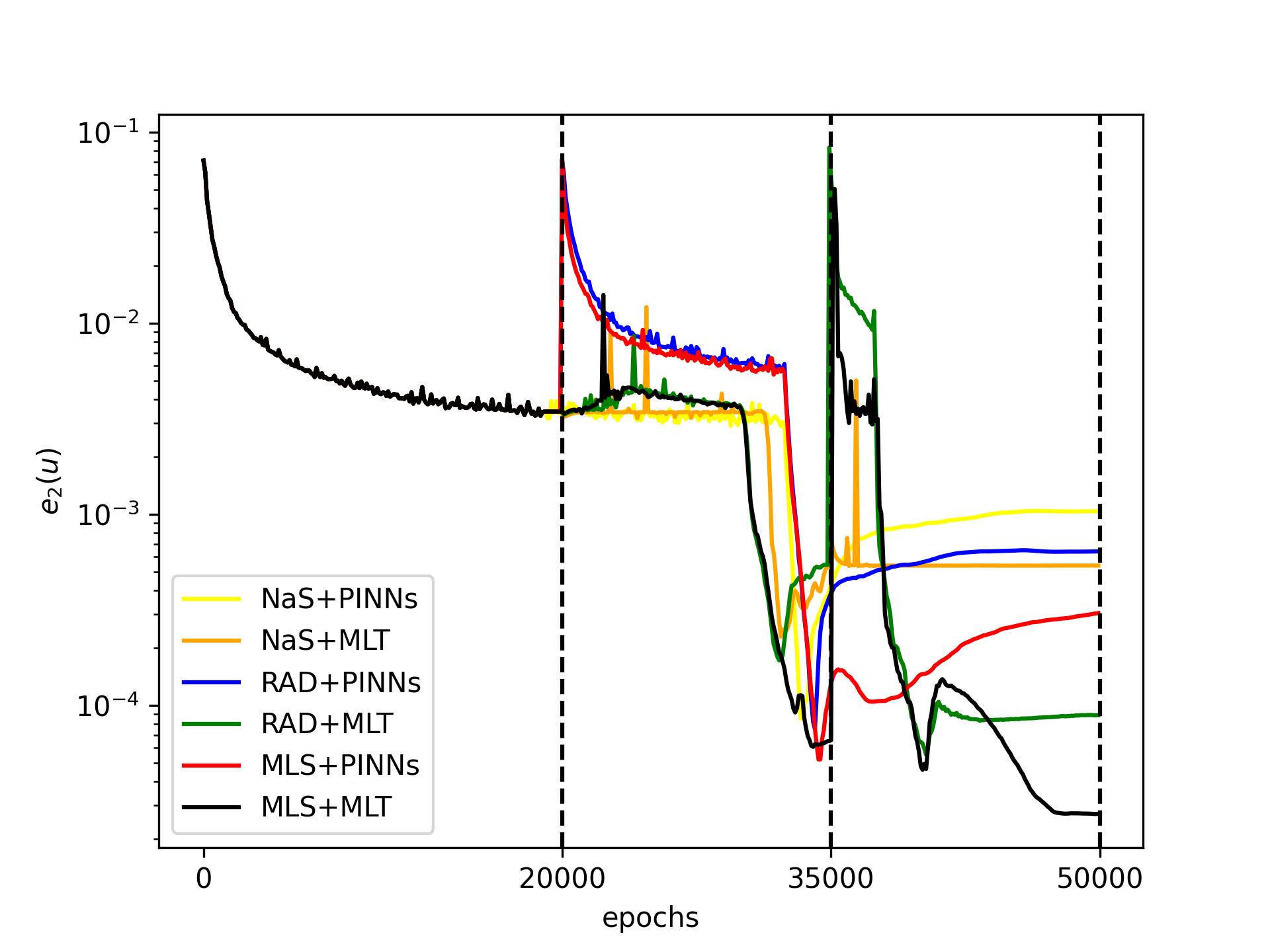}}
\caption{L-shaped Poisson problem: relative errors versus training epochs.}
\label{fig:2dLshape_ErrorEpochs}
\end{figure}

\begin{table}[htbp]
\scriptsize
\centering
\caption{Relative errors for the L-shaped Poisson problem.}
\setlength{\tabcolsep}{1.mm}{
\begin{tabular}{|c|c|c|c|c|c|c|}
\hline
Relative  error     &  NaS+PINNs & NaS+MLT & RAD+PINNs  &  RAD+MLT& MLS+PINNs & MLS+MLT \\
\hline
$e_\infty(u)$  & $ 2.396\times 10^{-2}$  & $1.565 \times 10^{-2}$ & $1.452 \times 10^{-2}$ &  $4.440 \times 10^{-3}$&$8.136 \times 10^{-3}$&$2.158 \times 10^{-3}$\\

\hline
$e_2(u)$  & $1.039 \times 10^{-3}$  & $5.398 \times 10^{-4}$ & $6.403 \times 10^{-4}$  & $8.890 \times 10^{-5}$&$3.045\times 10^{-4}$&$2.798 \times 10^{-5}$\\

\hline
\end{tabular}
}
\label{tab:2dLshape_Results}
\end{table}

\subsection{Poisson equation with a sharp interior layer}
\label{sec:PoissonSharp_2D}
We consider the Poisson problem \eqref{eq:2d_Lshape}
on $\Omega=(0,1)^d$, with exact solution \cite{dang2024adaptive}
\begin{equation}
    \label{eq:PoissonSharpsolution}
    \hspace{-0.3cm}
    \begin{array}{r@{}l}
        \begin{aligned}
            u = 4^d\left(\frac{1}{2}+ \frac{1}{\pi} \arctan \left( A(\frac{1}{16} - \sum_{i=1}^d(x_i-\frac{1}{2}^2))\right)\right)\prod_{i=1}^d(1-x_i)x_i,
        \end{aligned}
    \end{array}
\end{equation}
where $A=120$, and $f$ and $g$ are determined from \eqref{eq:PoissonSharpsolution}. The large value of $A$ produces a
sharp transition near
\[
\sum_{i=1}^d\left(x_i-\frac12\right)^2=\frac{1}{16},
\]
providing a test of the ability of MLS to resolve localized high-gradient
features.

\subsubsection{Two-dimensional problem}
For $d=2$, the sharp transition occurs near the circle centered at
$(1/2,1/2)$ with radius $1/4$. As shown in
Figure~\ref{fig:PoissonSharp2D_points}, MLS progressively concentrates
collocation points in this region. After three levels, the errors are
$e_\infty(\bu)=4.400\times10^{-7}$ and $e_2(\bu)=8.197\times10^{-8}.$
The latter is lower than the best reported value
$e_2(\bu)=8.769\times10^{-7}$ in \cite{dang2024adaptive}.

\begin{figure}[htbp]
\centering
\subfloat[Points distribution, level 1]{\includegraphics[width=0.32\textwidth]{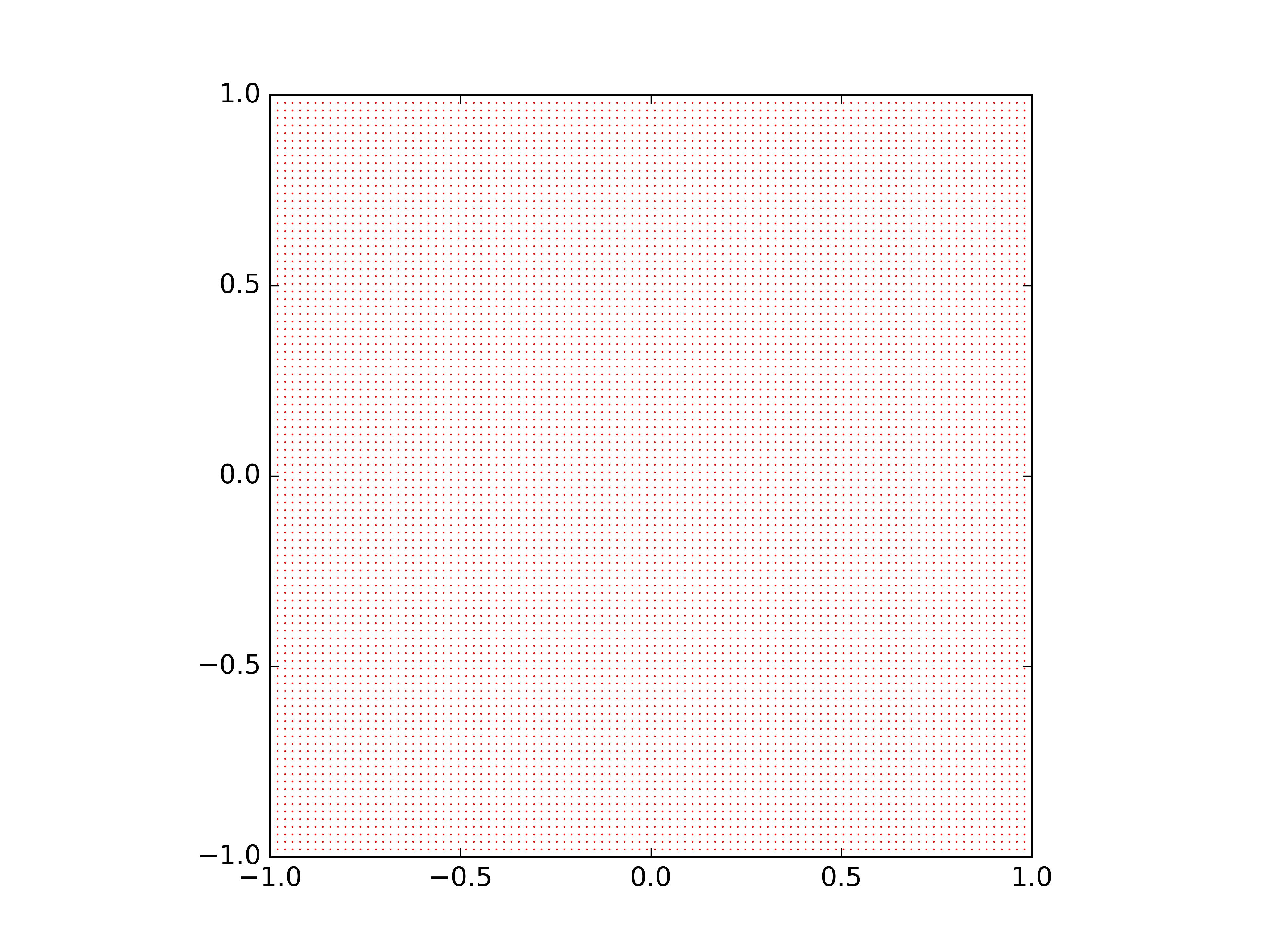}}\hfill
\subfloat[Points distribution, level 2]{\includegraphics[width=0.32\textwidth]{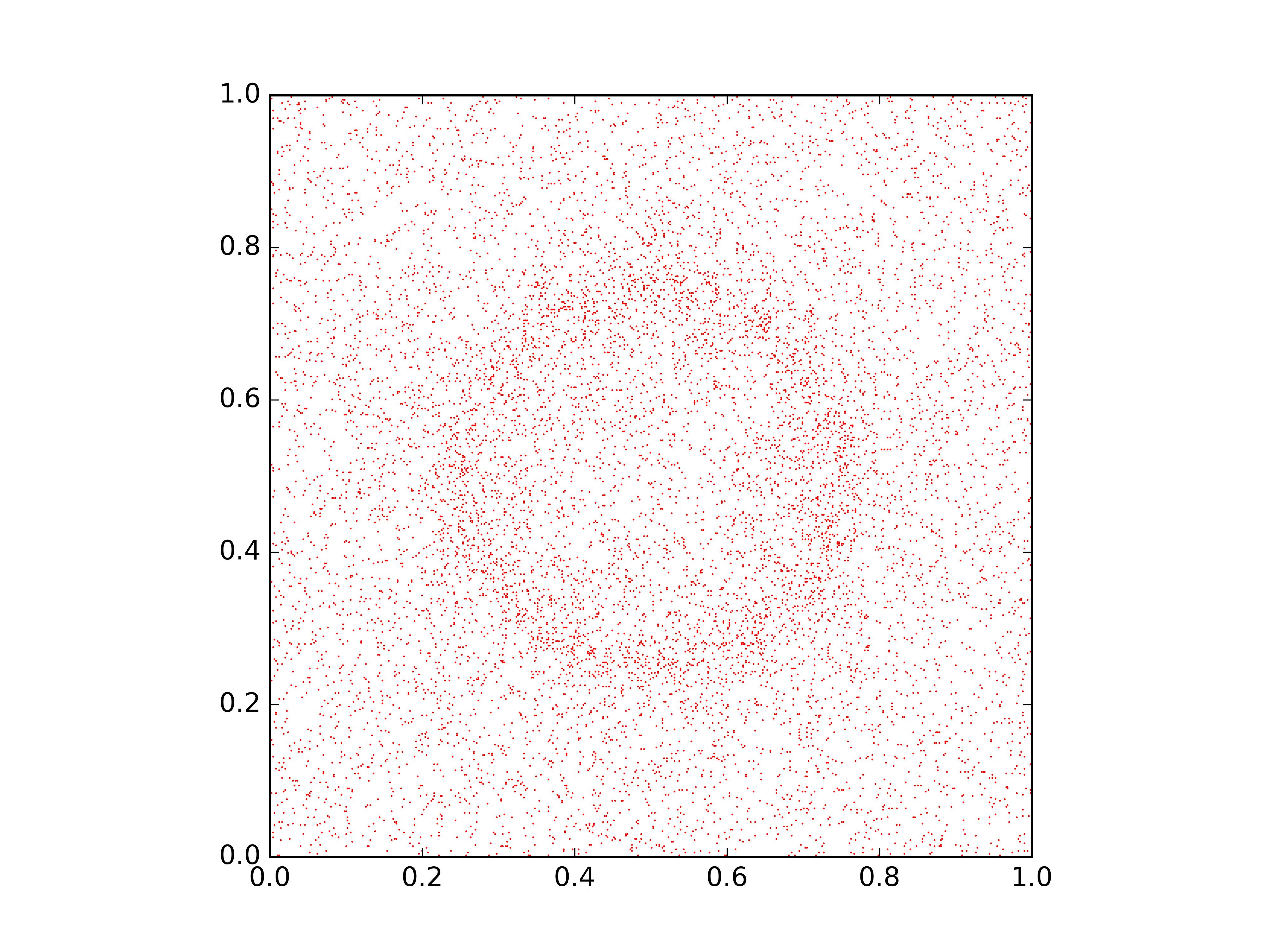}}\hfill
\subfloat[Points distribution, level 3]{\includegraphics[width=0.32\textwidth]{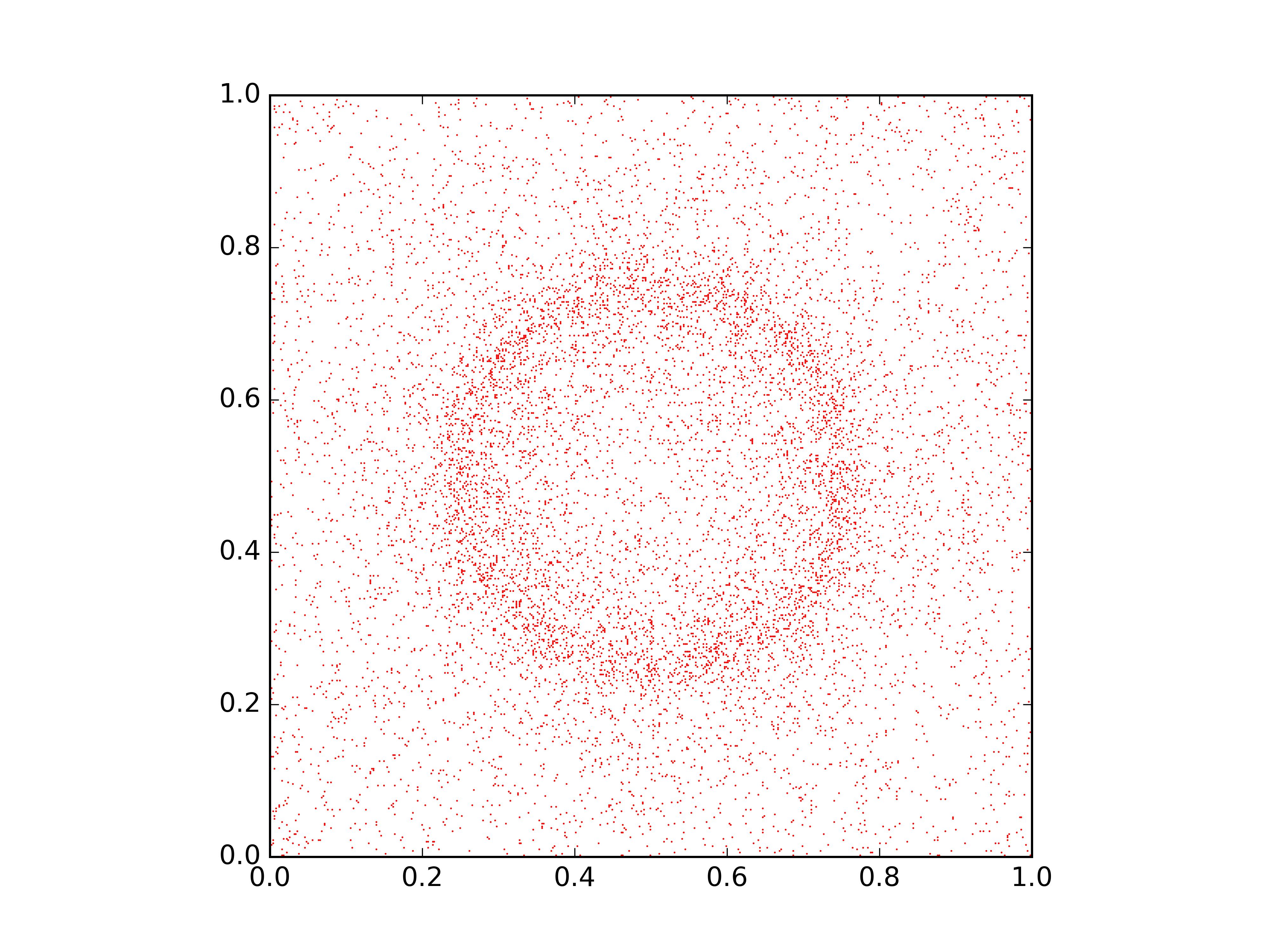}}\\
\subfloat[Numerical solution, level 1]{\includegraphics[width=0.32\textwidth]{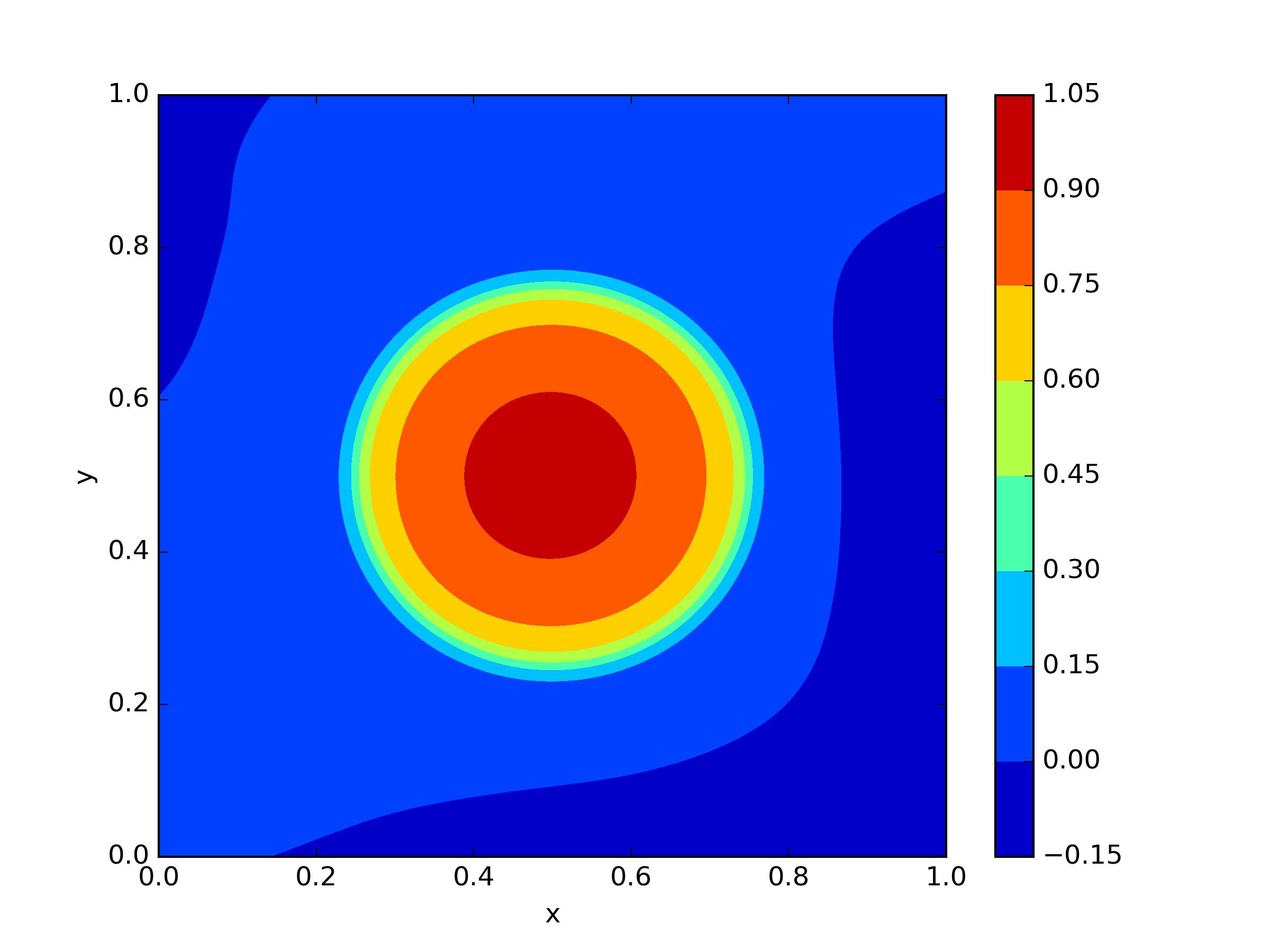}}\hfill
\subfloat[Numerical solution, level 2]{\includegraphics[width=0.32\textwidth]{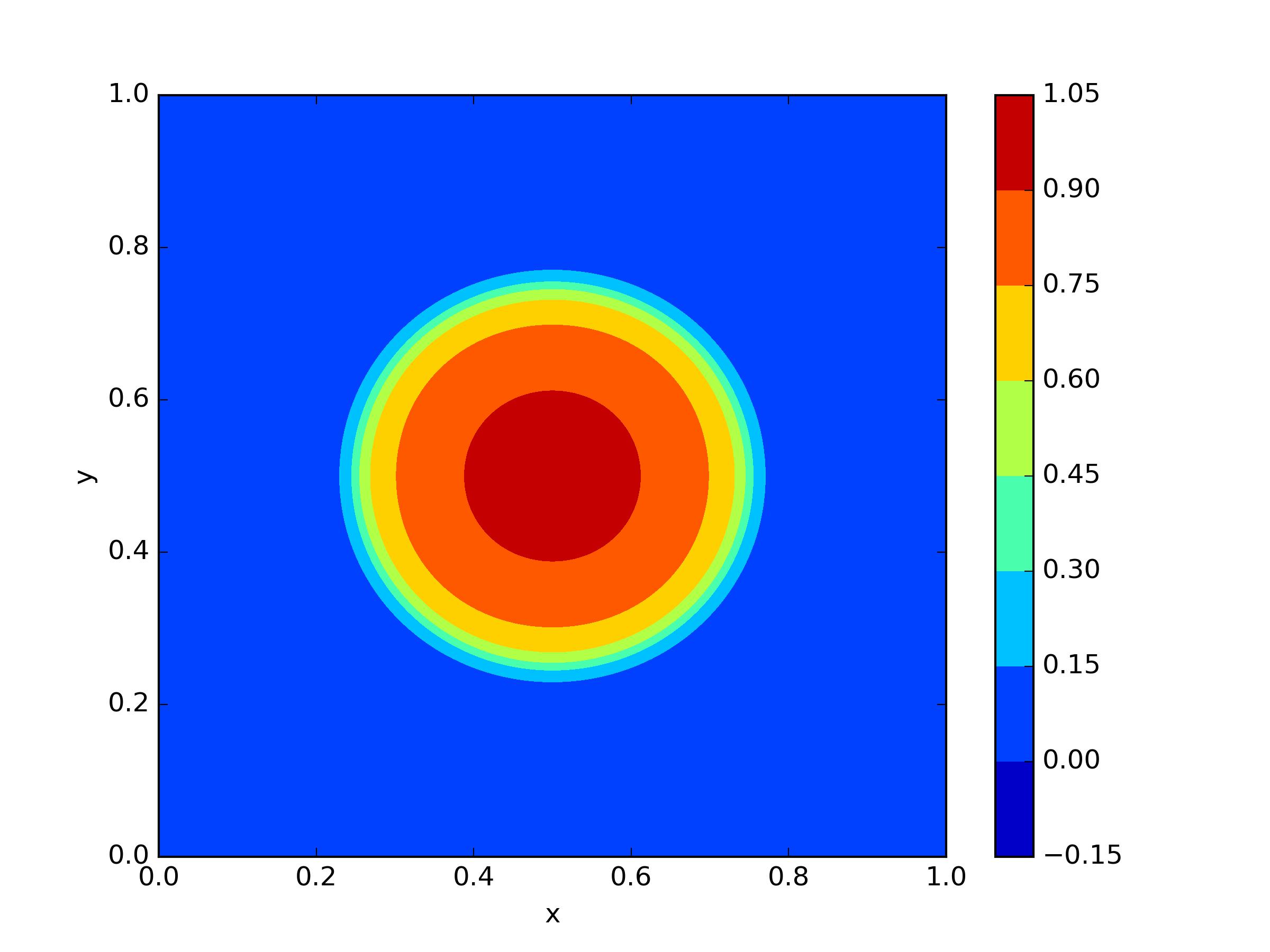}}\hfill
\subfloat[Numerical solution, level 3]{\includegraphics[width=0.32\textwidth]{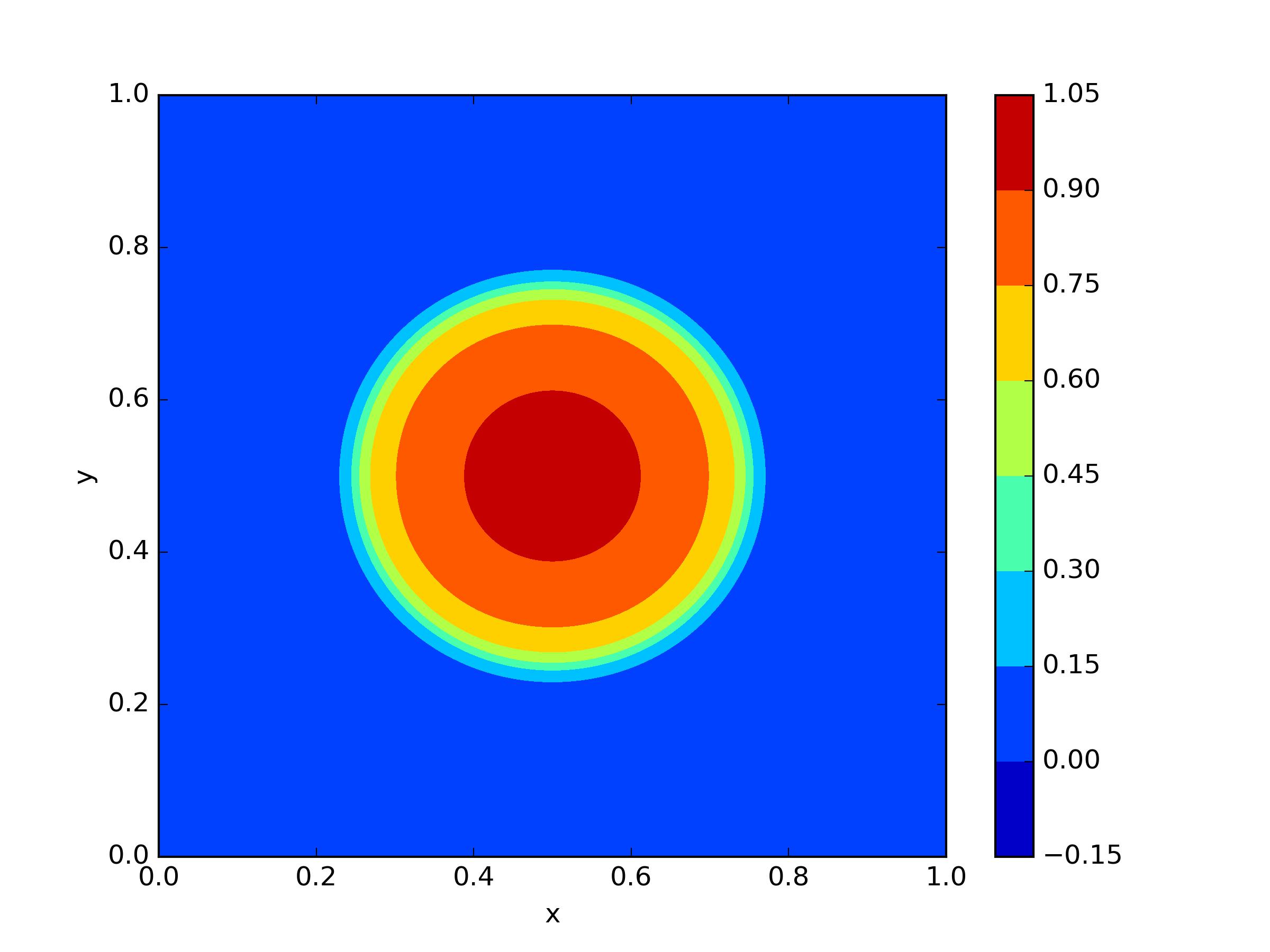}}\\
\subfloat[Error, level 1]{\includegraphics[width=0.32\textwidth]{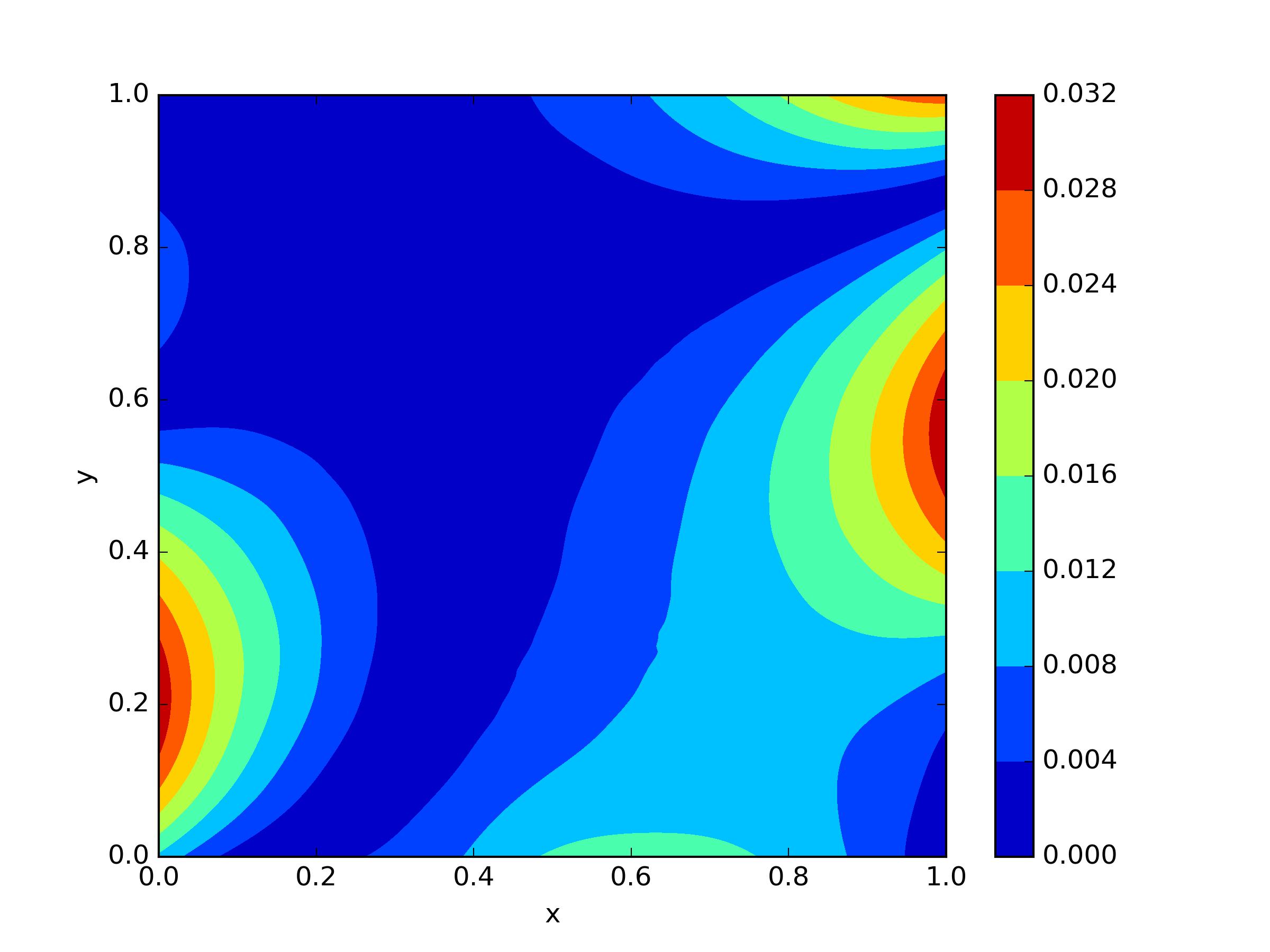}}\hfill
\subfloat[Error, level 2]{\includegraphics[width=0.32\textwidth]{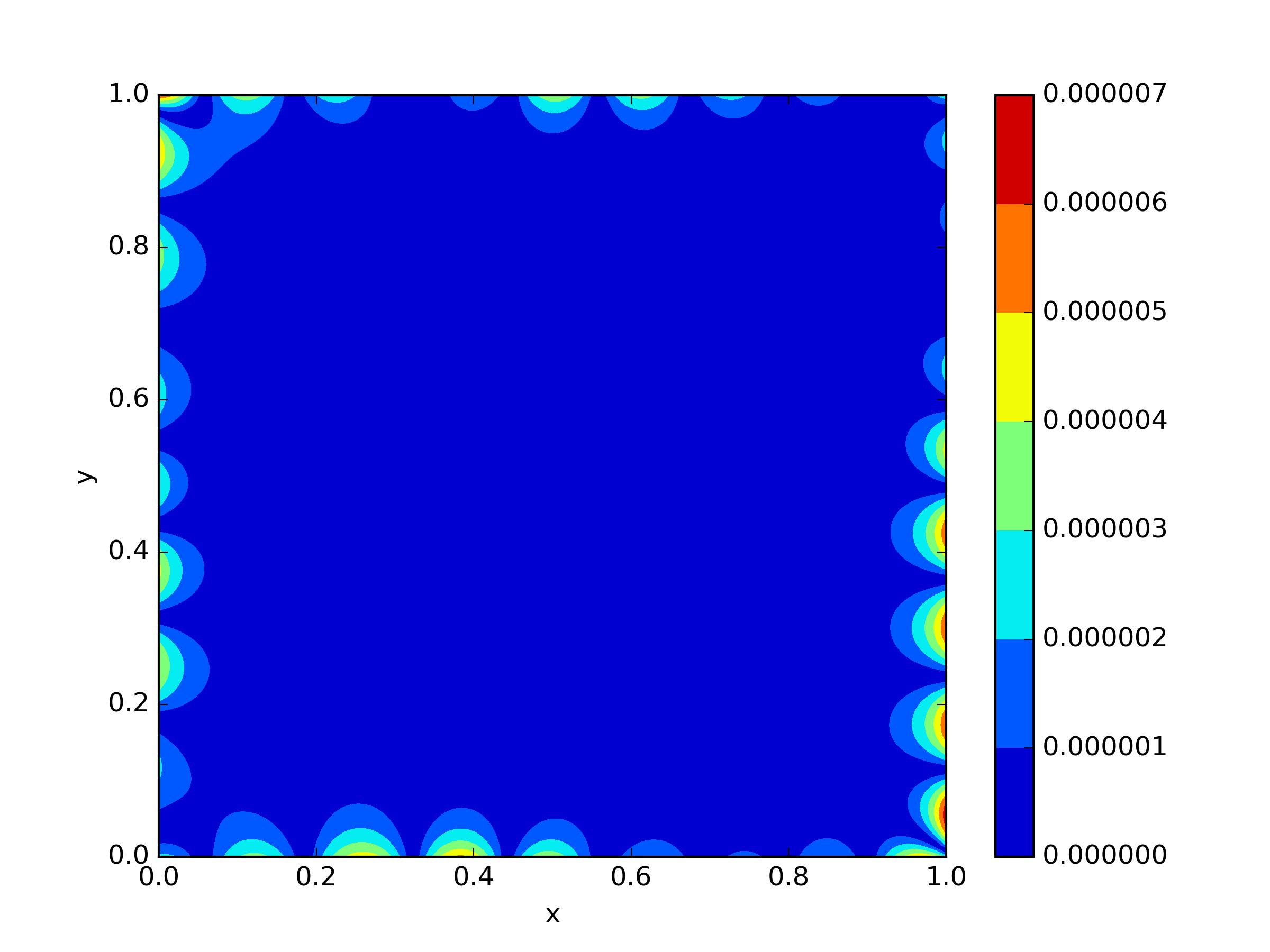}}\hfill
\subfloat[Error, level 3]{\includegraphics[width=0.32\textwidth]{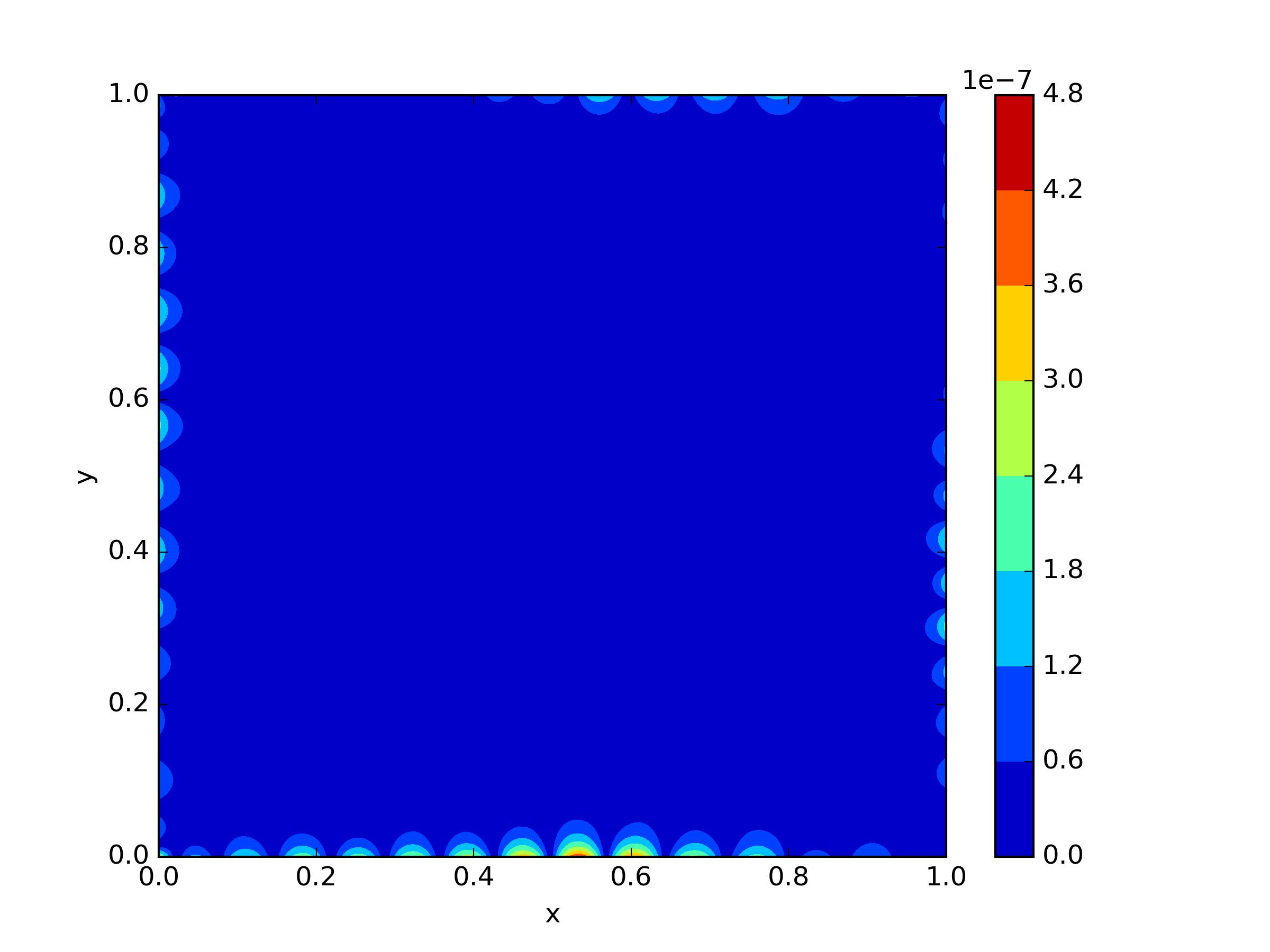}}
\caption{
Two-dimensional Poisson problem with a sharp interior layer at levels 1--3: collocation points (top), numerical solutions (middle), and pointwise absolute errors (bottom).}
\label{fig:PoissonSharp2D_points}
\label{fig:PoissonSharp2D_results}
\end{figure}

\subsubsection{Three-dimensional problem}
For $d=3$, the high-gradient region is localized near the sphere centered
at $(1/2,1/2,1/2)$ with radius $1/4$. The $xy$ projection in
Figure~\ref{fig:PoissonSharp3D_points} illustrates the corresponding
adaptive concentration of collocation points. After three levels, we obtain
$e_\infty(\bu)=1.746\times10^{-7}$  and $e_2(\bu)=7.931\times10^{-8}$, 
compared with the best reported value
$e_2(\bu)=8.330\times10^{-4}$ in \cite{dang2024adaptive}.

\begin{figure}[htbp]
\centering
\subfloat[Points distribution, level 1]{\includegraphics[width=0.32\textwidth]{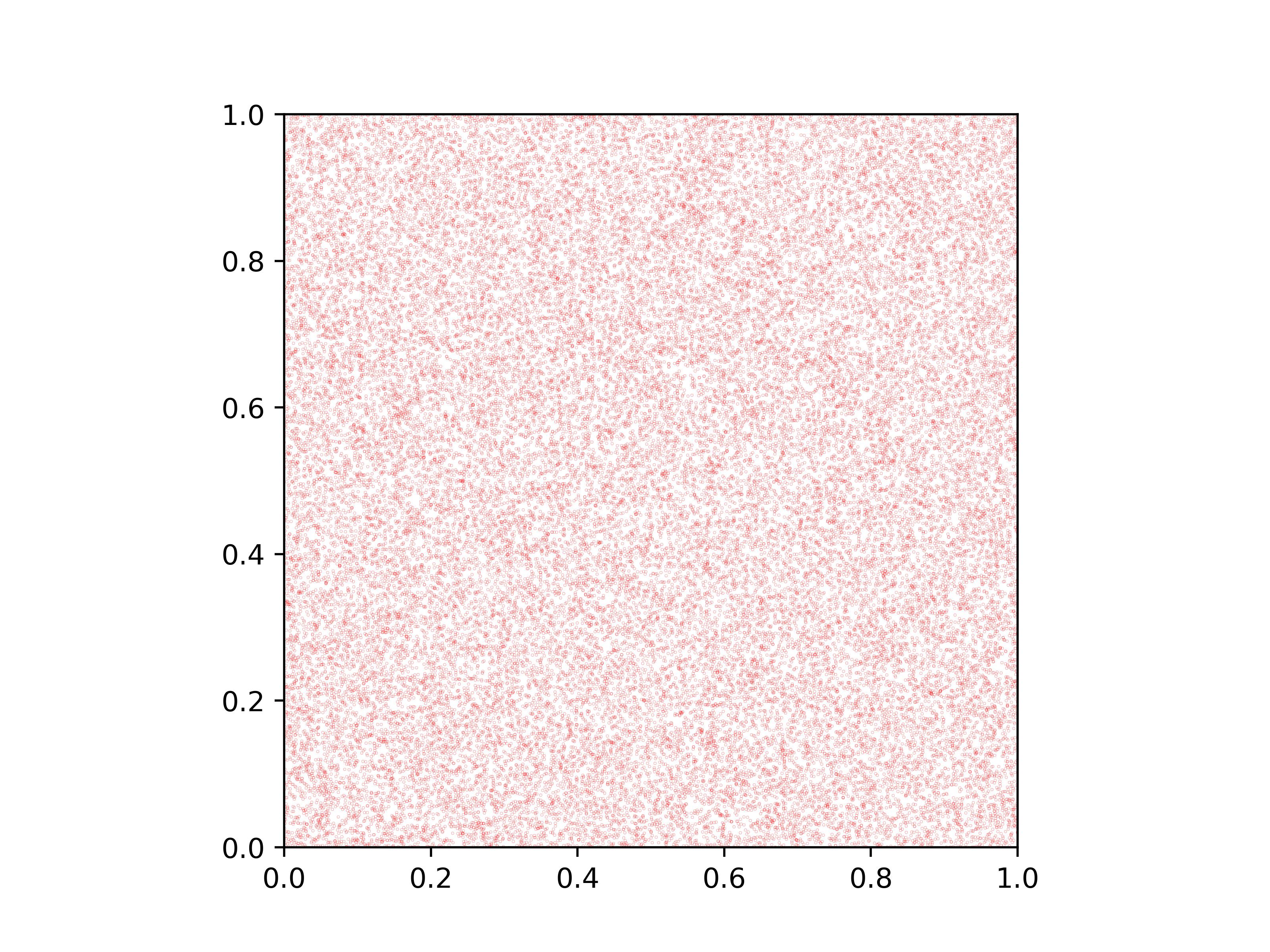}}\hfill
\subfloat[Points distribution, level 2]{\includegraphics[width=0.32\textwidth]{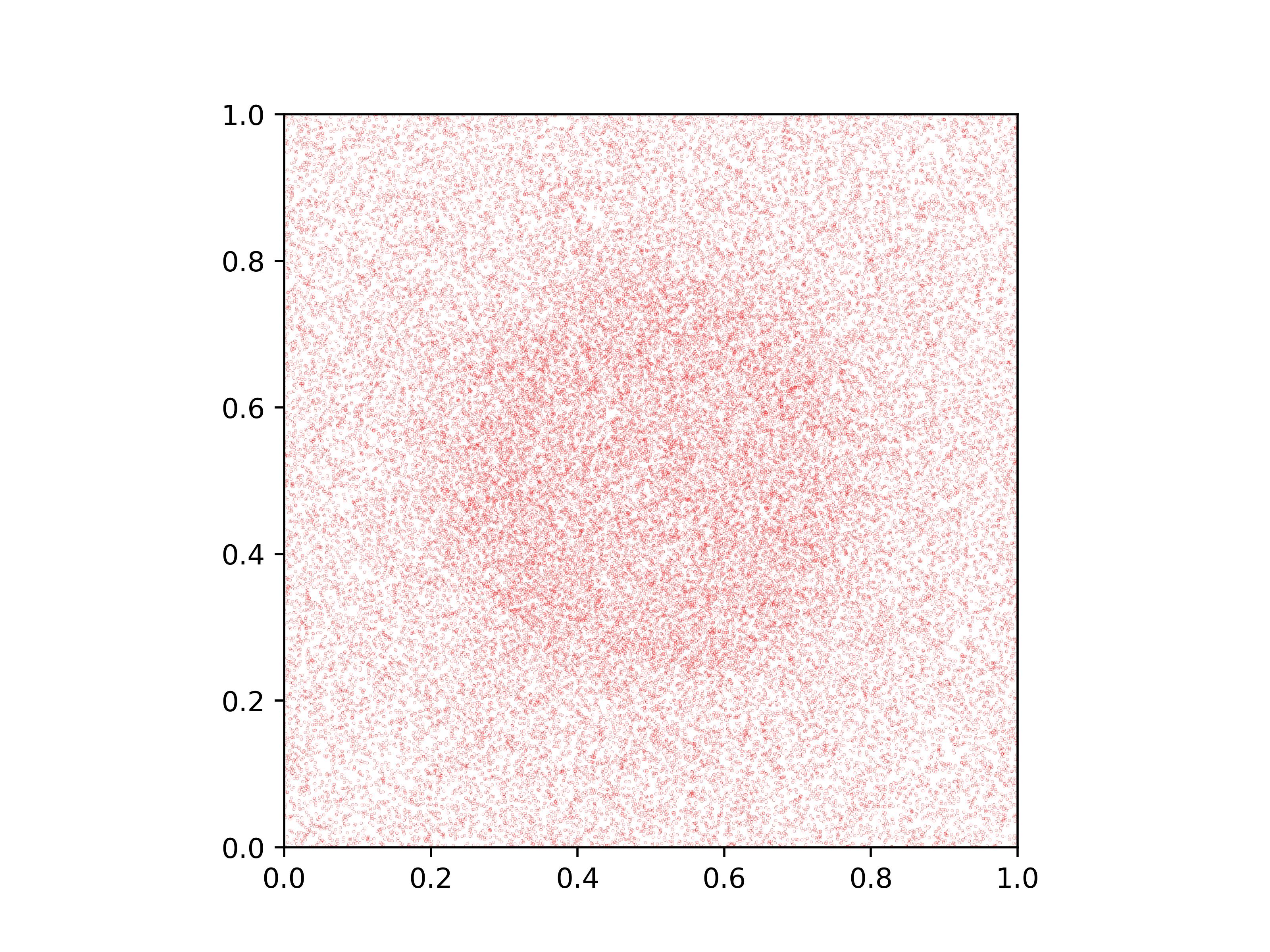}}\hfill
\subfloat[Points distribution, level 3]{\includegraphics[width=0.32\textwidth]{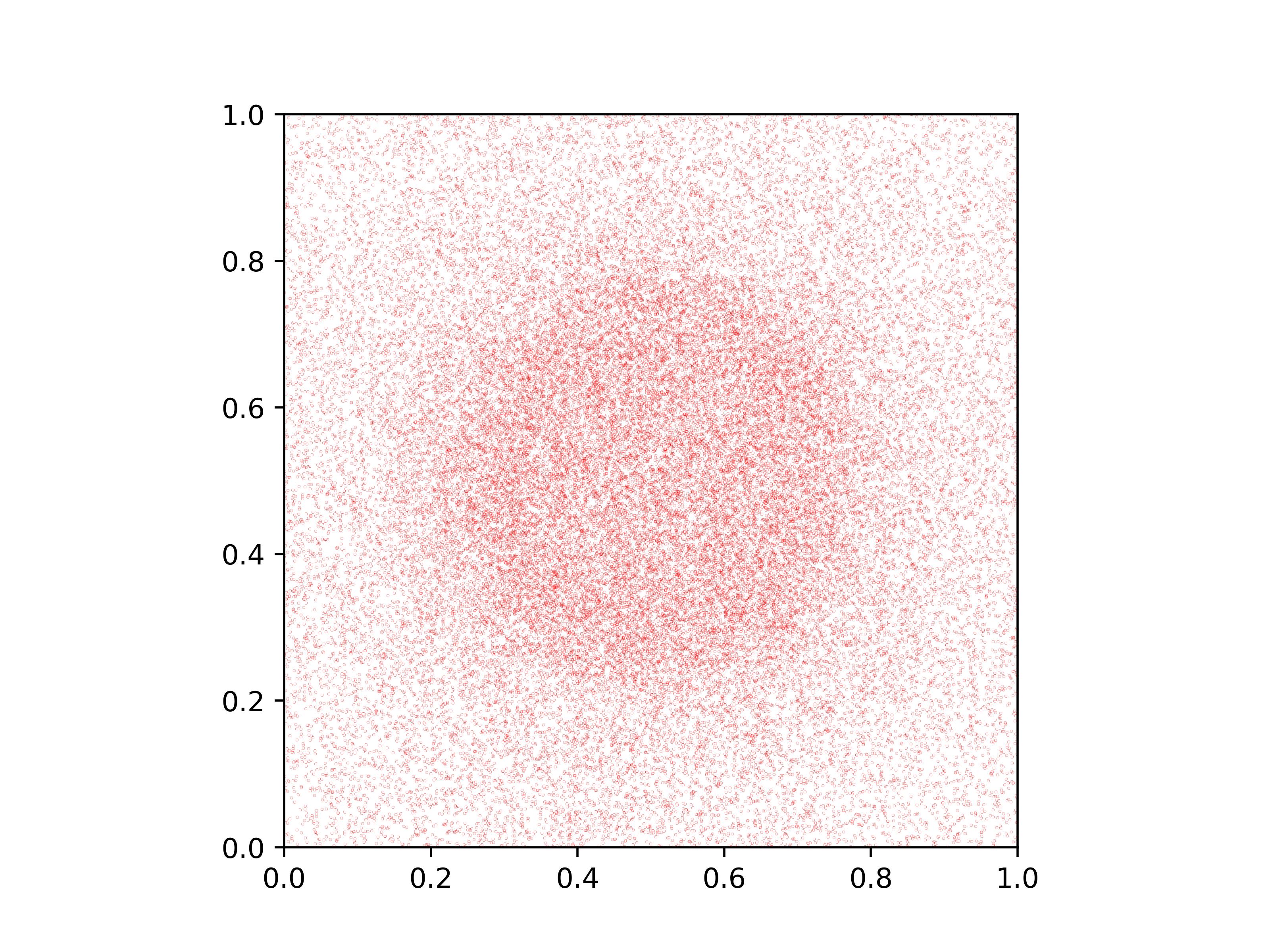}}\\
\subfloat[Numerical solution, level 1]{\includegraphics[width=0.32\textwidth]{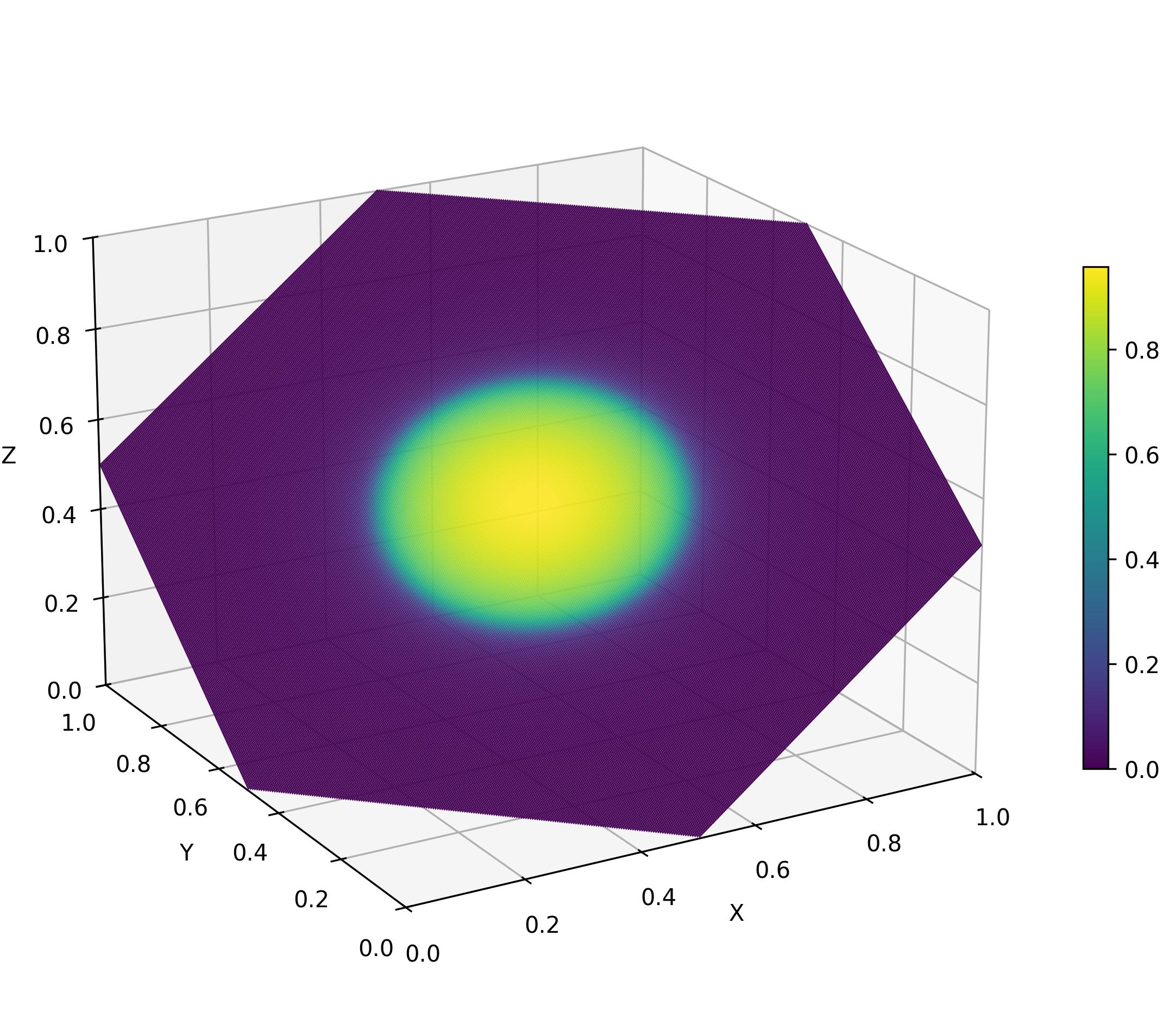}}\hfill
\subfloat[Numerical solution, level 2]{\includegraphics[width=0.32\textwidth]{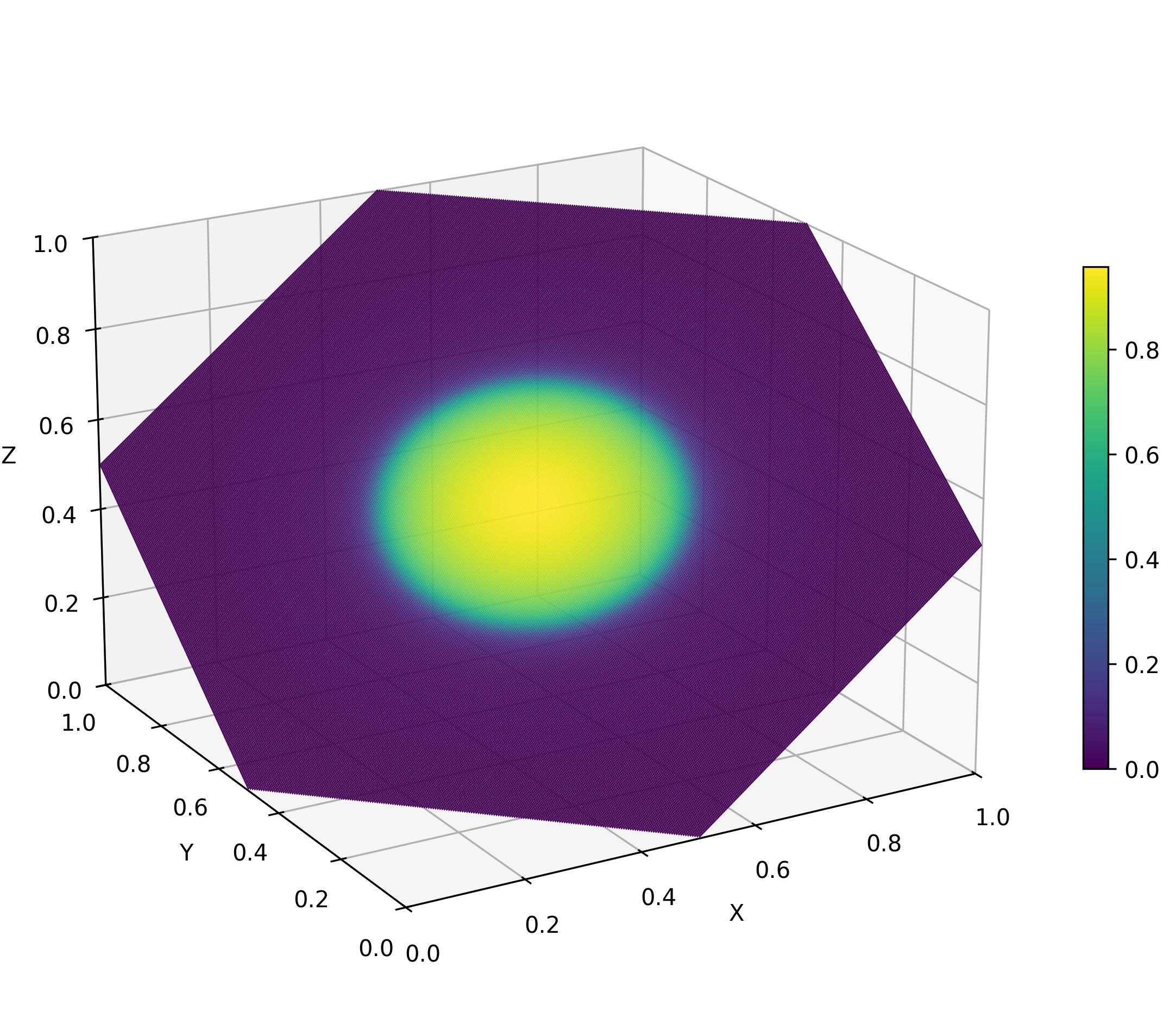}}\hfill
\subfloat[Numerical solution, level 3]{\includegraphics[width=0.32\textwidth]{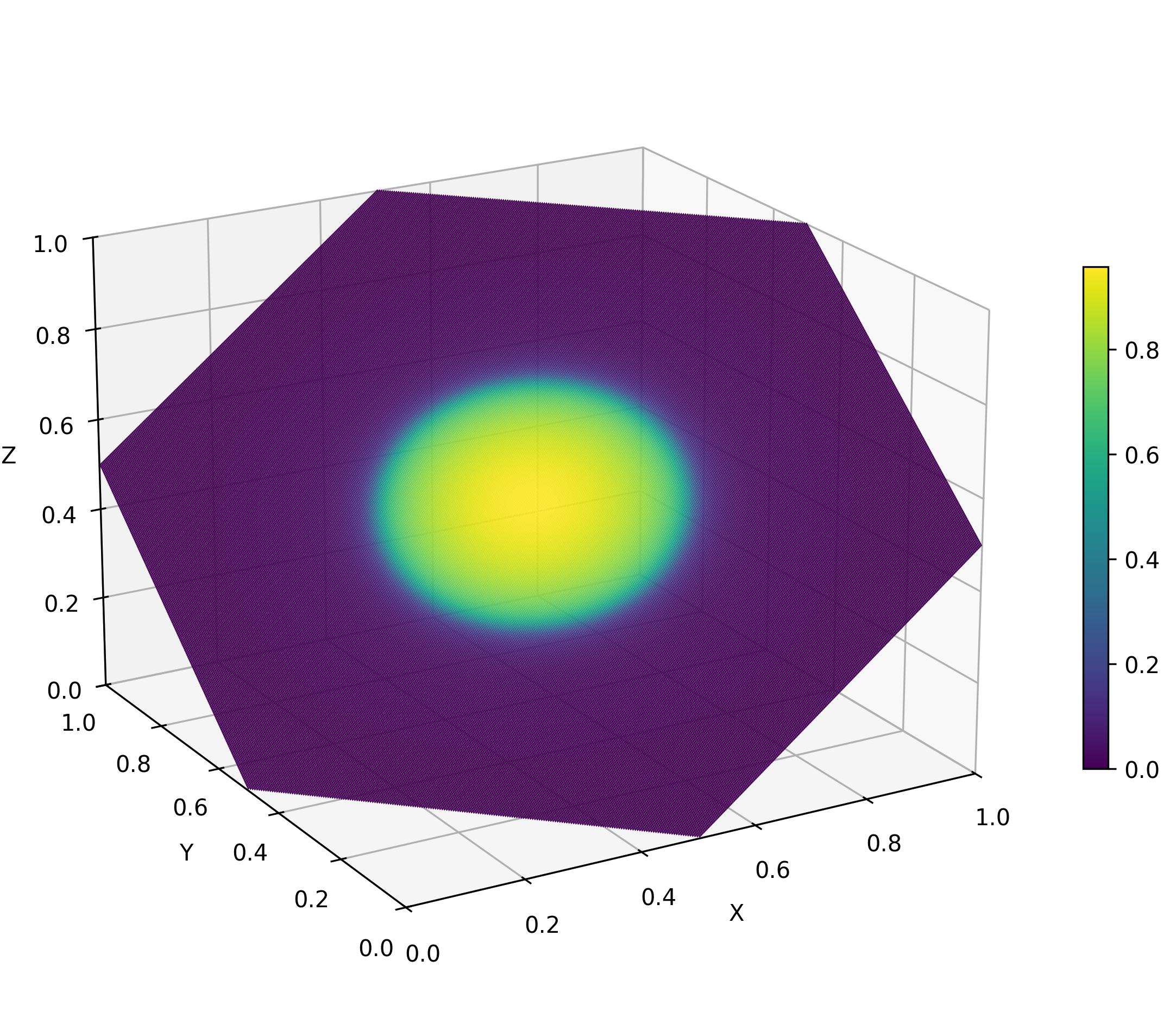}}\\
\subfloat[Error, level 1]{\includegraphics[width=0.32\textwidth]{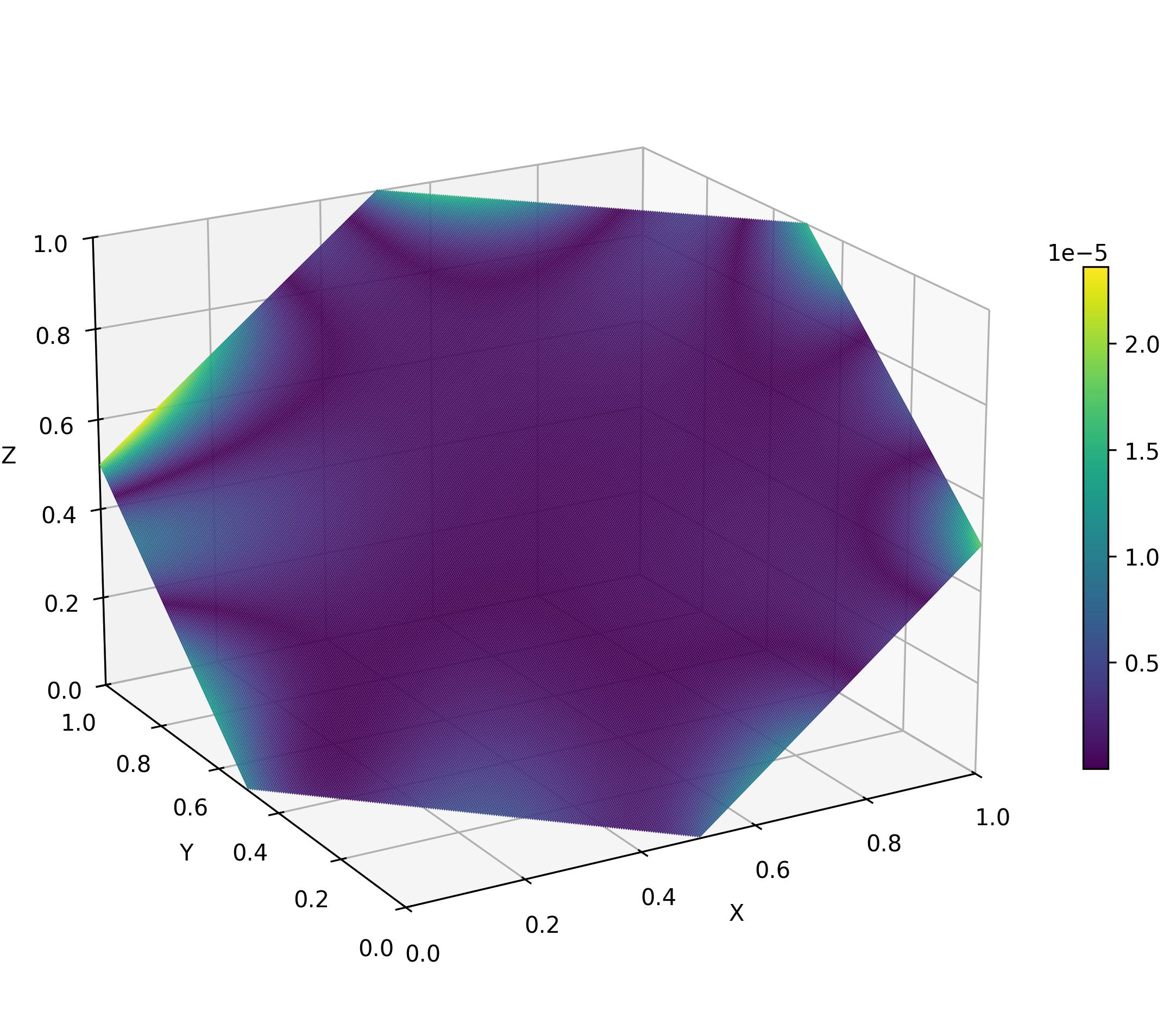}}\hfill
\subfloat[Error, level 2]{\includegraphics[width=0.32\textwidth]{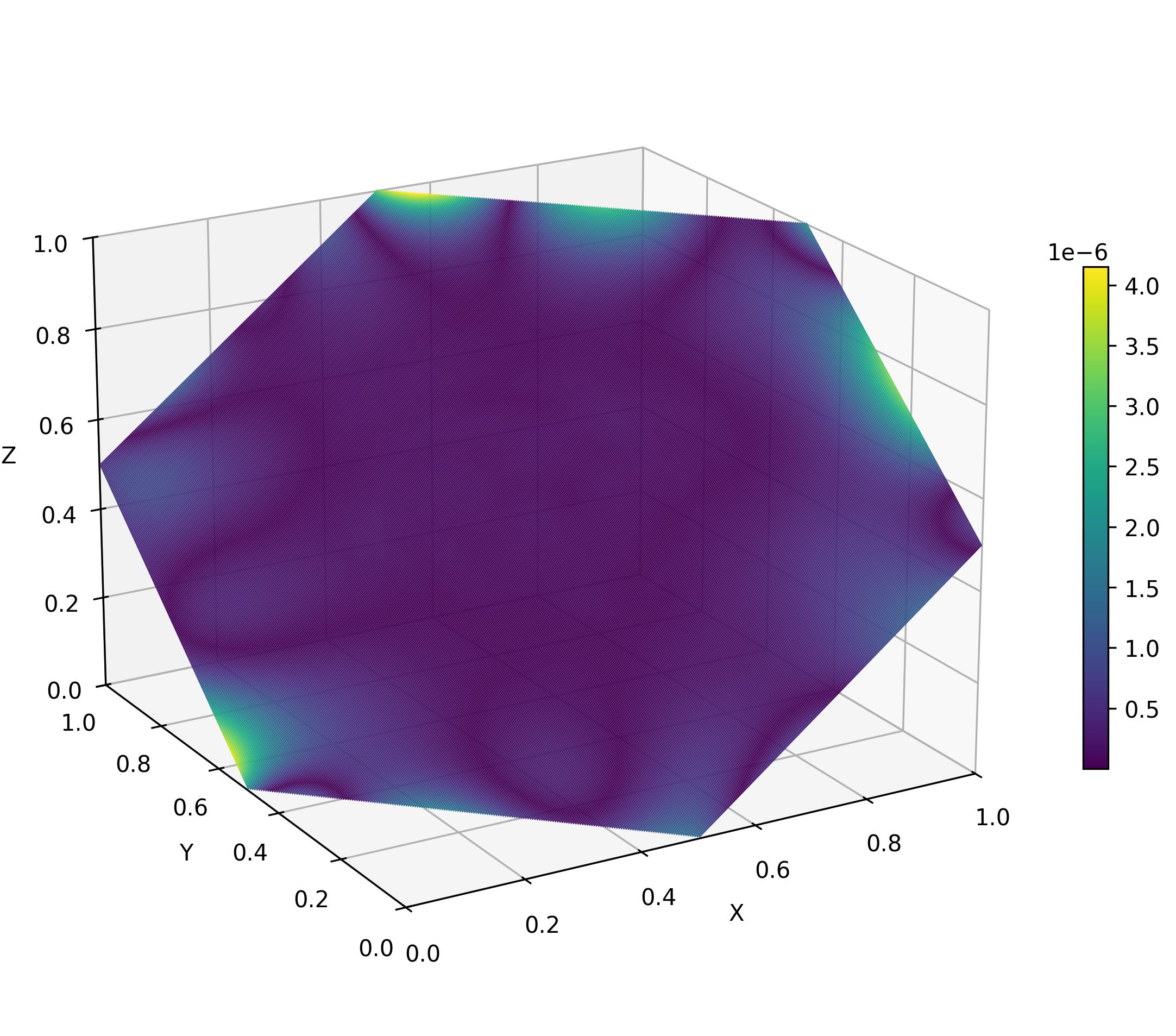}}\hfill
\subfloat[Error, level 3]{\includegraphics[width=0.32\textwidth]{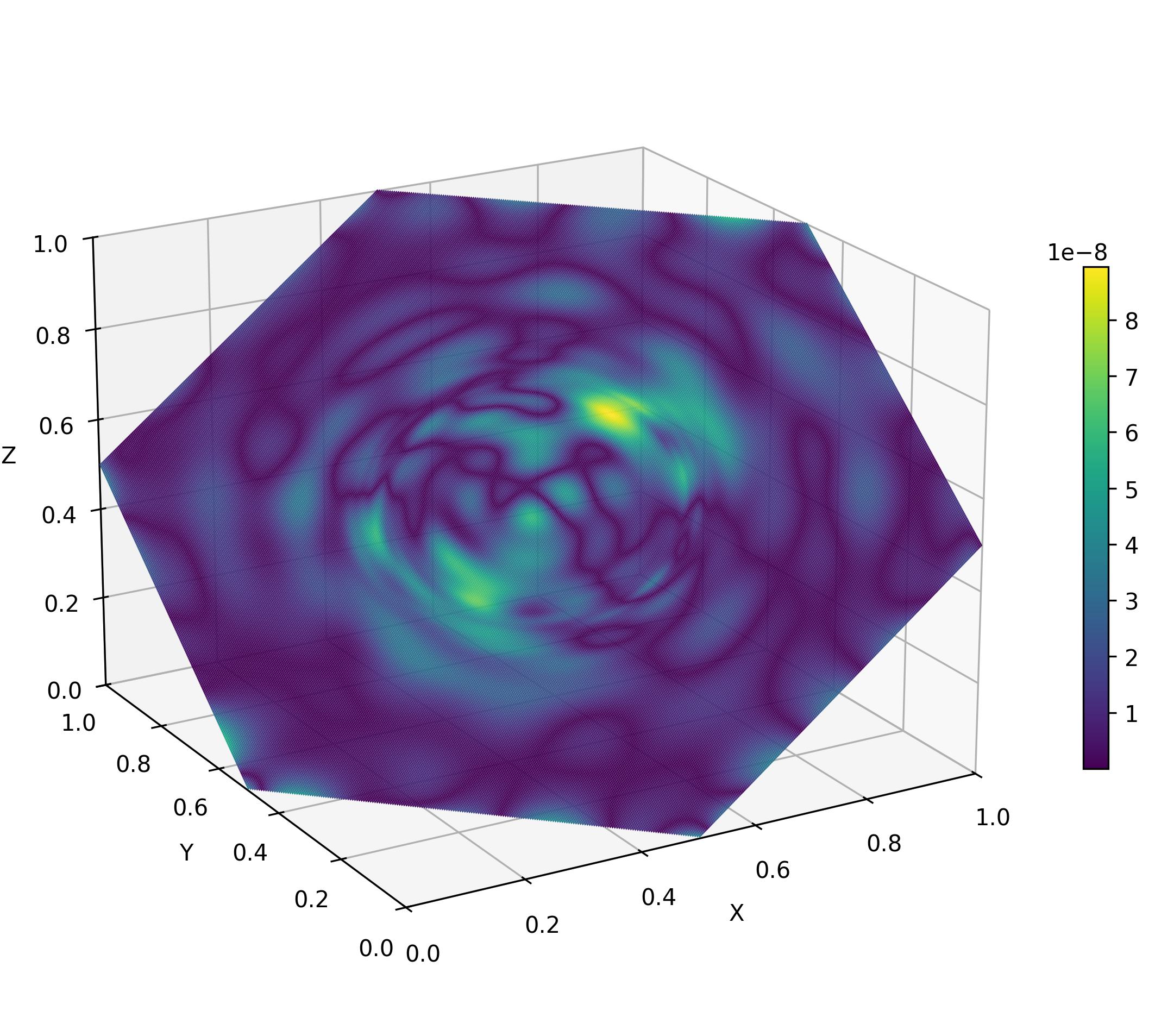}}
\caption{Three-dimensional Poisson problem with a sharp interior layer at levels 1--3: \(xy\)-projections of the collocation-point distributions (top), numerical solutions (middle), and pointwise absolute errors (bottom).}
\label{fig:PoissonSharp3D_points}
\label{fig:PoissonSharp3D_results}
\end{figure}

\subsection{Two-dimensional Poisson equations with two peaks}
\label{sec:OnePeak_2D}
We solve the Poisson problem \eqref{eq:2d_Lshape} on $\Omega=(-1,1)^2$.

The prescribed two-peak solution is
\begin{align}
 u&=e^{-1000\left(x^2+(y-0.5)^2\right)}+e^{-1000\left(x^2+(y+0.5)^2\right)}.
 \label{eq:TwoPeaks2d_solution}
\end{align}

The source term $f$ and boundary data $g$ are determined accordingly.
The solution exhibits two sharply localized peaks centered at
$(0,0.5)$ and $(0,-0.5)$, providing a stringent test of the ability of
the adaptive strategy to resolve multiple localized features.
Figure~\ref{fig:TwoPeaks2D_results} shows the resulting collocation-point
distributions and pointwise absolute errors across successive levels.

\begin{figure}[htbp]
\centering
\subfloat[Points distribution, level 1]
{\includegraphics[width=0.32\textwidth]{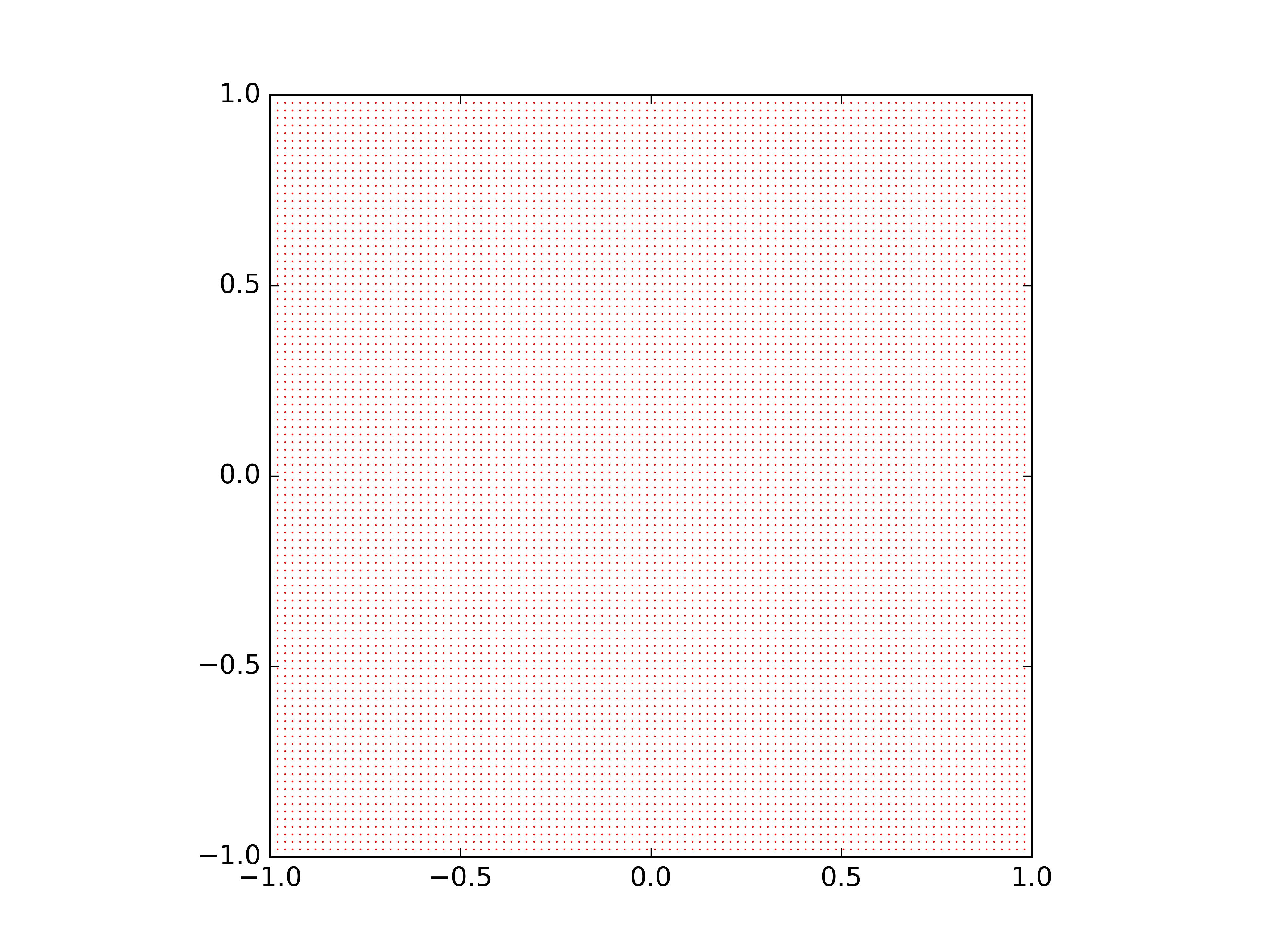}} \hfill
\subfloat[Points distribution, level 2] {\includegraphics[width=0.32\textwidth]{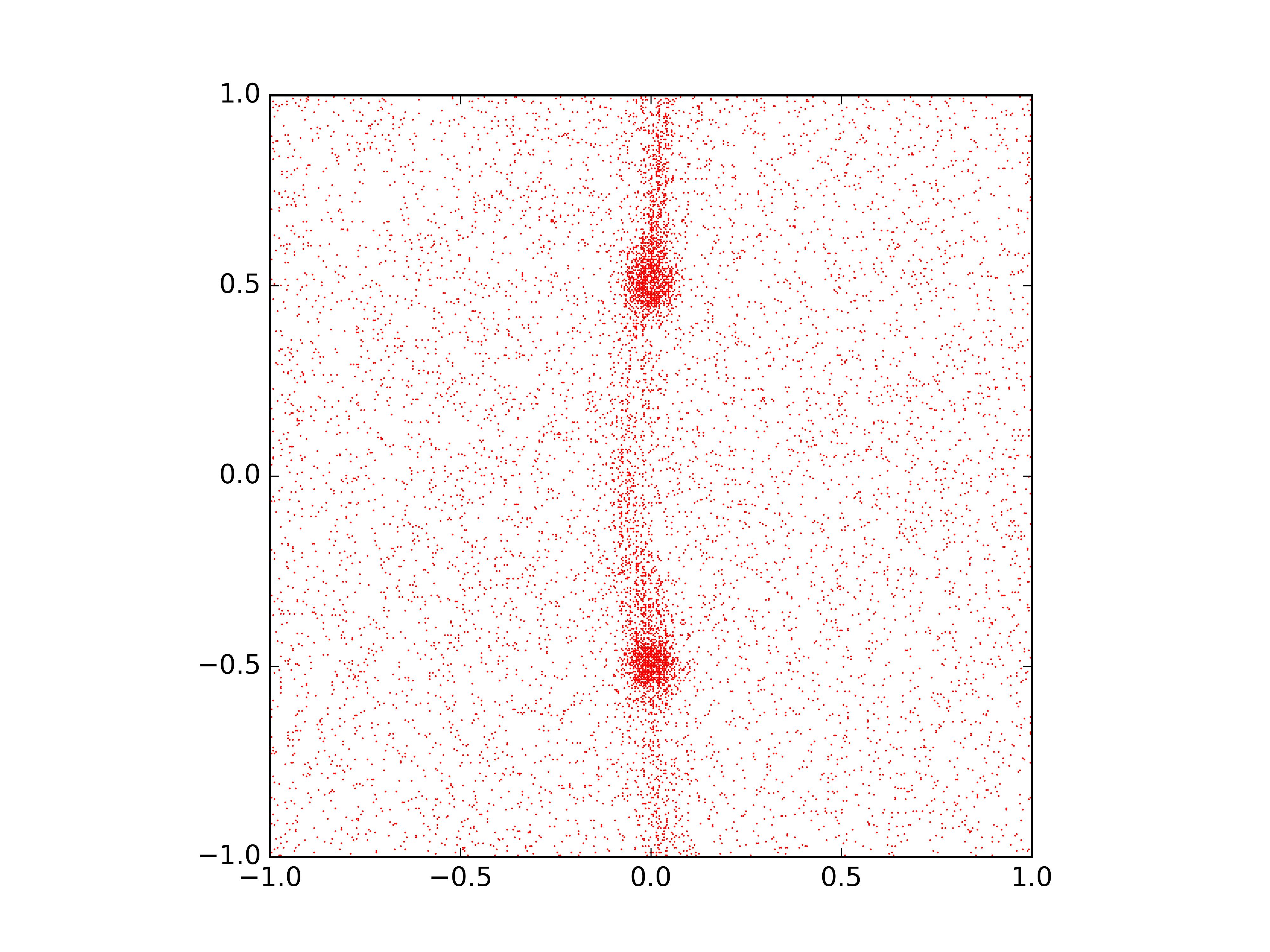}} \hfill
\subfloat[Points distribution, level 3] {\includegraphics[width=0.32\textwidth]{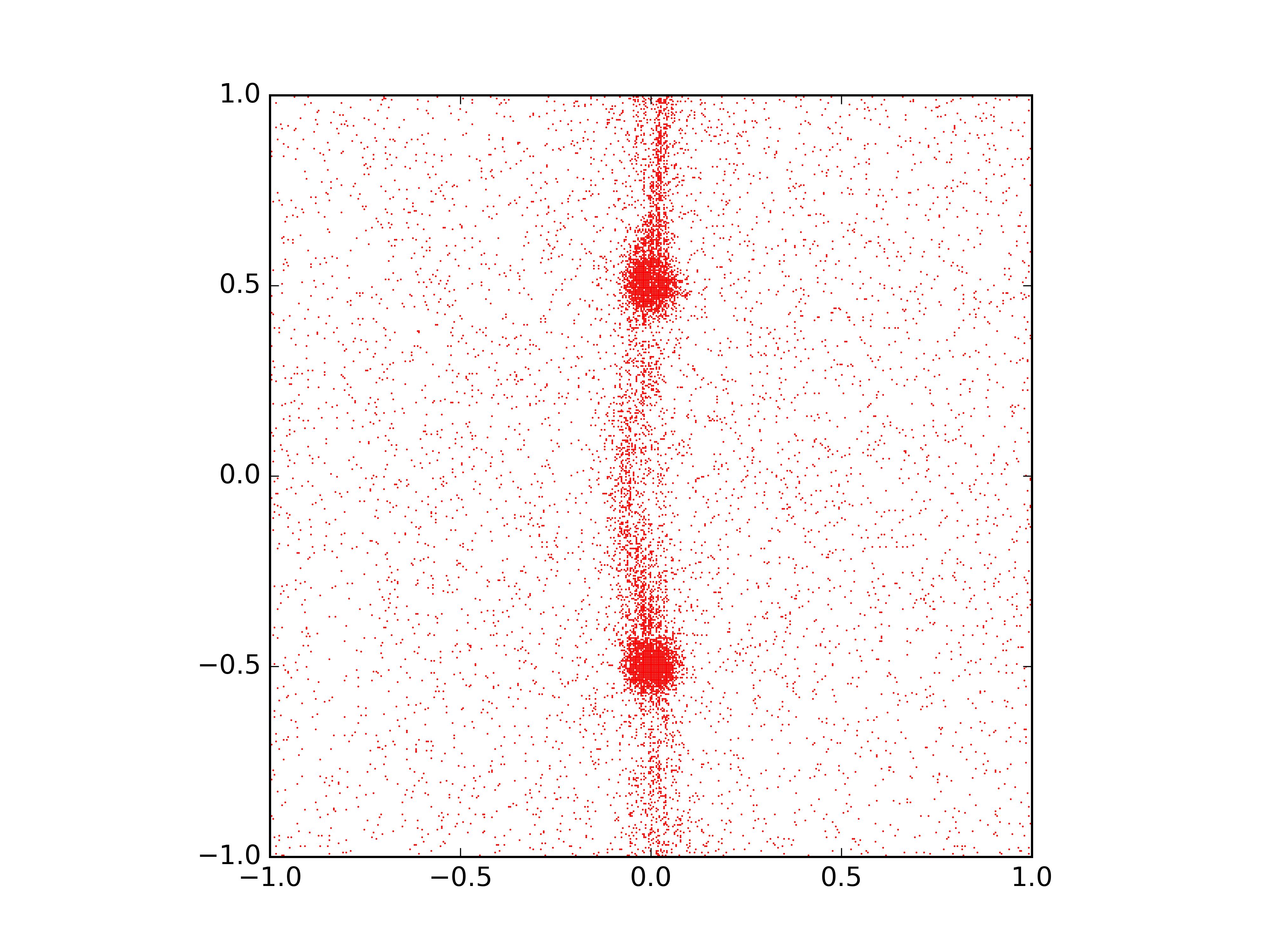}} \\
\subfloat[Error, level 1]
{\includegraphics[width=0.32\textwidth]{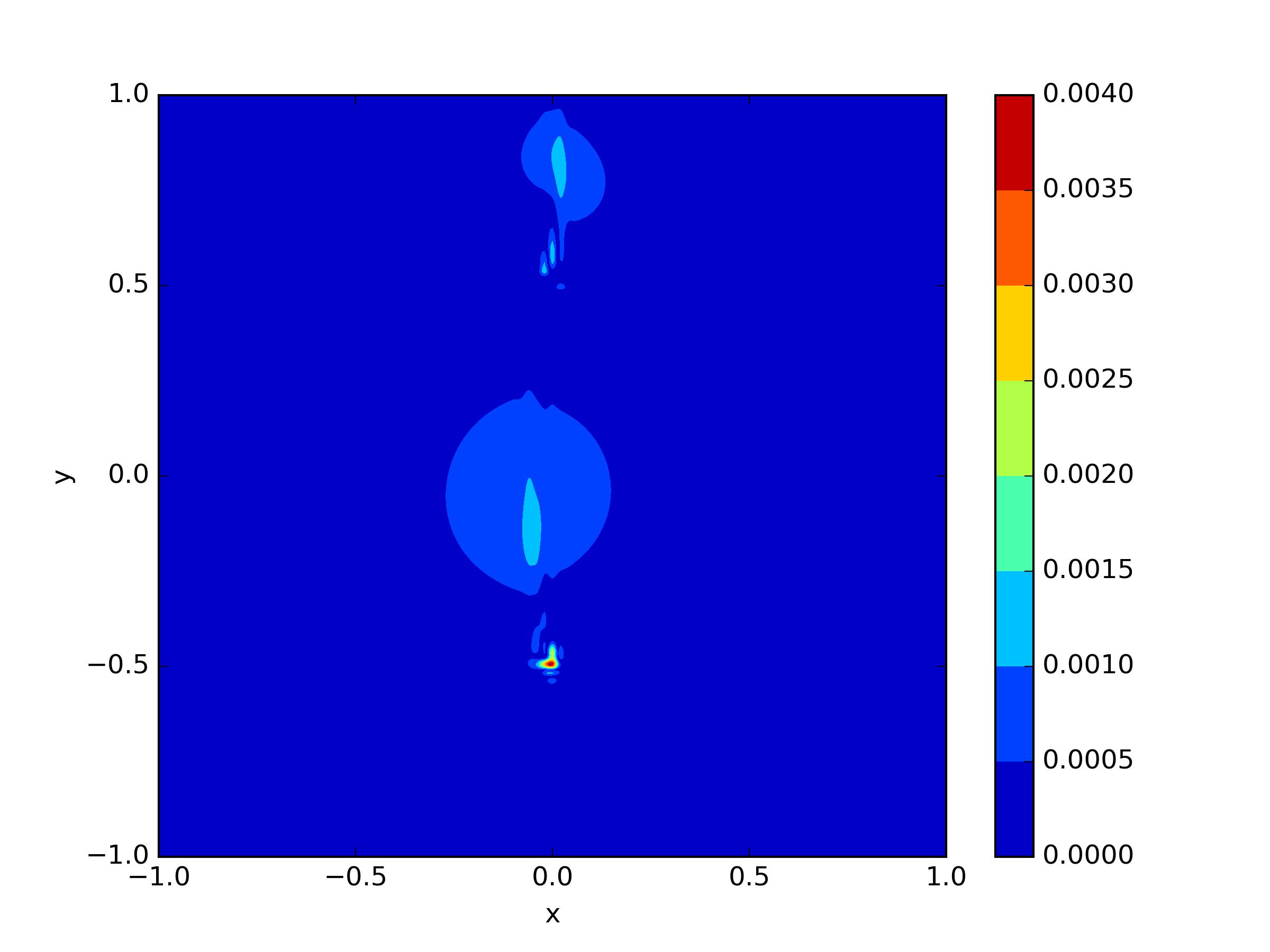}} \hfill
\subfloat[Error, level 2] 
{\includegraphics[width=0.32\textwidth]{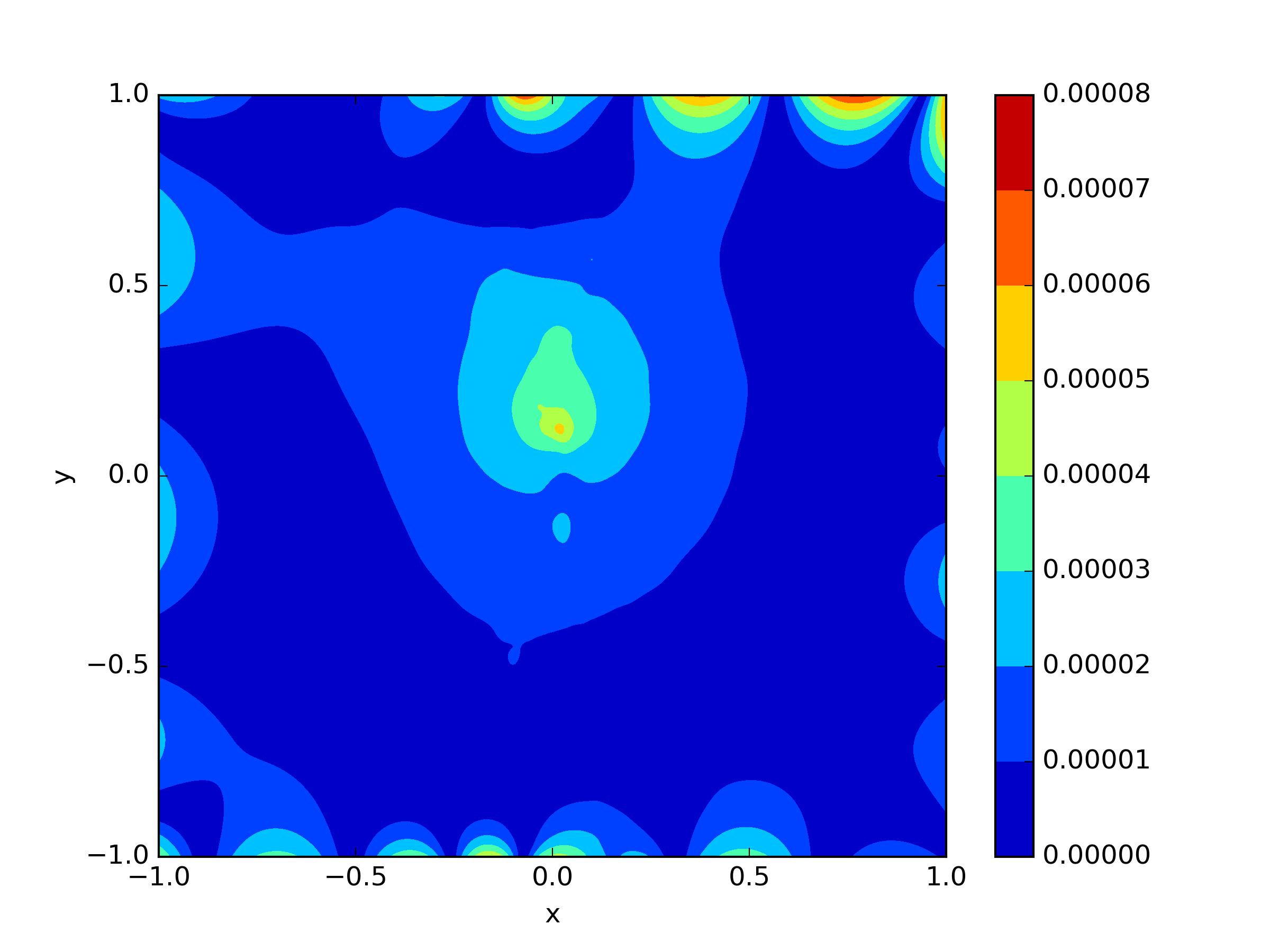}} \hfill
\subfloat[Error, level 3]
{\includegraphics[width=0.32\textwidth]{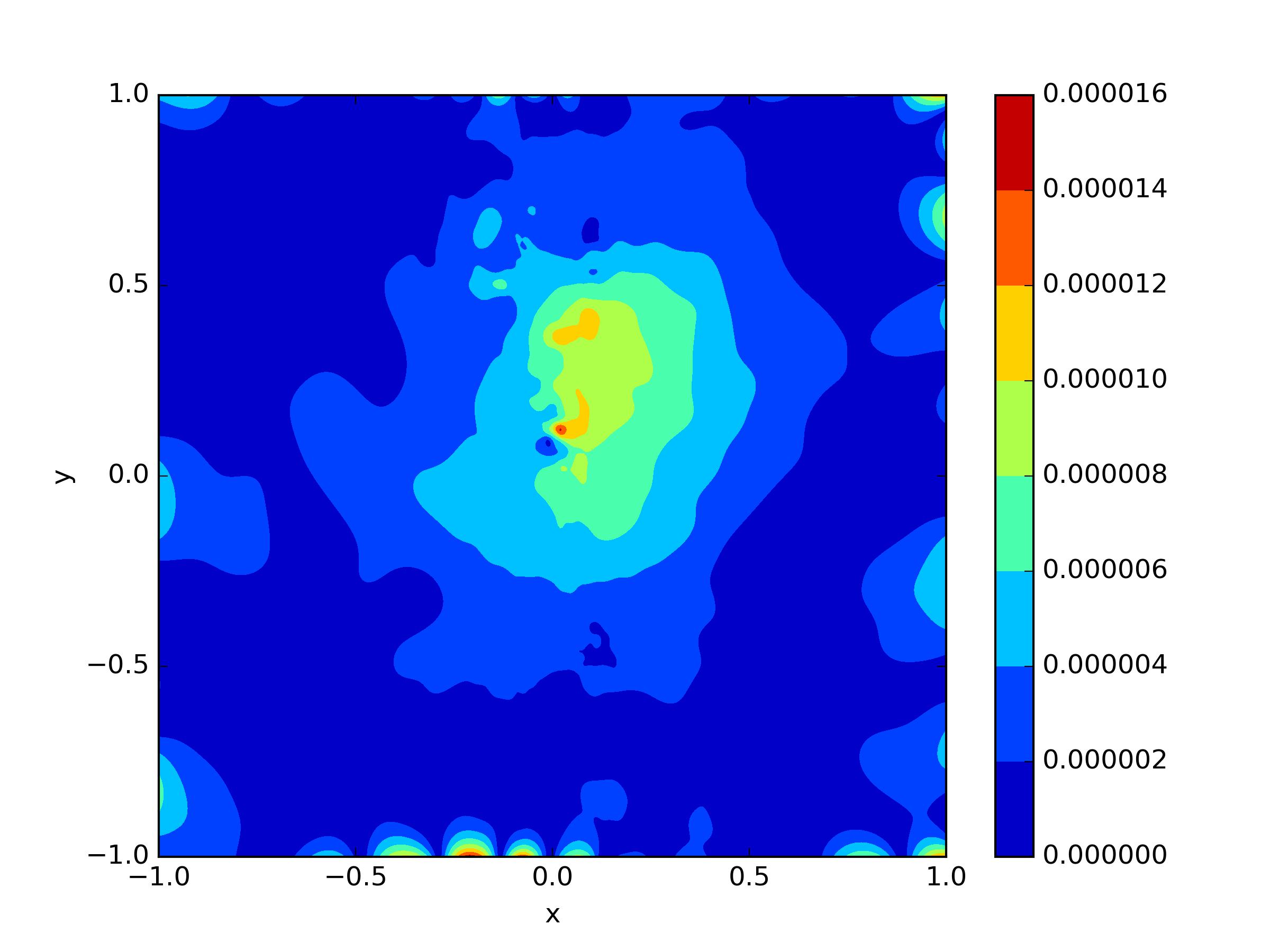}} 
\caption{
Two-peak Poisson problem \eqref{eq:TwoPeaks2d_solution} at levels 1--3: collocation-point distributions (top) and pointwise absolute errors (bottom).}
\label{fig:TwoPeaks2D_results}
\end{figure}

\subsection{Three-dimensional shock cube} 
\label{sec:ShockCube_3D}
We solve the three-dimensional Poisson problem from \cite{d2016multi}

\begin{equation}\label{eq:3d_ShockCube}
\begin{cases}
-\Delta u=f, & \text{in } \Omega,\\
\nabla u\cdot\bn=g_1, & \text{on } \Gamma_N,\\
u=g_2, & \text{on } \Gamma_D,
\end{cases}
\qquad
\Omega=(0,1)^3,
\end{equation}
with
\[
\Gamma_D=\{\boldsymbol 0\},
\qquad
\Gamma_N=\partial\Omega\setminus\Gamma_D.
\]
Let
\[
r(\bx)=\|\bx-\bx_0\|_2,
\qquad
\bx_0=\left(-\frac14,-\frac14,-\frac14\right)^T.
\]
The exact solution is
\begin{equation}\label{eq:ShockCube_solution}
u(\bx)
=
\arctan\!\bigl(\alpha(r(\bx)-r_0)\bigr)
-\arctan(-\alpha r_0),
\end{equation}
where $\alpha=80$ and $r_0=\sqrt{3}$. The data $f$, $g_1$, and $g_2$
are determined from \eqref{eq:ShockCube_solution}. The solution exhibits a
sharp transition near the spherical surface $r(\bx)=r_0$.

The solution develops a sharp transition near \(r(\bx)=r_0\), with steepness controlled by \(\alpha\). Figures~\ref{fig:ShockCube3D_results} and~\ref{fig:ShockCube3D_points} show the numerical approximation and the adaptive sampling pattern on the slice \(z=0.5\). MLS concentrates collocation points near the high-gradient region, and successive corrections reduce the errors to $e_\infty(\bu)=1.471\times10^{-7}$ and $e_2(\bu)=1.027\times10^{-7}$.

\begin{figure}[htbp]
\centering
\subfloat[Exact]{\includegraphics[width=0.24\textwidth]{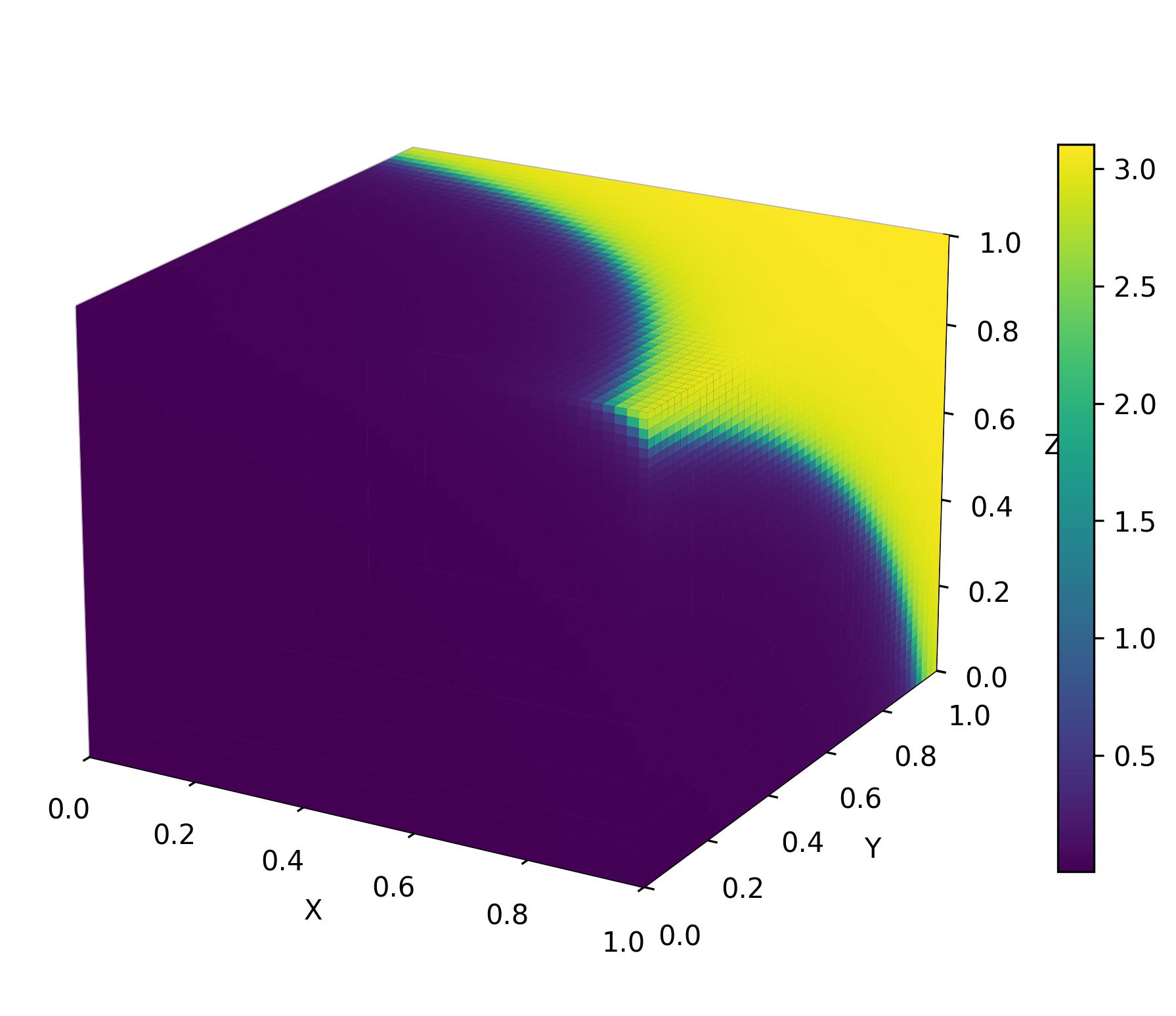}}\hfill
\subfloat[Level 1]{\includegraphics[width=0.24\textwidth]{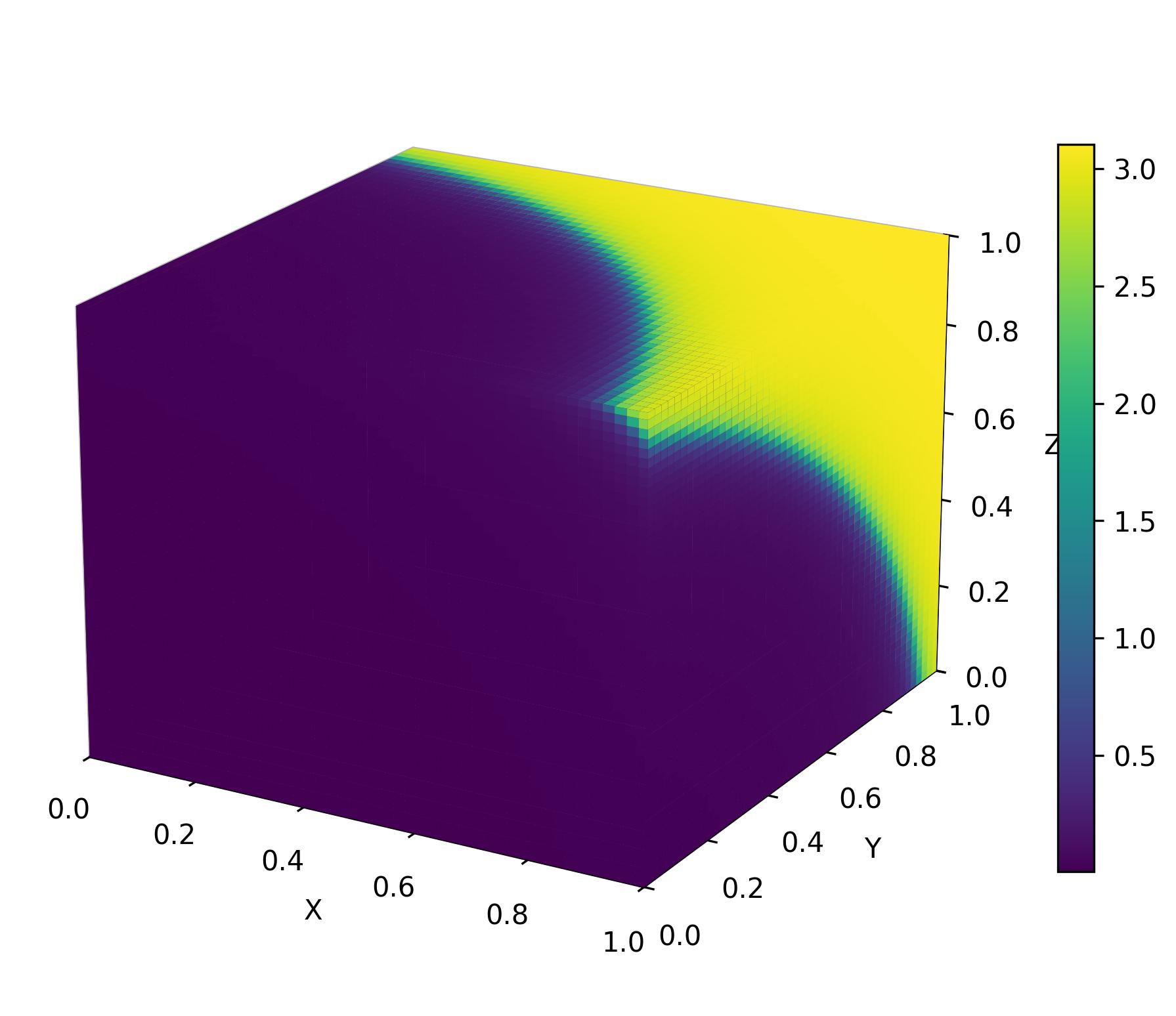}}\hfill
\subfloat[Level 2]{\includegraphics[width=0.24\textwidth]{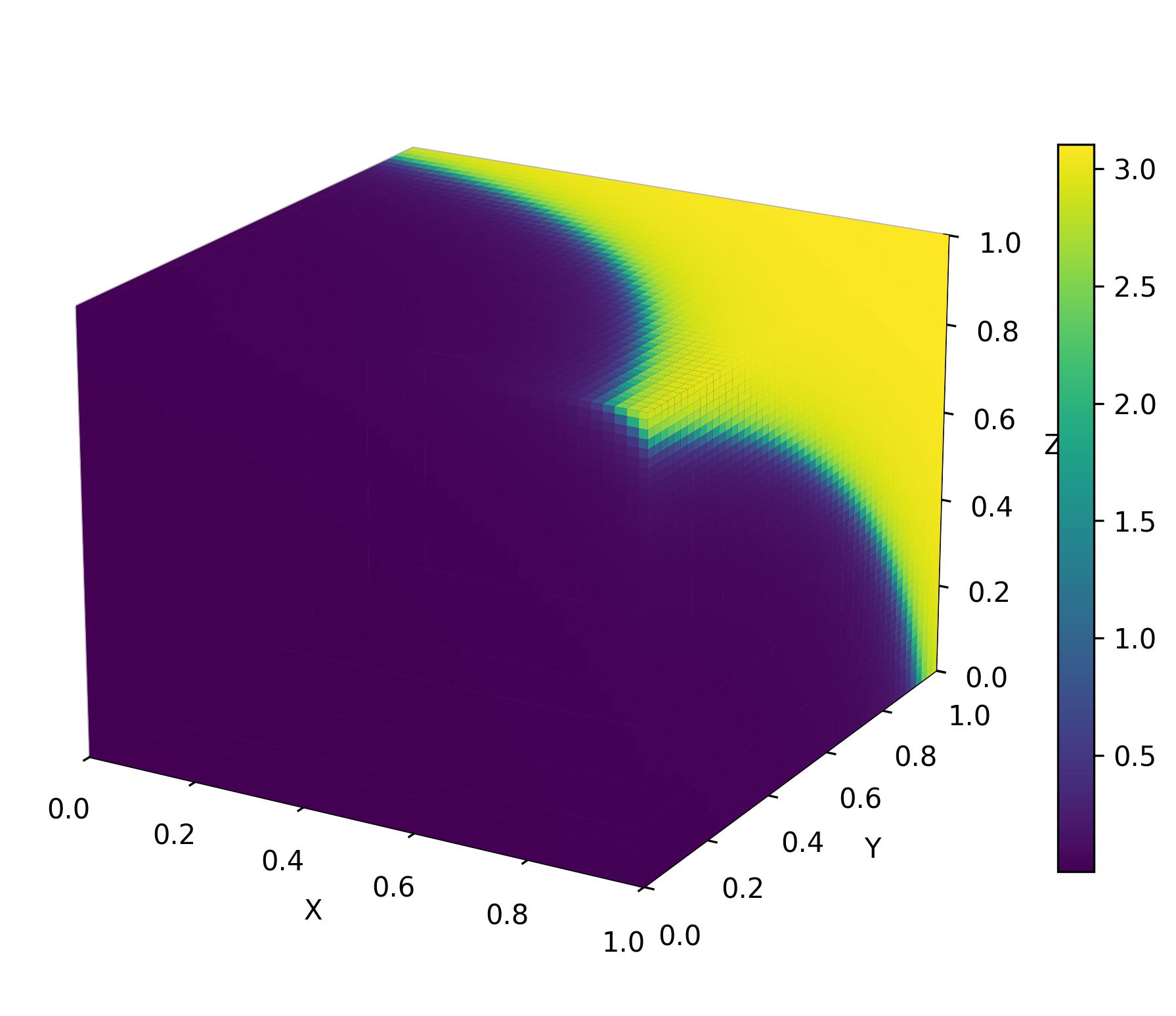}}\hfill
\subfloat[Level 3]{\includegraphics[width=0.24\textwidth]{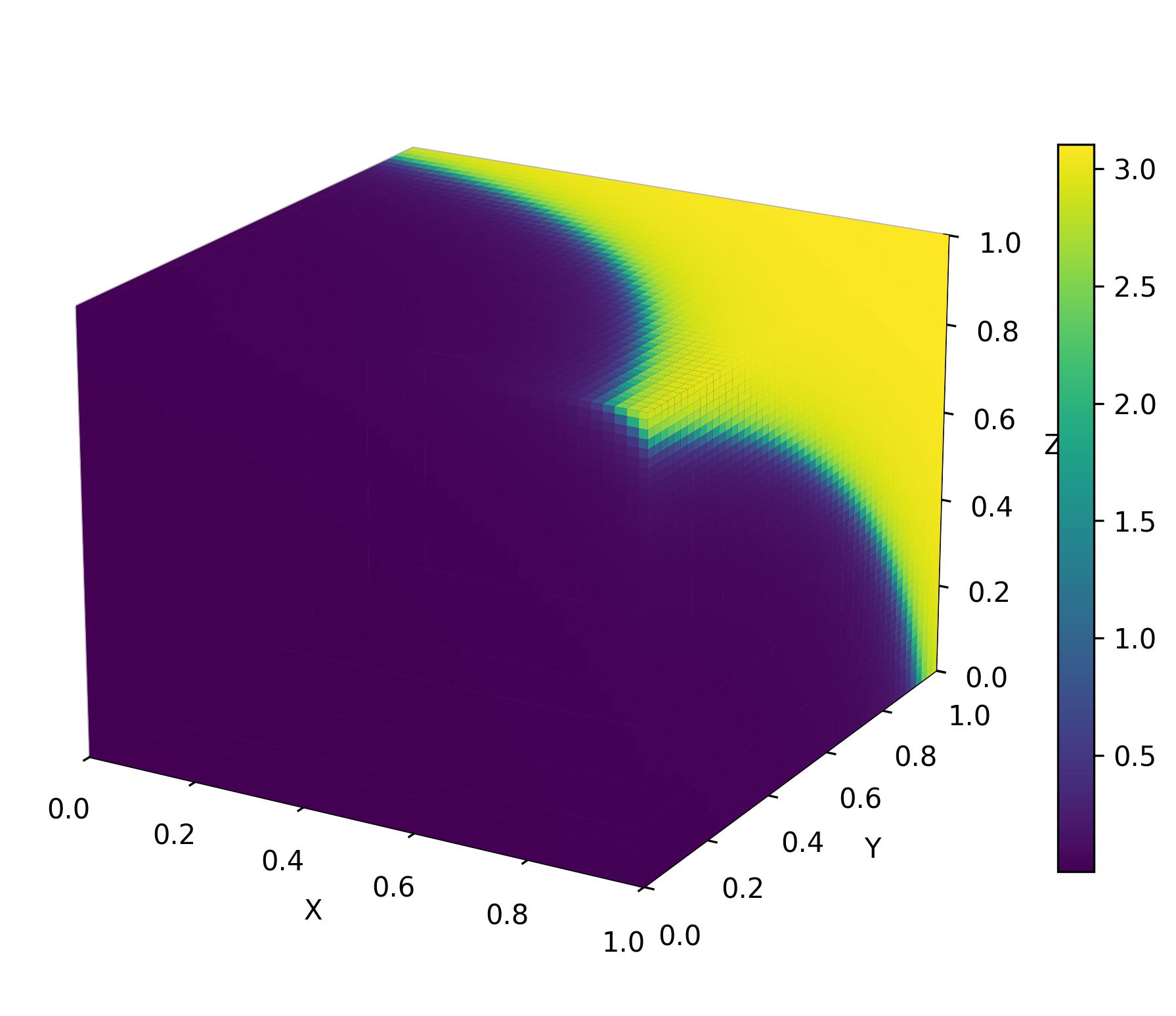}}
\caption{Shock cube: exact and predicted solutions.}
\label{fig:ShockCube3D_results}
\end{figure}

\begin{figure}[htbp]
\centering
\subfloat[Points distribution, level 1]{\includegraphics[width=0.32\textwidth]{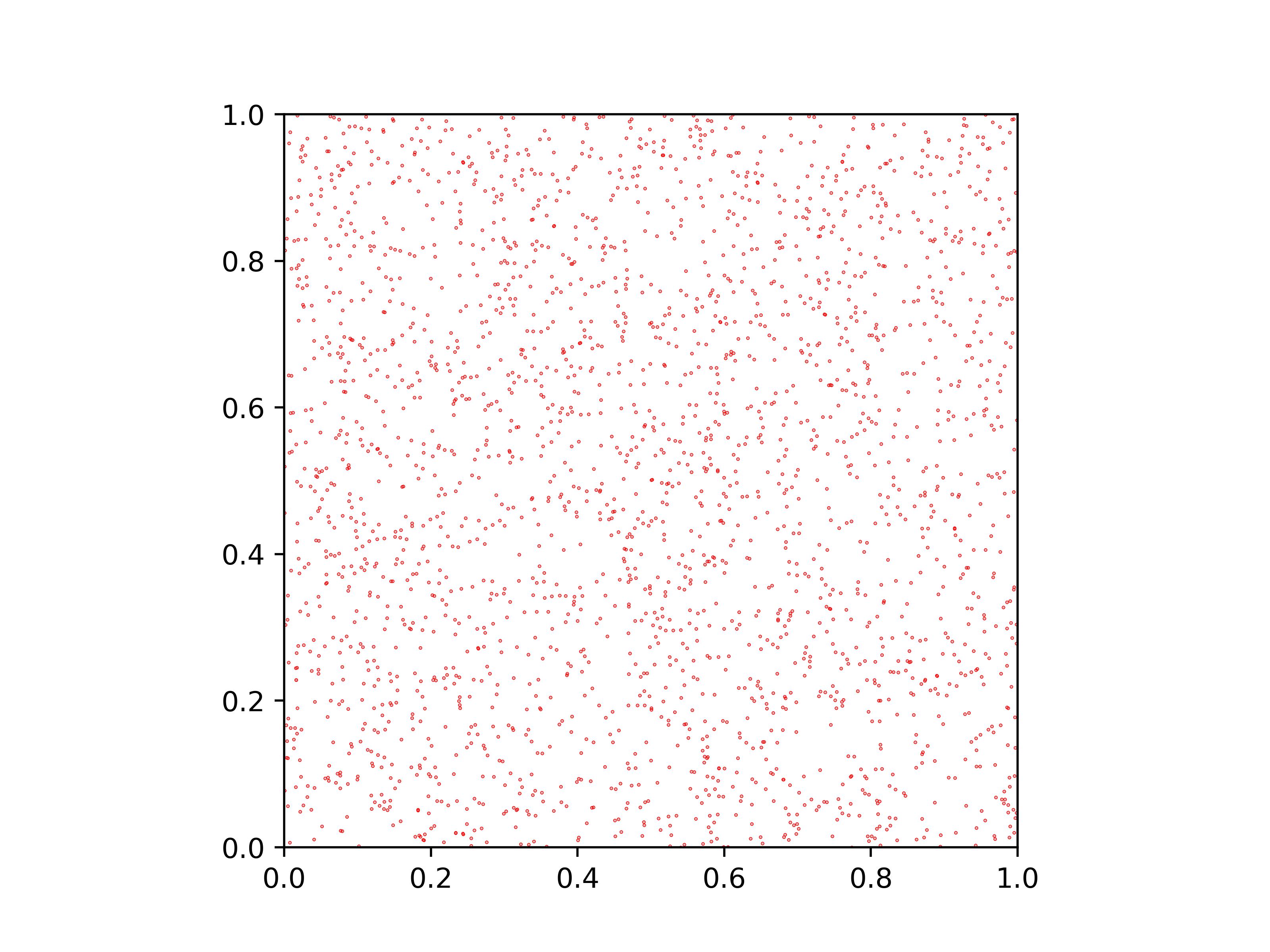}}\hfill
\subfloat[Points distribution, level 2]{\includegraphics[width=0.32\textwidth]{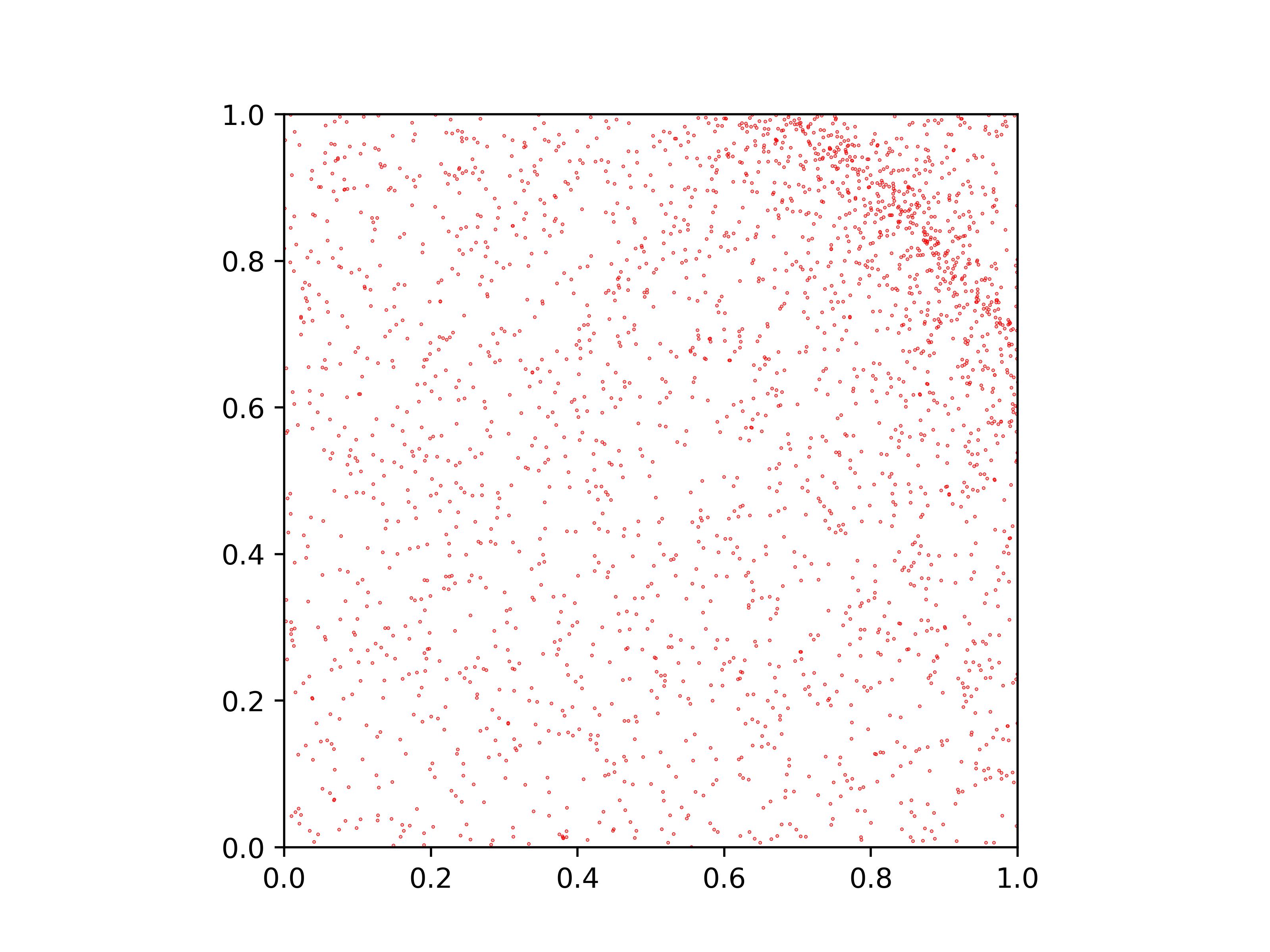}}\hfill
\subfloat[Points distribution, level 3]{\includegraphics[width=0.32\textwidth]{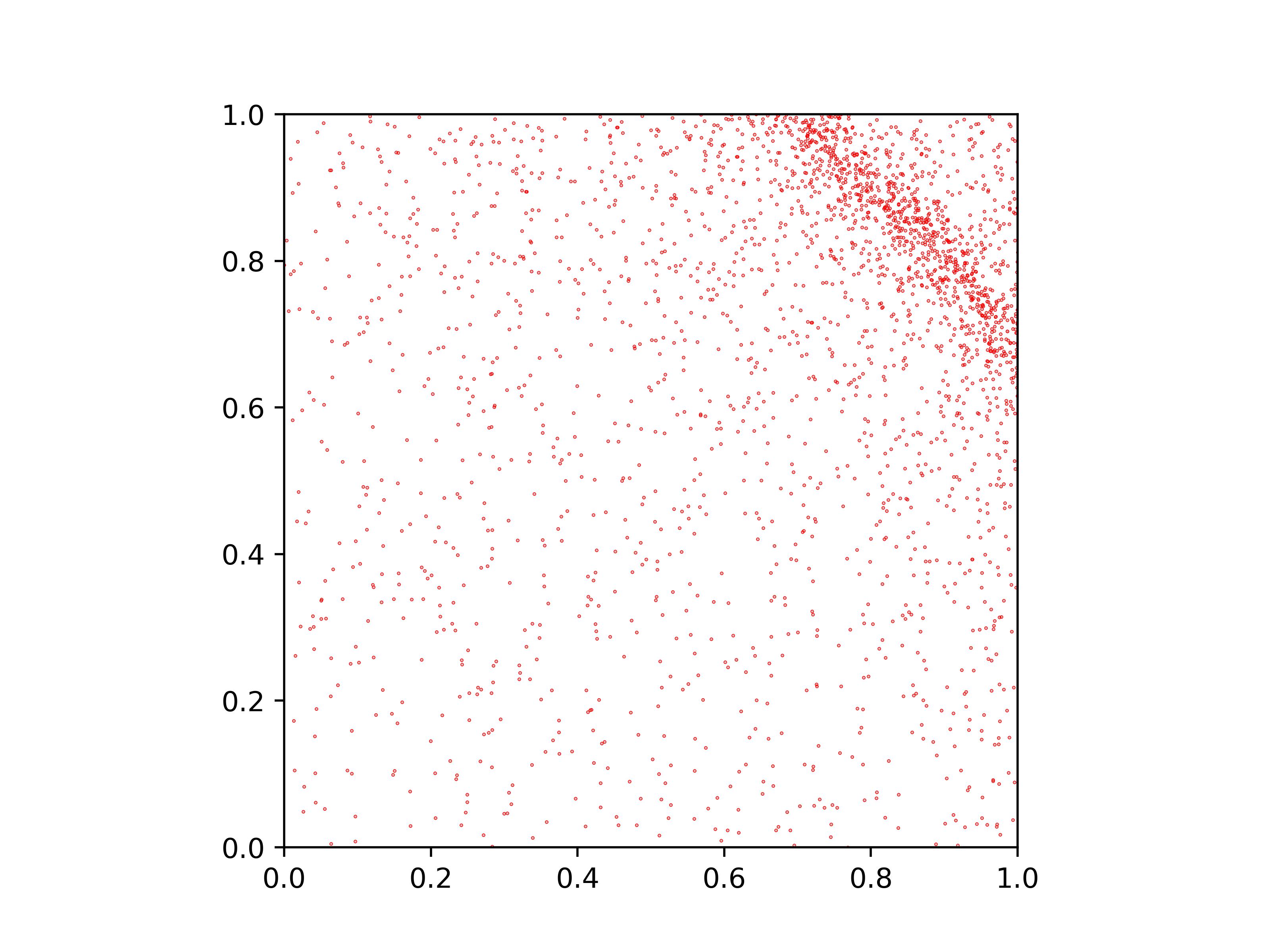}}\\
\subfloat[Error, level 1]{\includegraphics[width=0.32\textwidth]{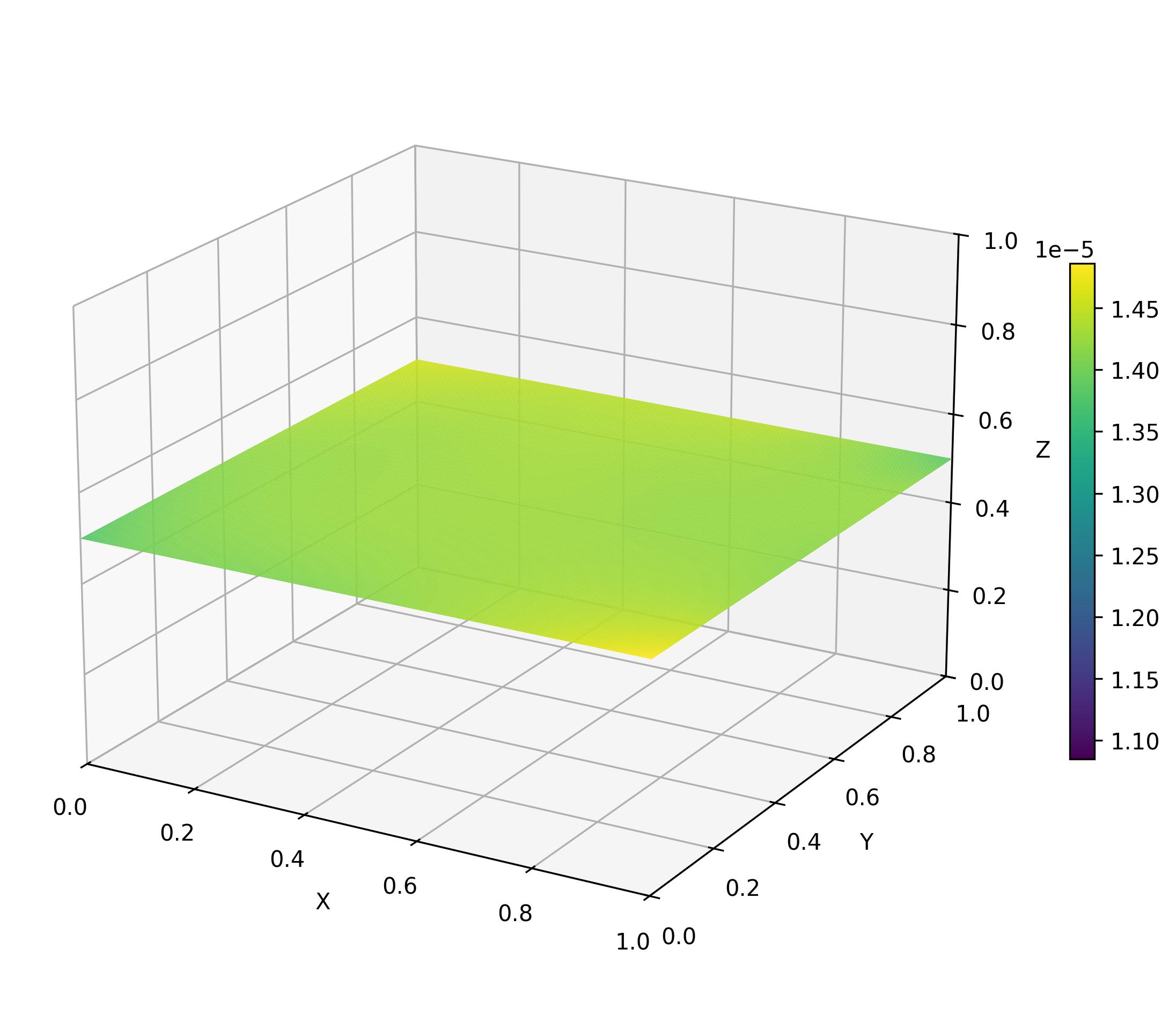}}\hfill
\subfloat[Error, level 2]{\includegraphics[width=0.32\textwidth]{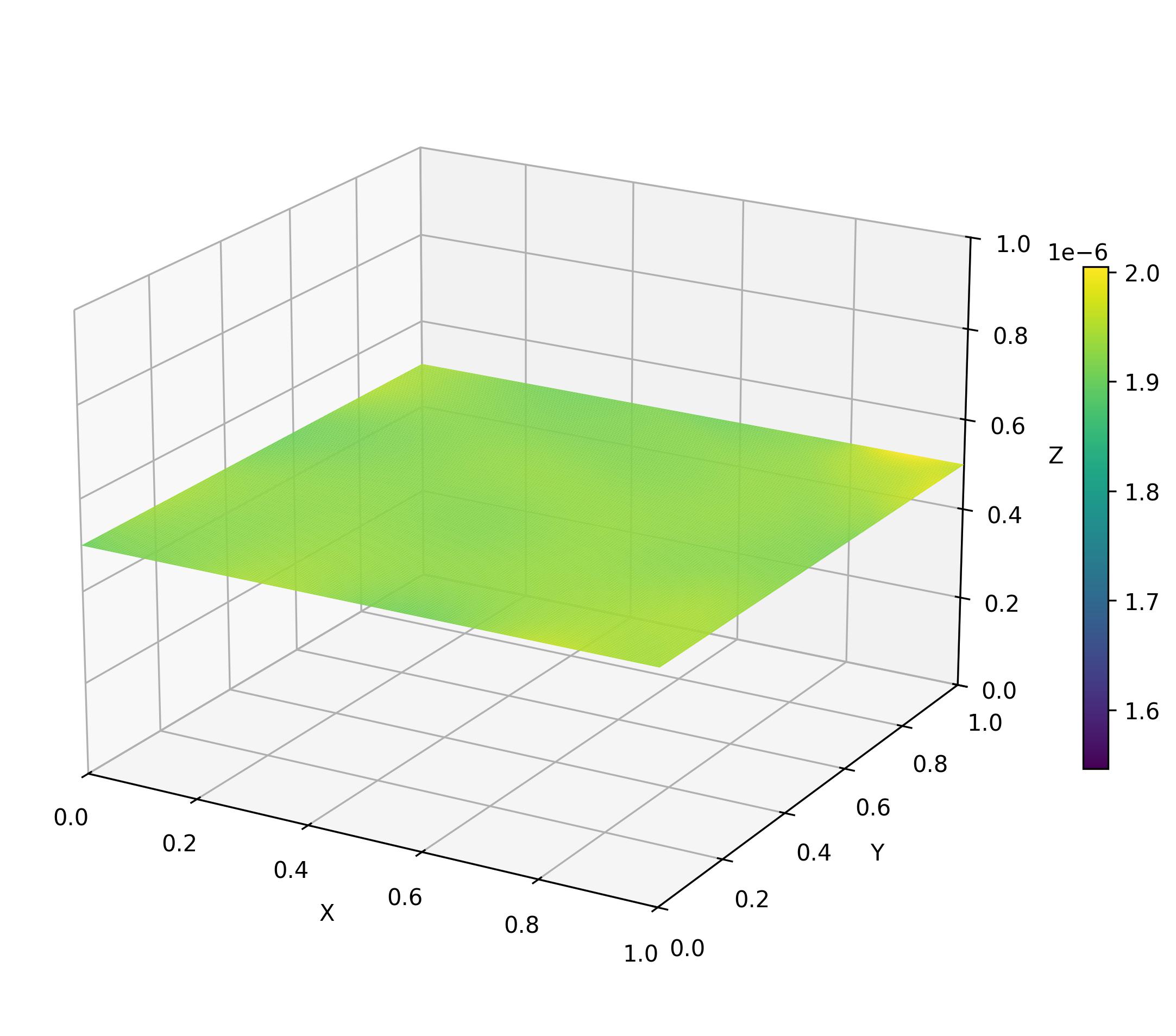}}\hfill
\subfloat[Error, level 3]{\includegraphics[width=0.32\textwidth]{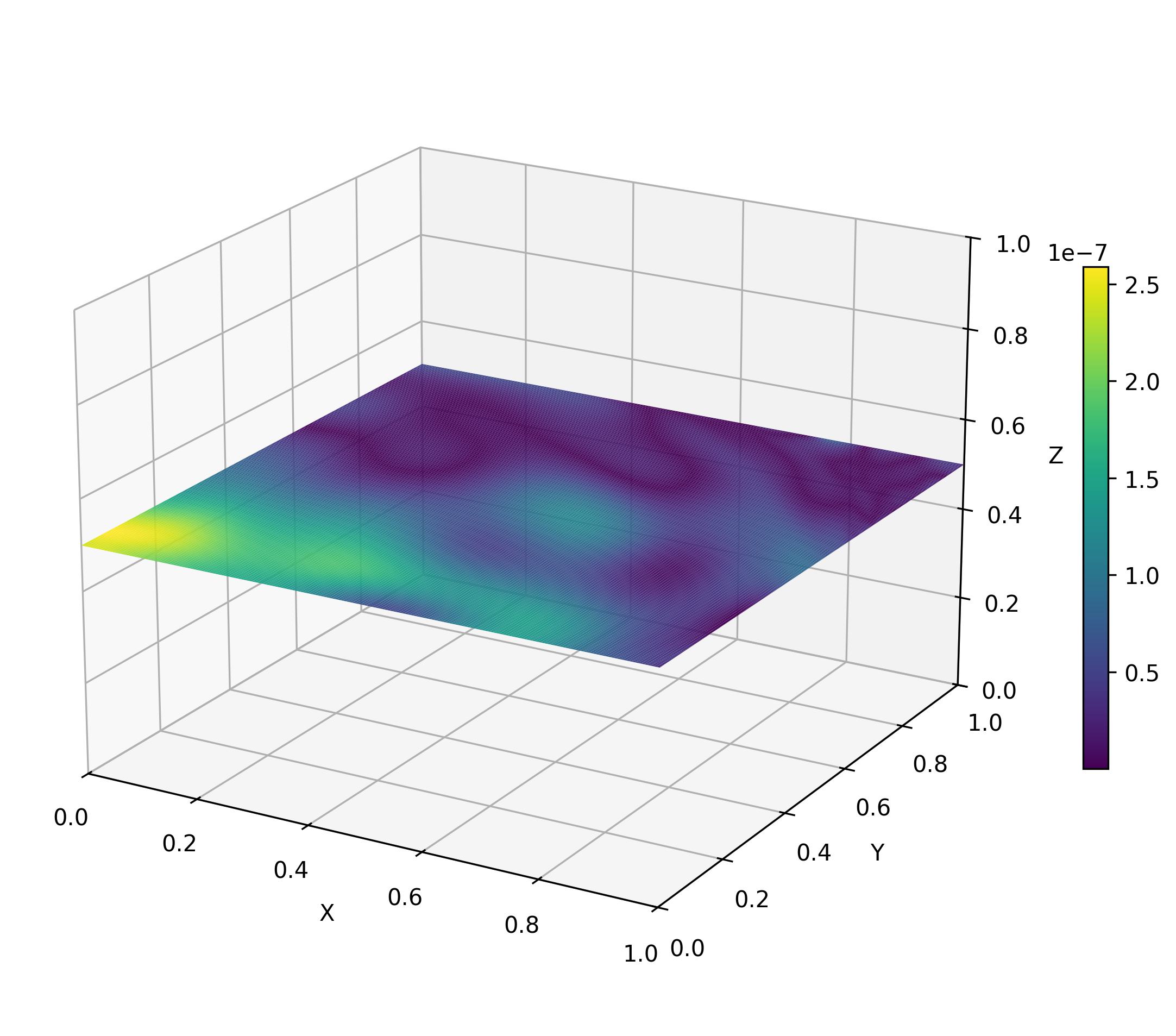}}
\caption{
Sharp-layer Poisson problem on the slice \(z=0.5\) at levels 1--3: collocation points (top) and pointwise absolute errors (bottom).}
\label{fig:ShockCube3D_points}
\label{fig:ShockCube3D_sliceresults}
\end{figure}

\subsection{Fourth-order biharmonic equation}
\label{sec:Biharmonic_2D}
On the unit disk
\[
\Omega:=\{(x,y)\in\mathbb R^2:x^2+y^2<1\},
\]
we consider the biharmonic problem
\begin{equation}\label{eq:2d_Biharmonic}
\begin{cases}
\Delta^2 u=\delta_{(0,0)}, & \text{in } \Omega,\\
u-\epsilon_1\partial_n(\Delta u)=0, & \text{on } \partial\Omega,\\
\Delta u+\epsilon_2\partial_n u=0, & \text{on } \partial\Omega,
\end{cases}
\end{equation}
where $\delta_{(0,0)}$ denotes the Dirac measure concentrated at the origin.
With $r=(x^2+y^2)^{1/2}$, the exact solution is
\begin{equation}\label{eq:2d_Biharmonicsolution}
u(r)
=
\frac{r^2\ln r}{8\pi}
+c_1r^2+c_2,
\end{equation}
where
\[
c_1
=
\frac{-1/(2\pi)-\epsilon_2/(8\pi)}
     {4+2\epsilon_2},
\qquad
c_2
=
-c_1+\frac{\epsilon_1}{2\pi}.
\]
The term $r^2\ln r$ induces reduced regularity at the origin, providing a
test of the ability of the proposed method to resolve localized singular
behavior in a fourth-order PDE. Figure~\ref{fig:2dBiharmonic_results} shows the resulting collocation-point
distributions and pointwise absolute errors across successive levels.
\begin{figure}[htbp]
\centering
\subfloat[Points distribution, level 1]
{\includegraphics[width=0.32\textwidth]{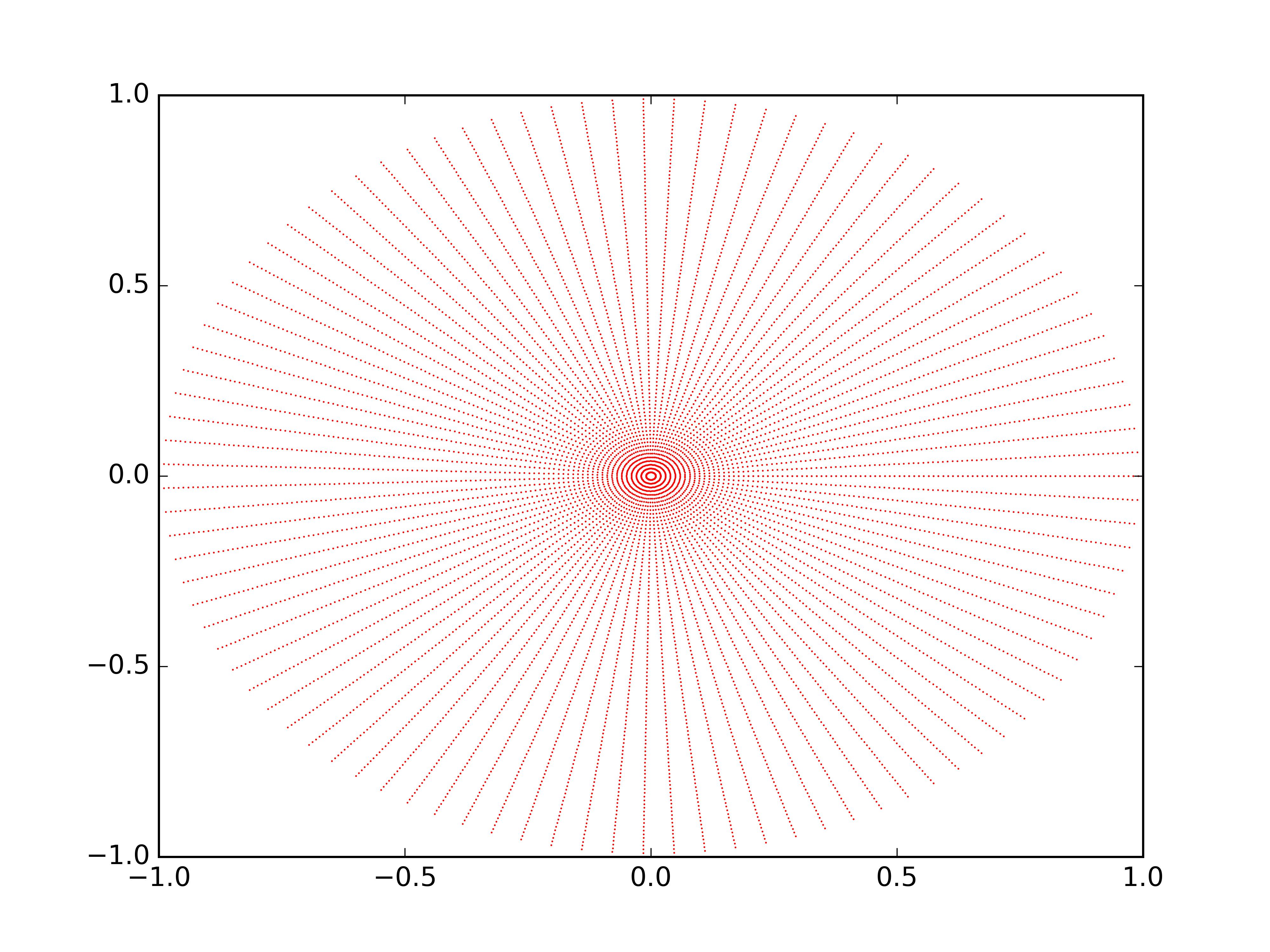}} \hfill
\subfloat[Points distribution, level 2]
{\includegraphics[width=0.32\textwidth]{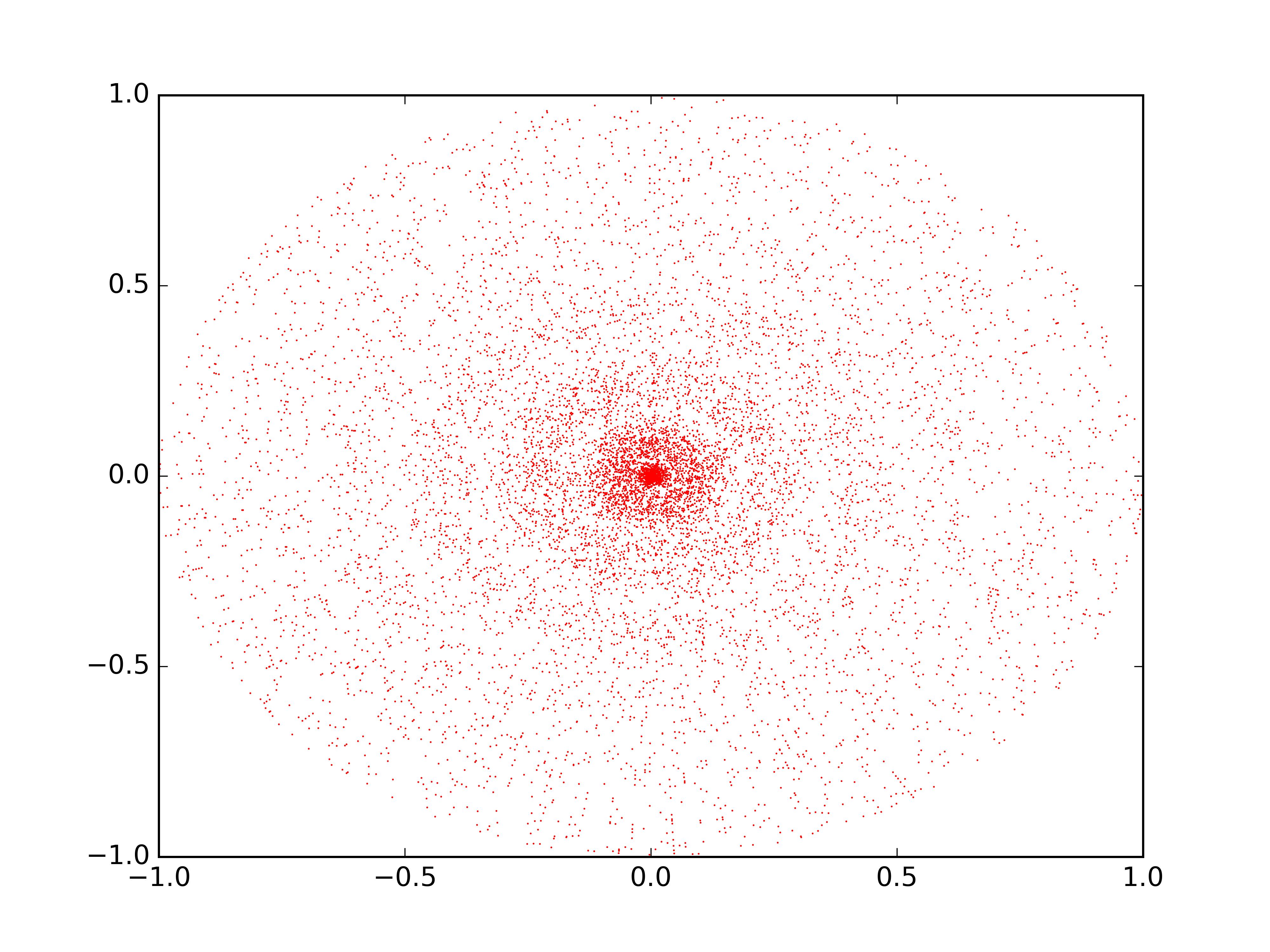}} \hfill
\subfloat[Points distribution, level 3]
{\includegraphics[width=0.32\textwidth]{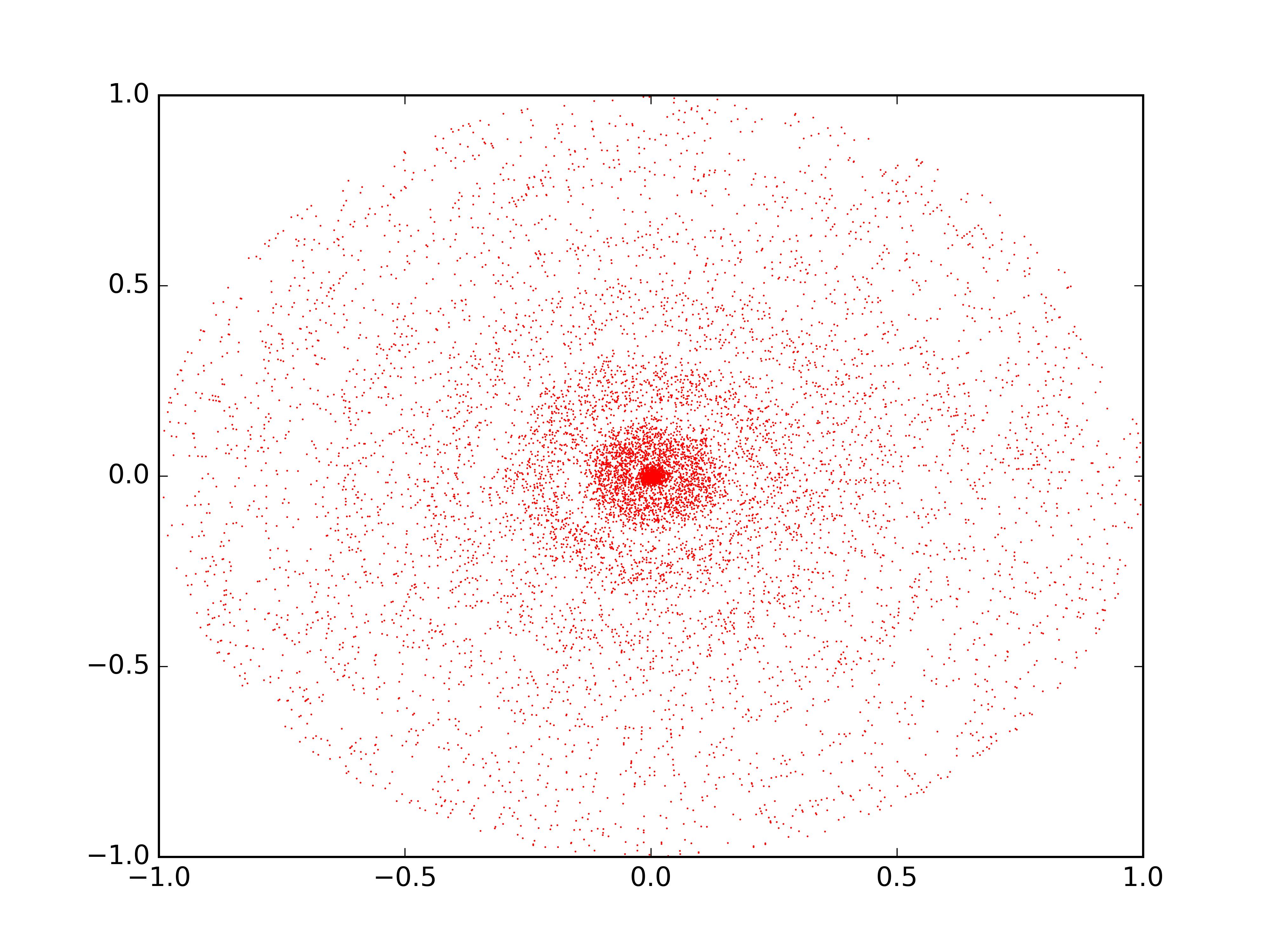}} \hfill
\subfloat[Error, level 1]
{\includegraphics[width=0.32\textwidth]{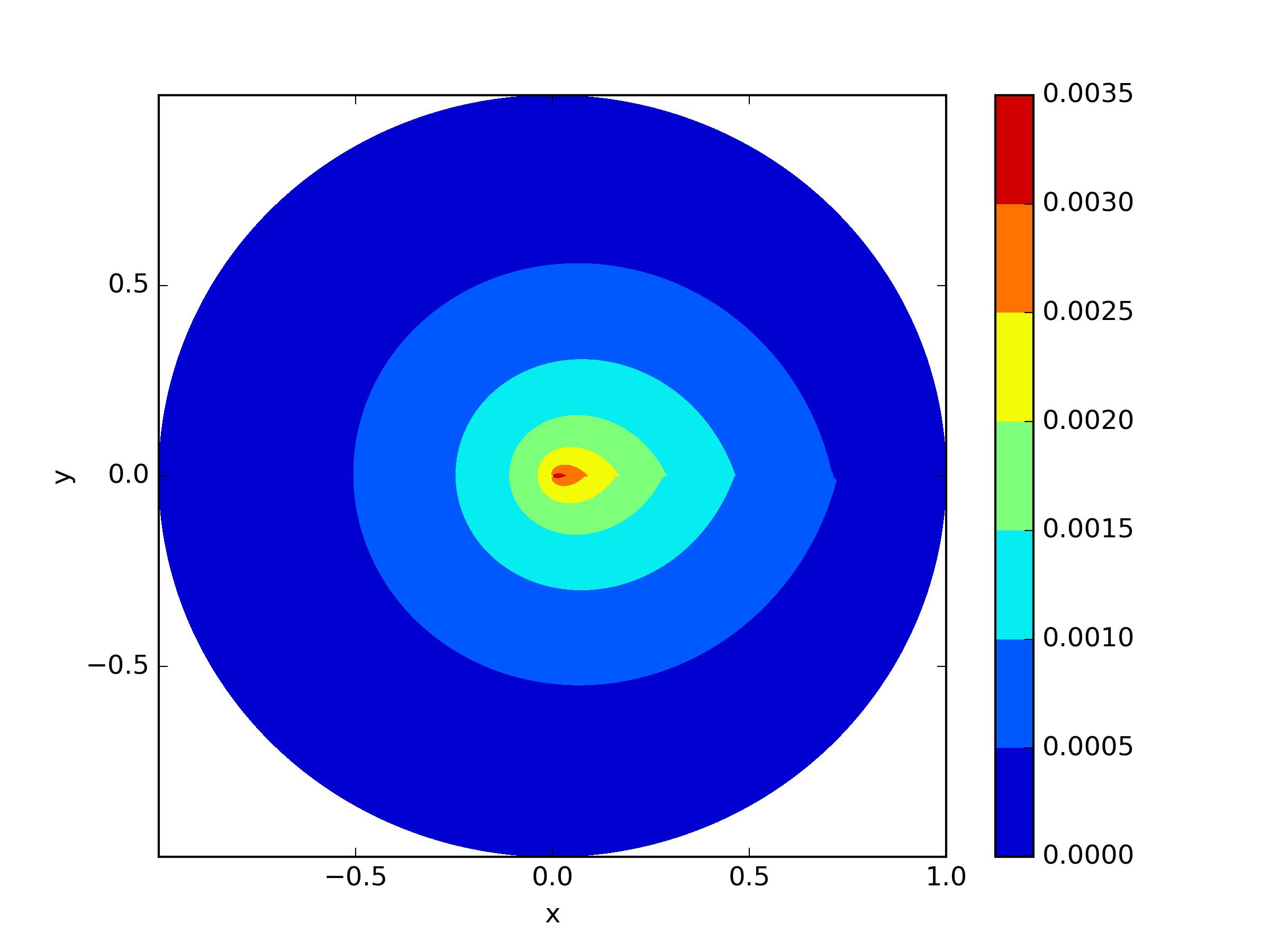}} \hfill
\subfloat[Error, level 2]
{\includegraphics[width=0.32\textwidth]{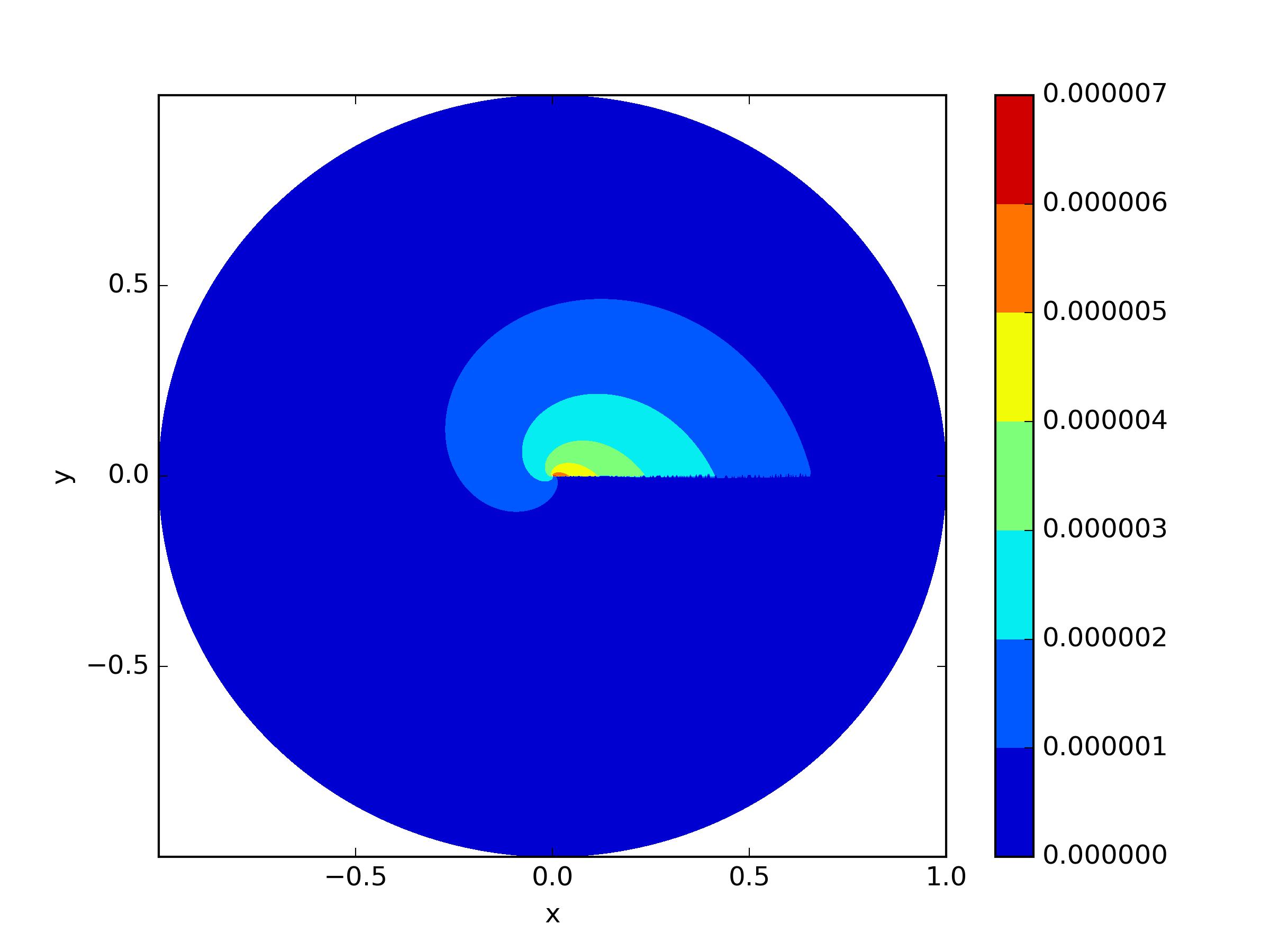}} \hfill
\subfloat[Error, level 3]
{\includegraphics[width=0.32\textwidth]{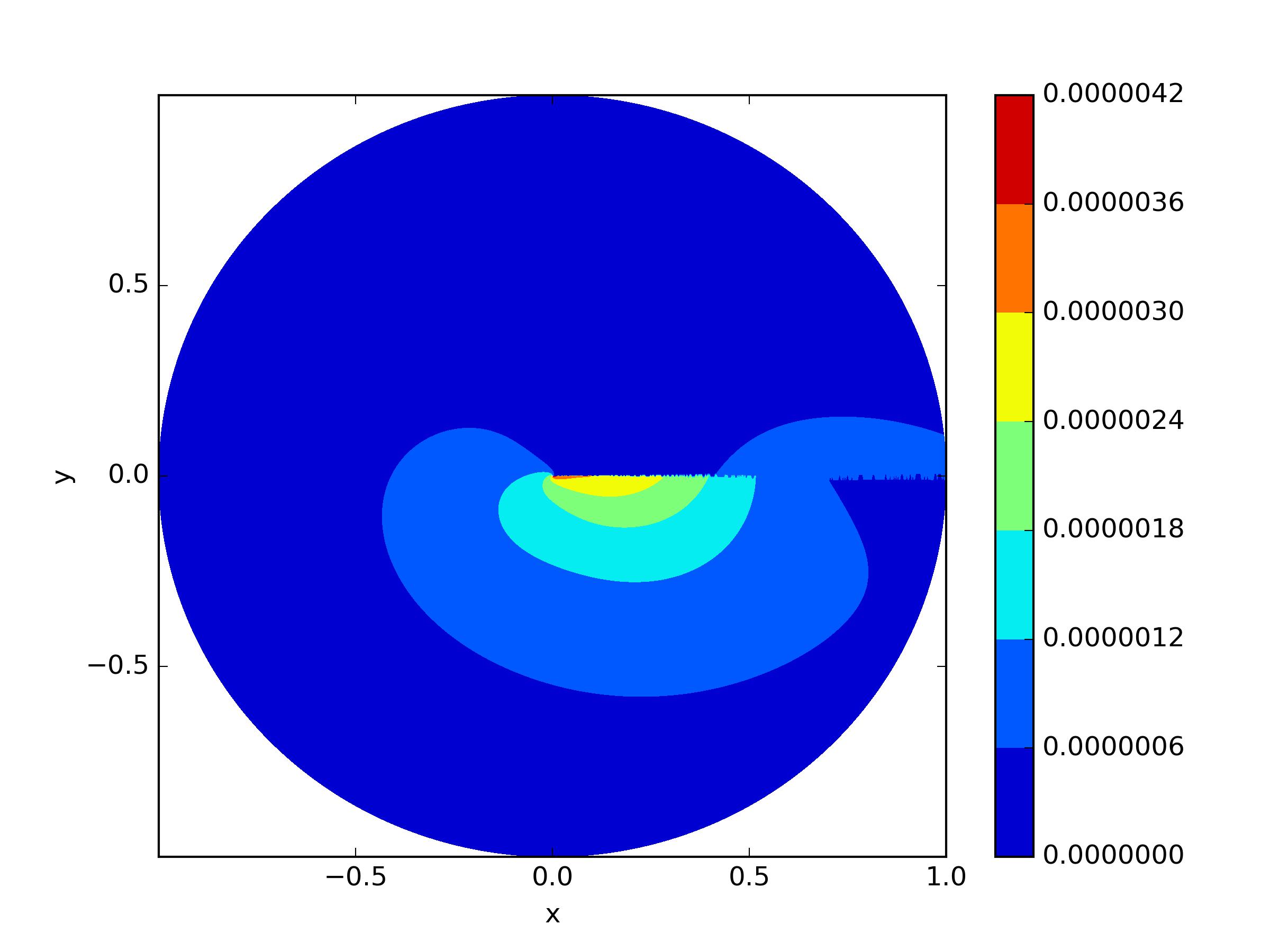}} 
\caption{Biharmonic problem \eqref{eq:2d_Biharmonicsolution} at levels 1--3: collocation-point distributions (top) and pointwise absolute errors (bottom).}
\label{fig:2dBiharmonic_results}
\end{figure}

\subsection{Two-dimensional Helmholtz equation}
\label{sec:Helmholtz_2D}
We solve the two-dimensional Helmholtz equation
\begin{equation}\label{eq:2d_Helmholtz}
\Delta u+\omega^2u=f \quad \text{in } \Omega,
\qquad
u=0 \quad \text{on } \partial\Omega,
\end{equation}
with $\Omega=(-1,1)^2$ and $\omega=12$. The source term is determined from
\begin{equation}\label{eq:Helmholtz_solution}
u(x,y)
=
4e^{-100(x^2+y^2)}
\sin(\omega\pi x)\sin(\omega\pi y),
\end{equation}
which exhibits localized high-frequency oscillations near the origin.


\begin{figure}[htbp]
\centering
\subfloat[Points distribution, level 1]{\includegraphics[width=0.32\textwidth]{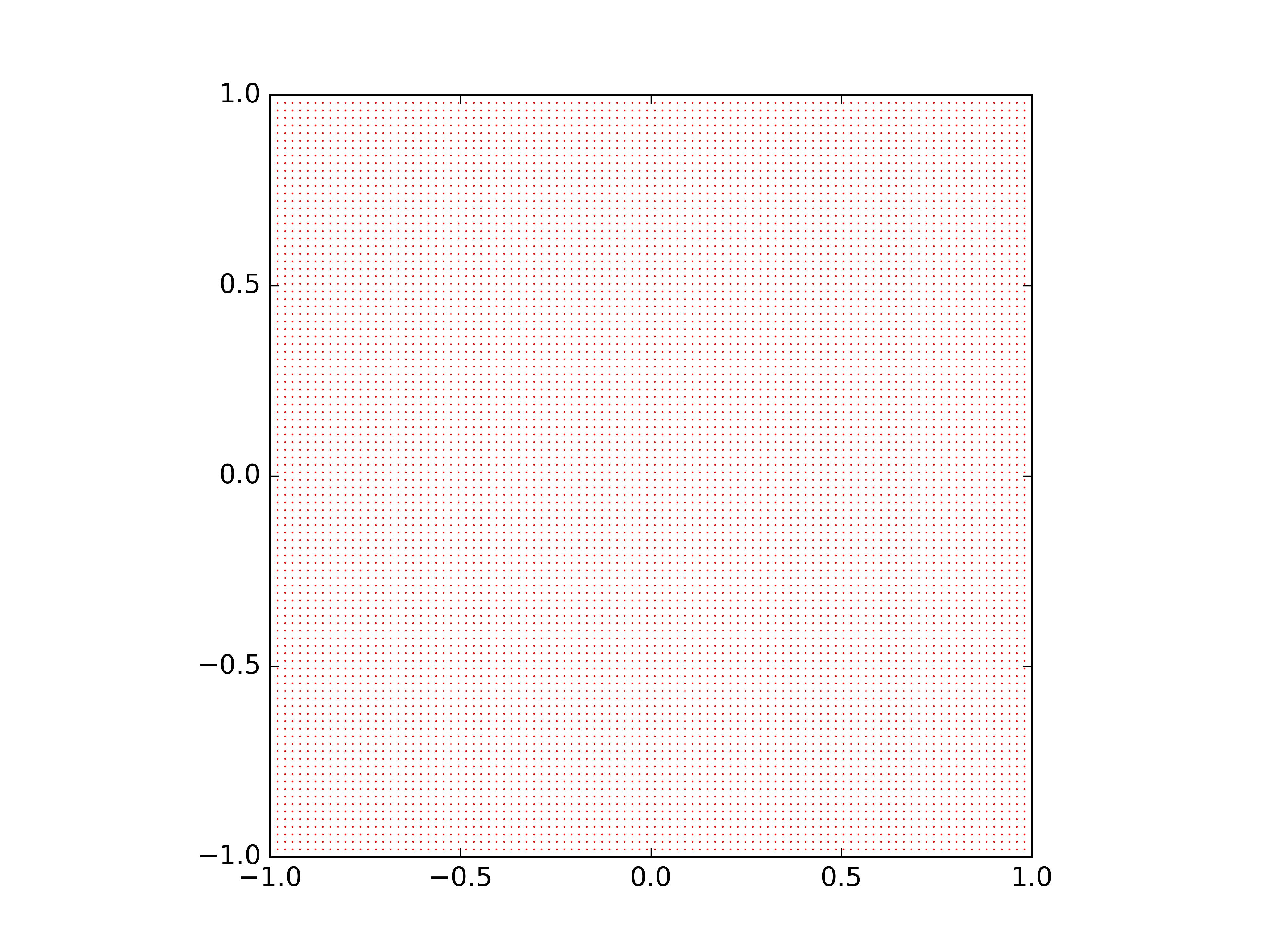}}\hfill
\subfloat[Points distribution, level 2]{\includegraphics[width=0.32\textwidth]{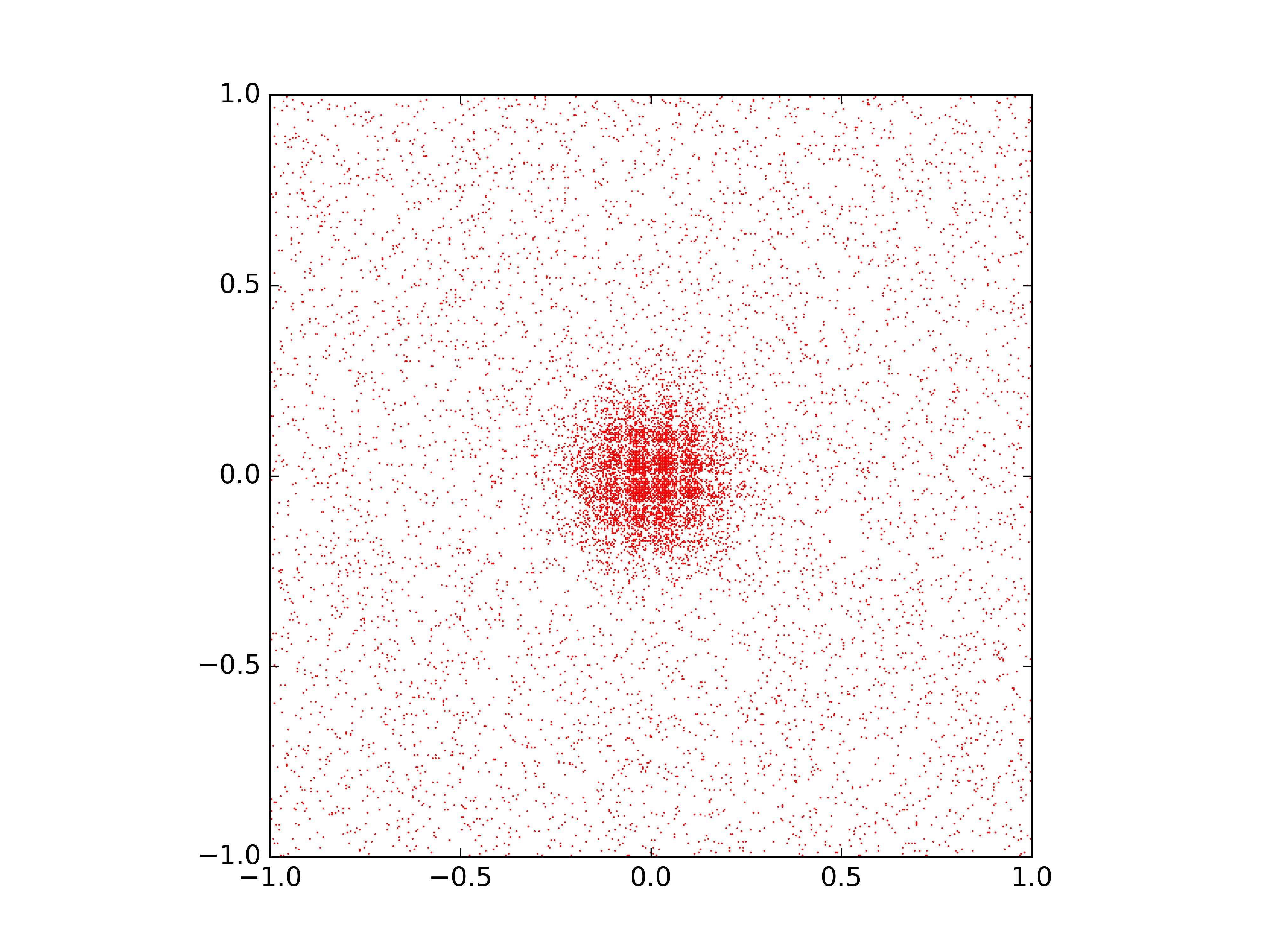}}\hfill
\subfloat[Points distribution, level 3]{\includegraphics[width=0.32\textwidth]{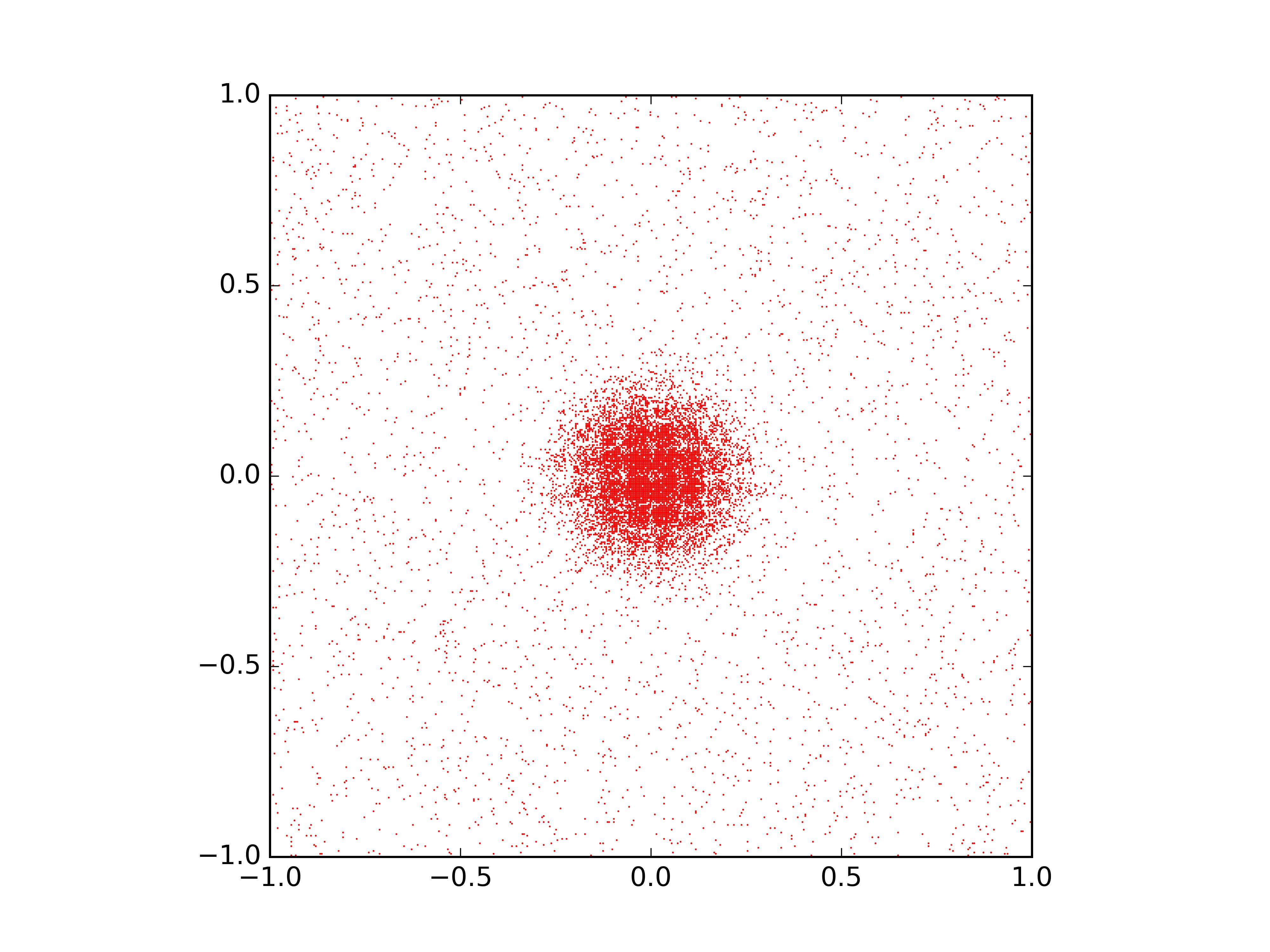}}\\
\subfloat[Error, level 1]{\includegraphics[width=0.32\textwidth]{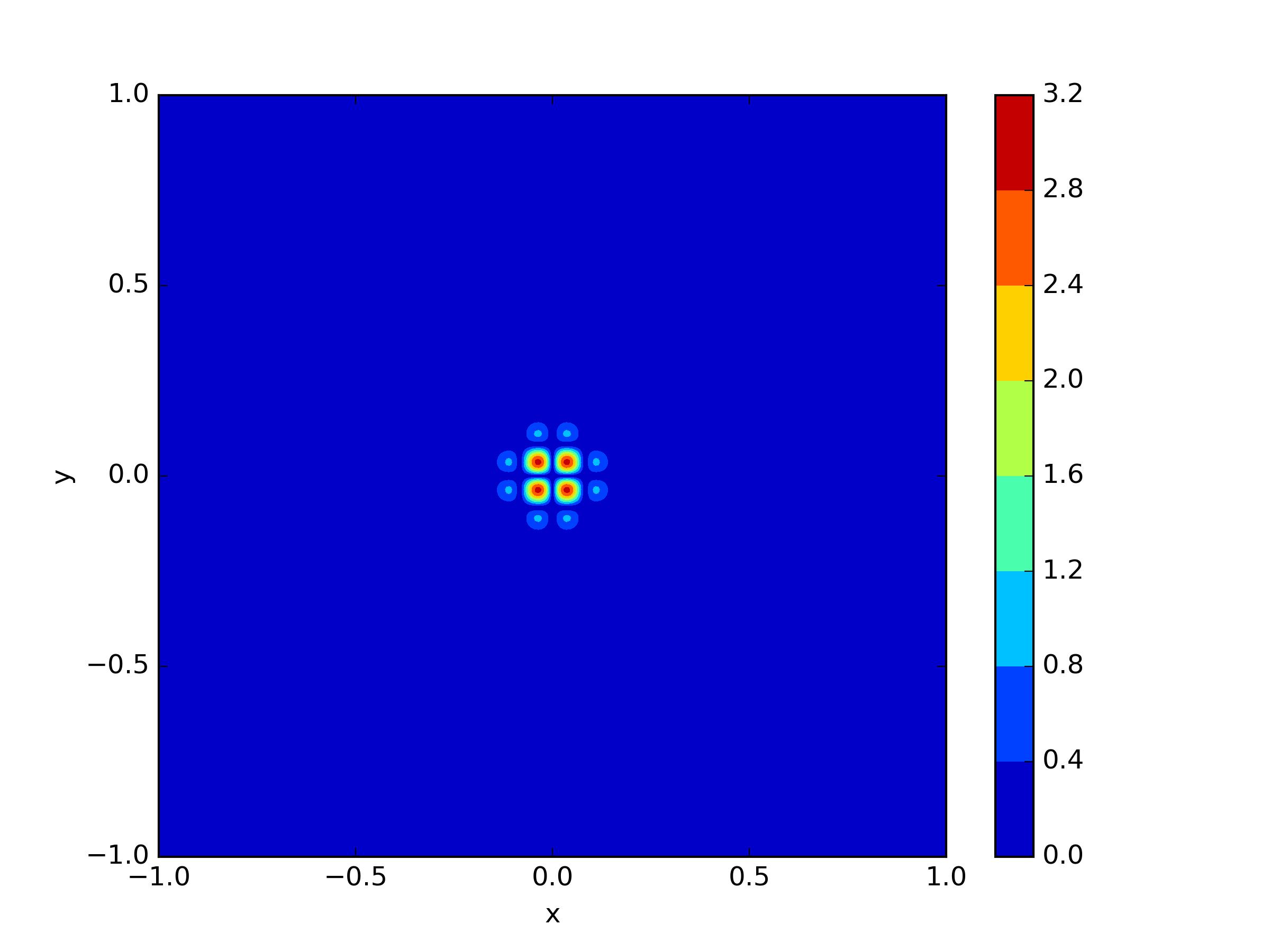}}\hfill
\subfloat[Error, level 2]{\includegraphics[width=0.32\textwidth]{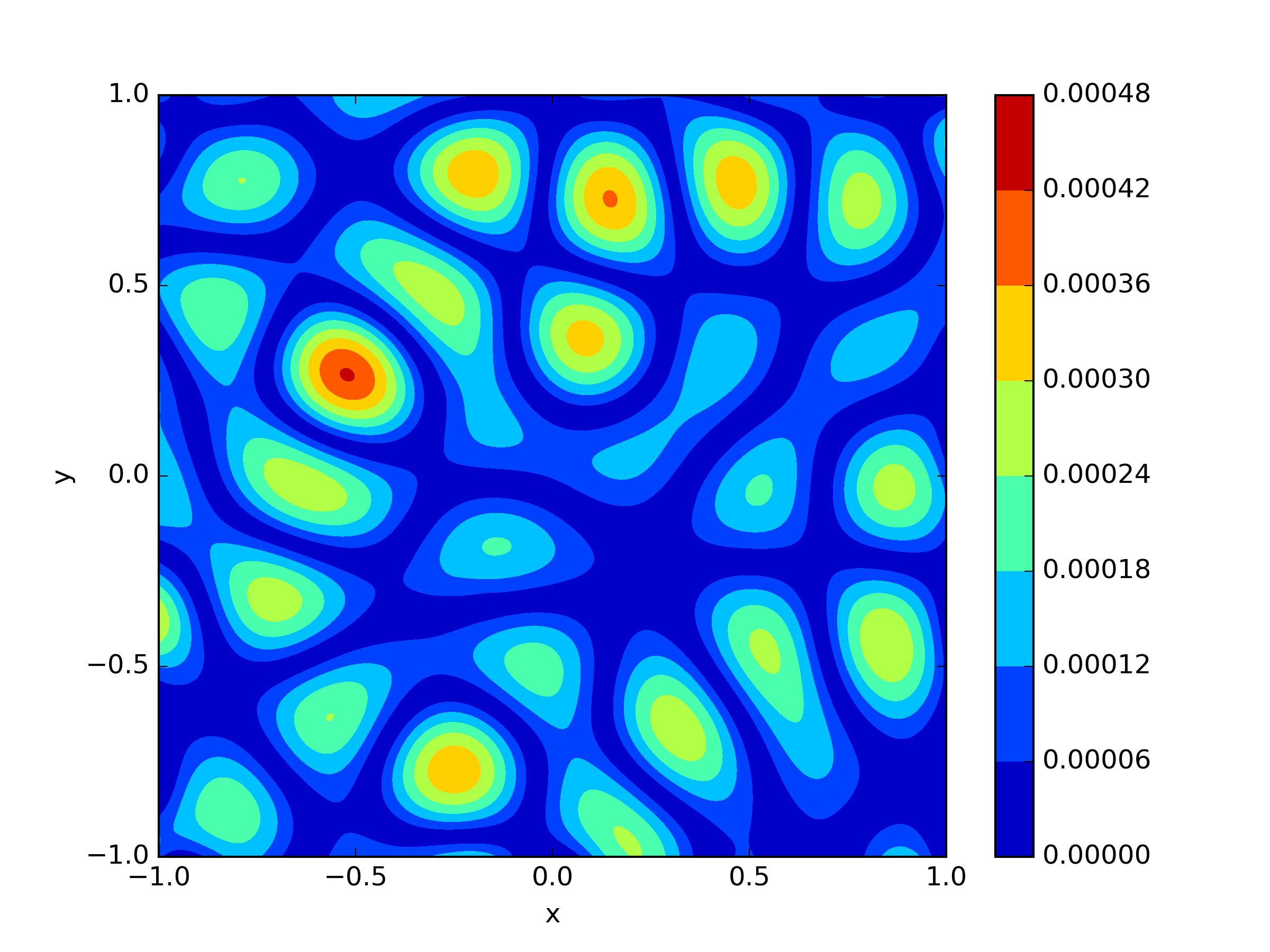}}\hfill
\subfloat[Error, level 3]{\includegraphics[width=0.32\textwidth]{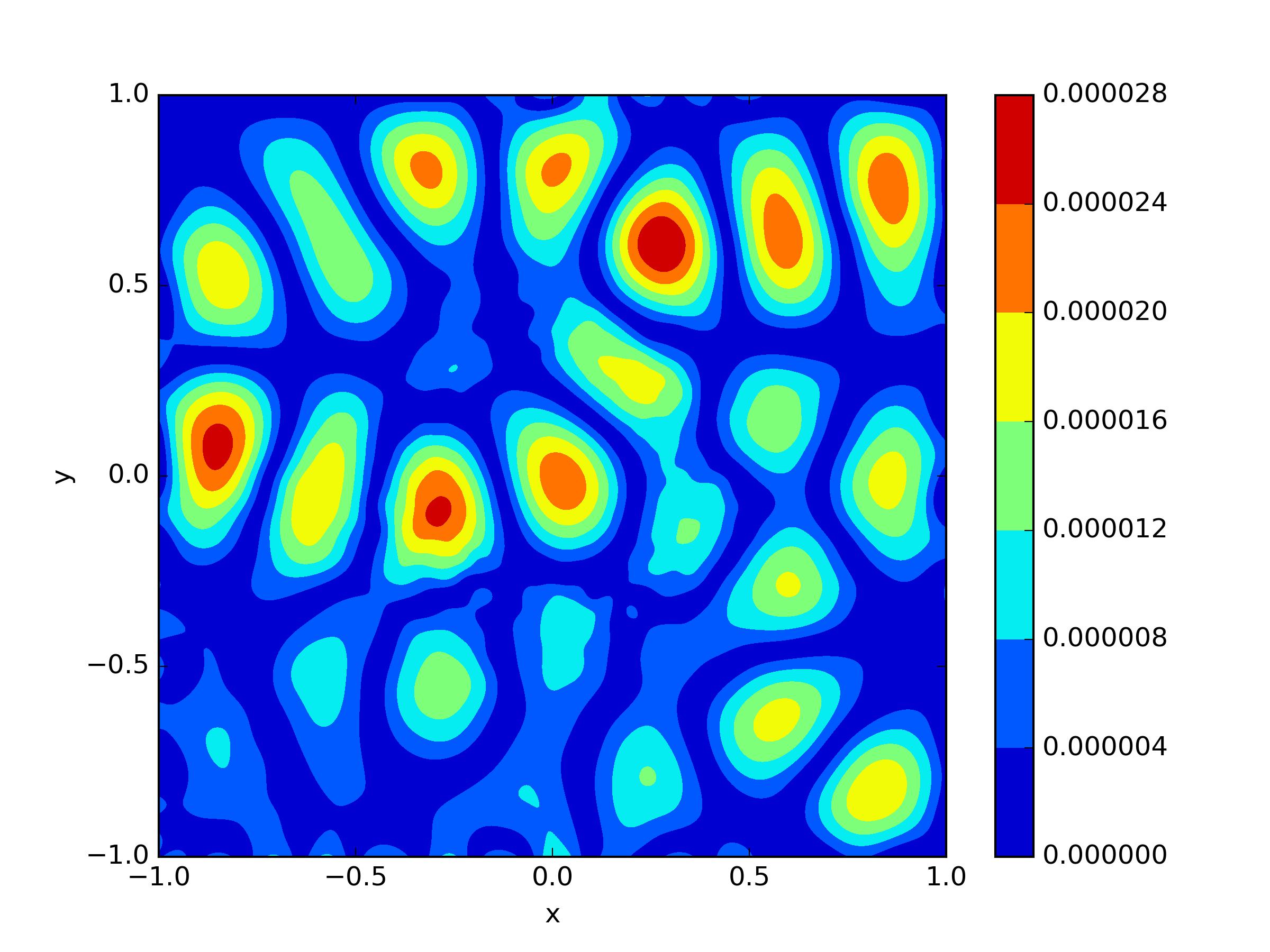}}
\caption{Helmholtz problem at levels 1--3 for $\omega = 12$: collocation-point distributions (top) and pointwise absolute errors (bottom).}
\label{fig:Helmholtz_points}
\label{fig:Helmholtz_results}
\end{figure}

Figure~\ref{fig:Helmholtz_points} shows that MLS adaptively concentrates collocation points near the origin, where the oscillatory solution is localized. As the level increases, the approximation is progressively refined and the pointwise absolute error is reduced. 
To assess the robustness of the proposed method against random initialization, we perform four additional runs with seeds 15, 61, 74, 77 and 100. Figure~\ref{fig:Helmholtz_Errork} reports the mean error histories over five independent runs, and Table~\ref{tab:Helmholtz_seed_Results} summarizes the seed-wise results, demonstrating consistent convergence behavior across different initializations.

\begin{figure}[htbp]
\centering
\includegraphics[width=0.55\textwidth]{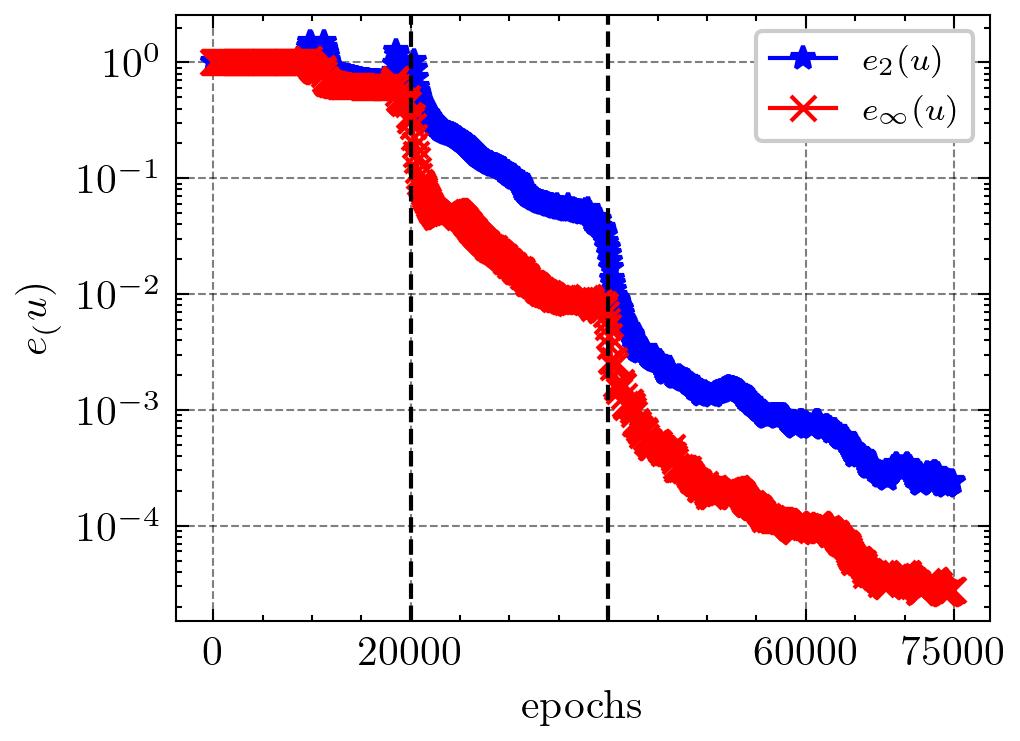}
\caption{Helmholtz problem: mean errors over five random seeds versus training iterations.}
\label{fig:Helmholtz_Errork}
\end{figure}

\begin{table}[htbp]
\scriptsize
\centering
\caption{Helmholtz relative errors for five random seeds}
\setlength{\tabcolsep}{3.mm}{
\begin{tabular}{|c|c|c|c|c|c|}
\hline
Seed    &  15 & 61 & 74  & 77&100\\
\hline
$e_\infty(u)$  & $ 1.785 \times 10^{-5}$  & $2.244 \times 10^{-5}$ & $2.220 \times 10^{-5}$ &  $1.590\times 10^{-5}$ & $9.469 \times 10^{-6}$\\

\hline
$e_2(u)$  & $1.561 \times 10^{-4}$  & $7.780 \times 10^{-5}$ & $1.889\times 10^{-4}$  & $1.439 \times 10^{-4}$ & $7.359 \times 10^{-5}$\\
\hline
\end{tabular}
}
\label{tab:Helmholtz_seed_Results}
\end{table}

\subsection{Two-dimensional linear elasticity}
\label{sec:linearelasticity_2D}
Let $\Omega$ be the unit disk with the sector corresponding to
$\theta\in[0,\pi/4]$ removed. We consider the linear elasticity problem
\begin{equation}\label{eq:linearelasticity_2D}
\begin{cases}
\nabla\cdot\boldsymbol{\sigma}=\mathbf f,
    & \text{in } \Omega,\\
\mathbf u=\mathbf u_0,
    & \text{on } \partial\Omega,
\end{cases}
\qquad
\boldsymbol{\sigma}
=2\mu\boldsymbol{\varepsilon}
+\lambda\operatorname{tr}(\boldsymbol{\varepsilon})\mathbf I,
\qquad
\boldsymbol{\varepsilon}
=\frac12\left(\nabla\mathbf u+\nabla\mathbf u^\top\right).
\end{equation}
The exact displacement is prescribed as
\begin{equation}\label{eq:2d_linearelasticitysolution}
\mathbf u(x,y)
=
\begin{pmatrix}
\sin x\cos y+\frac12x\\[1mm]
\cos x\sin y-\frac{\nu}{2}y
\end{pmatrix},
\end{equation}
from which $\mathbf u_0$ and $\mathbf f$ are determined. In particular,
\begin{equation}\label{eq:2d_linearelasticitysource}
\mathbf f
=
-2(\lambda+2\mu)
\begin{pmatrix}
\sin x\cos y\\
\cos x\sin y
\end{pmatrix}.
\end{equation}
The corresponding stress components are
\begin{equation}\label{eq:2d_linearelasticitystress}
\begin{aligned}
\sigma_{xx}
&=2(\lambda+\mu)\cos x\cos y
+\frac12\bigl[2\mu+(1-\nu)\lambda\bigr],\\
\sigma_{yy}
&=2(\lambda+\mu)\cos x\cos y,\\
\sigma_{xy}
&=\sigma_{yx}=-2\mu\sin x\sin y.
\end{aligned}
\end{equation}

In this experiment, we consider the case of Poisson ratio $\nu$ = 0.4995, that is, $\mu$ = 1, $\lambda$ = 1000. Figure~\ref{fig:2dlinearelasticity_points} illustrates how the collocation points adapt to the solution features and how the pointwise absolute errors decrease across successive levels.

\begin{figure}[htbp]
\centering
\subfloat[Points distribution, level 1]
{\includegraphics[width=0.32\textwidth]{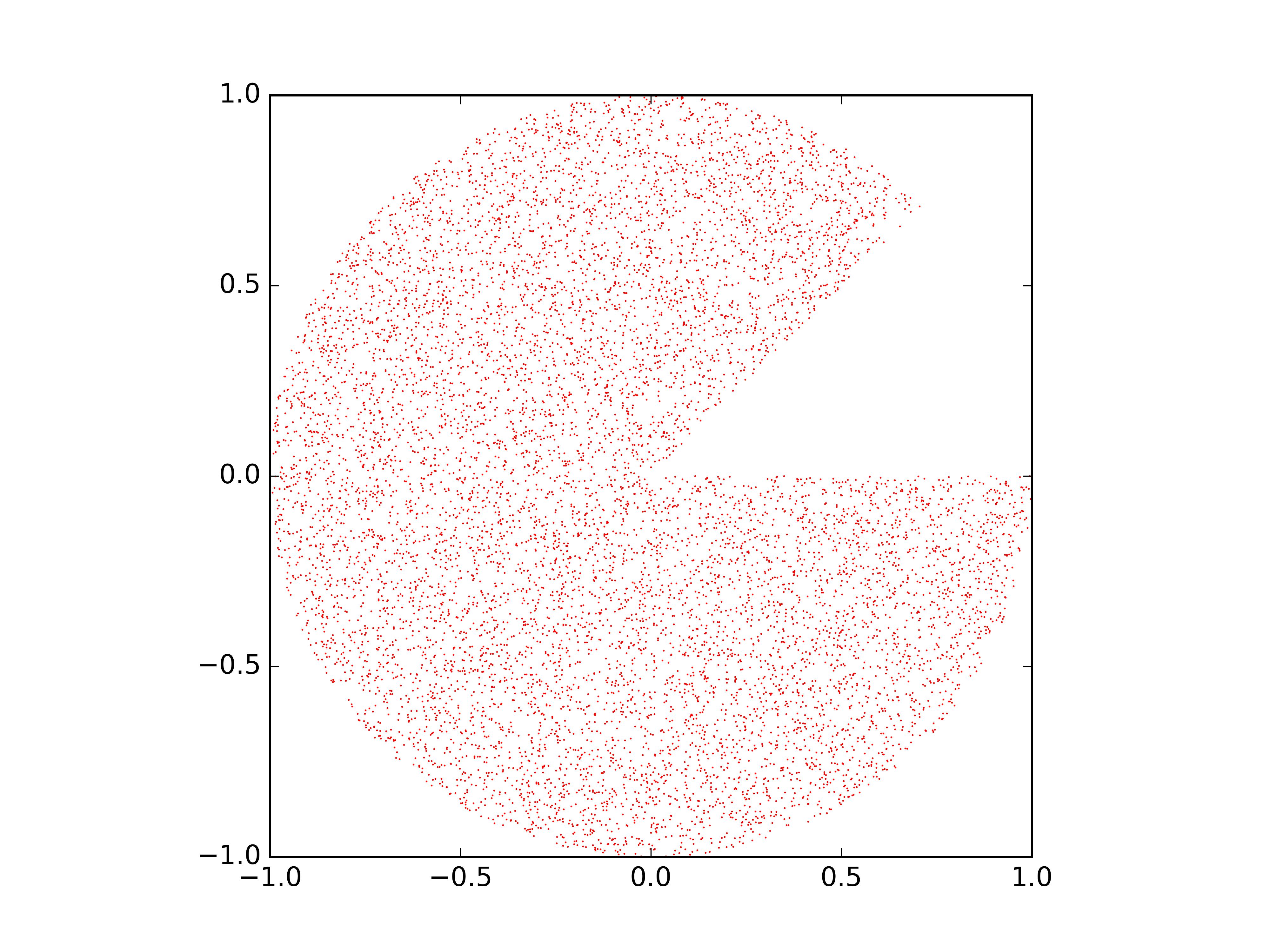}} \hfill
\subfloat[Points distribution, level 2]
{\includegraphics[width=0.32\textwidth]{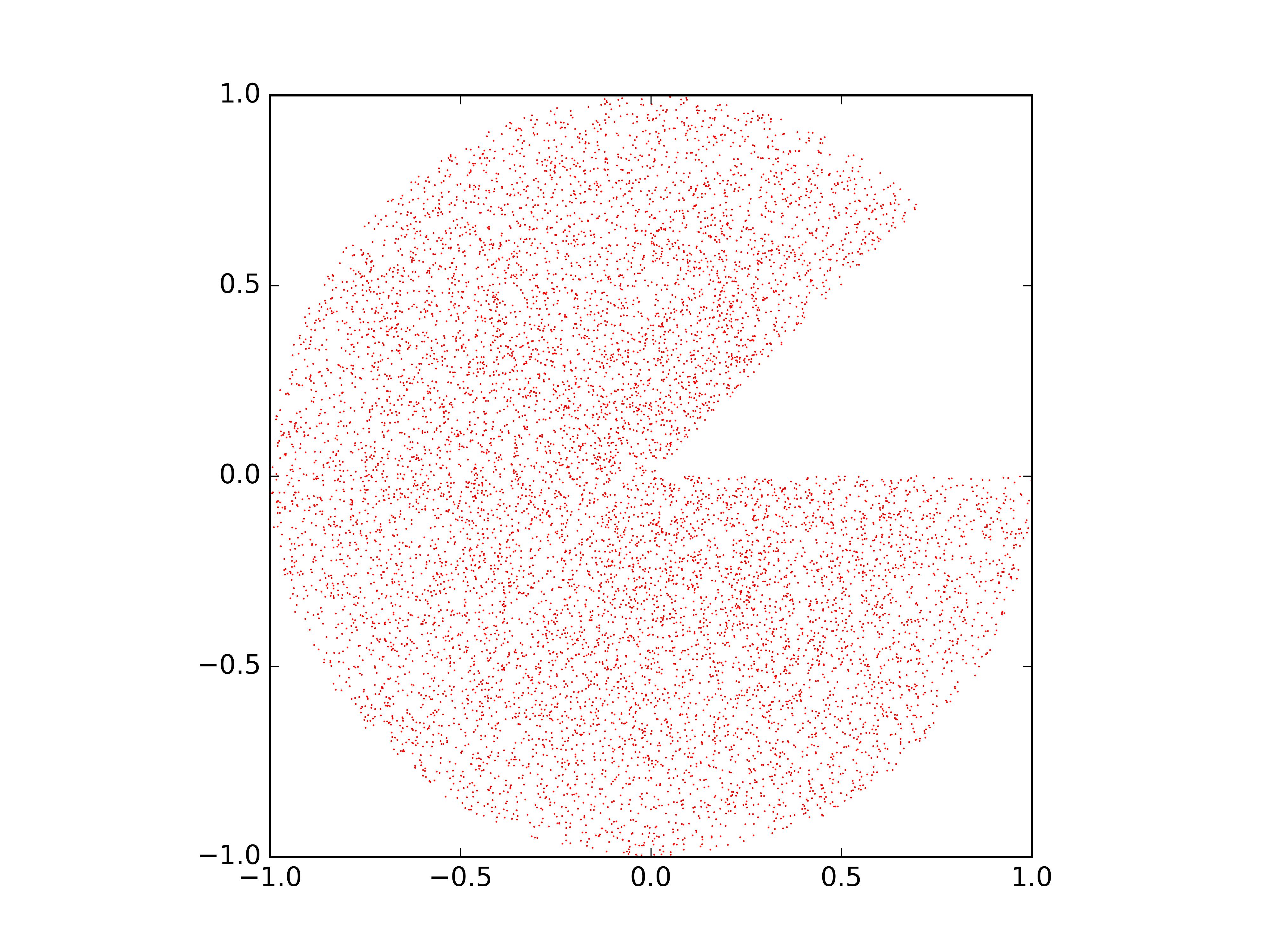}} \hfill
\subfloat[Points distribution, level 3]
{\includegraphics[width=0.32\textwidth]{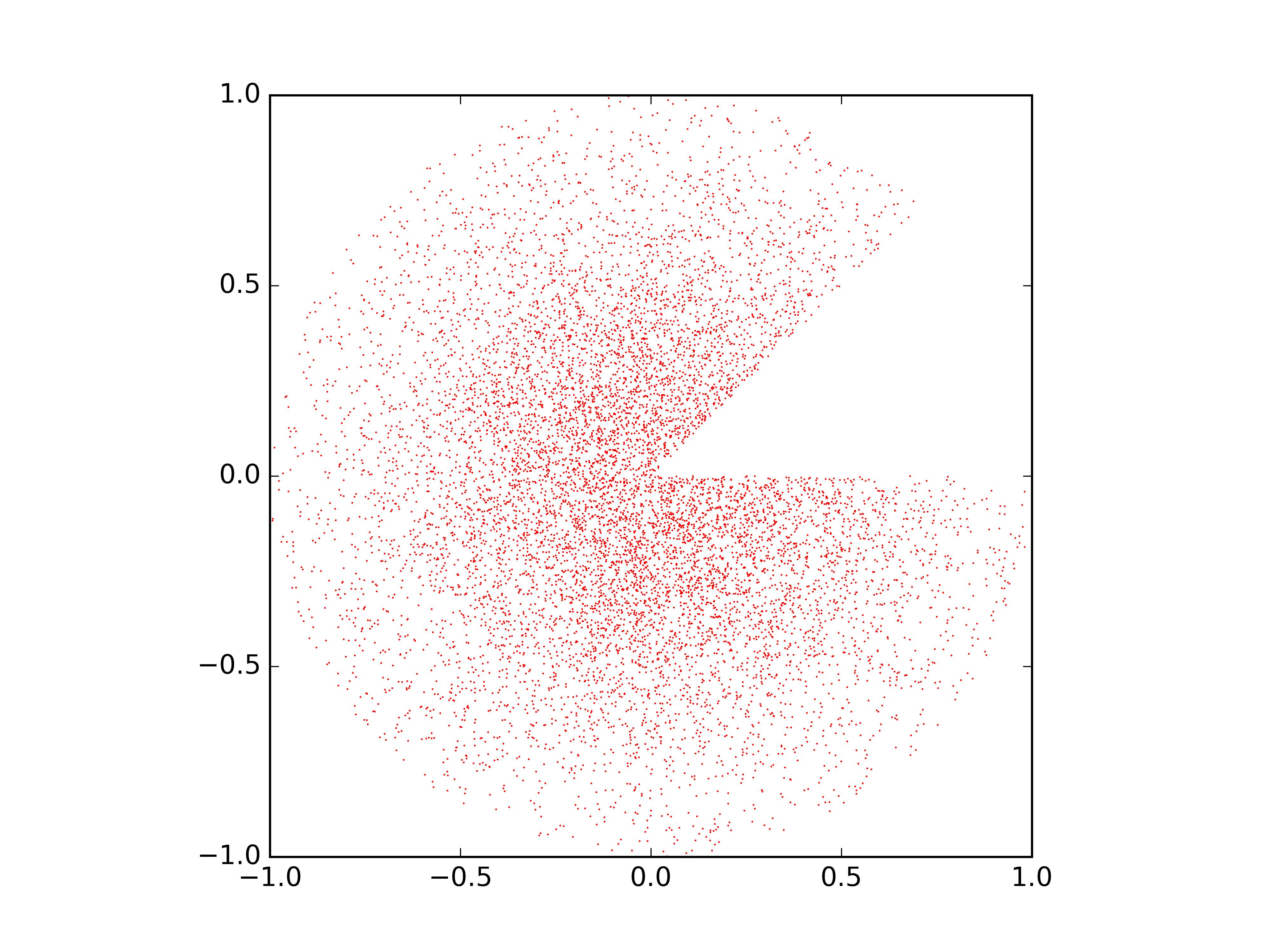}}\\
\subfloat[Displacement error, level 1]
{\includegraphics[width=0.32\textwidth]{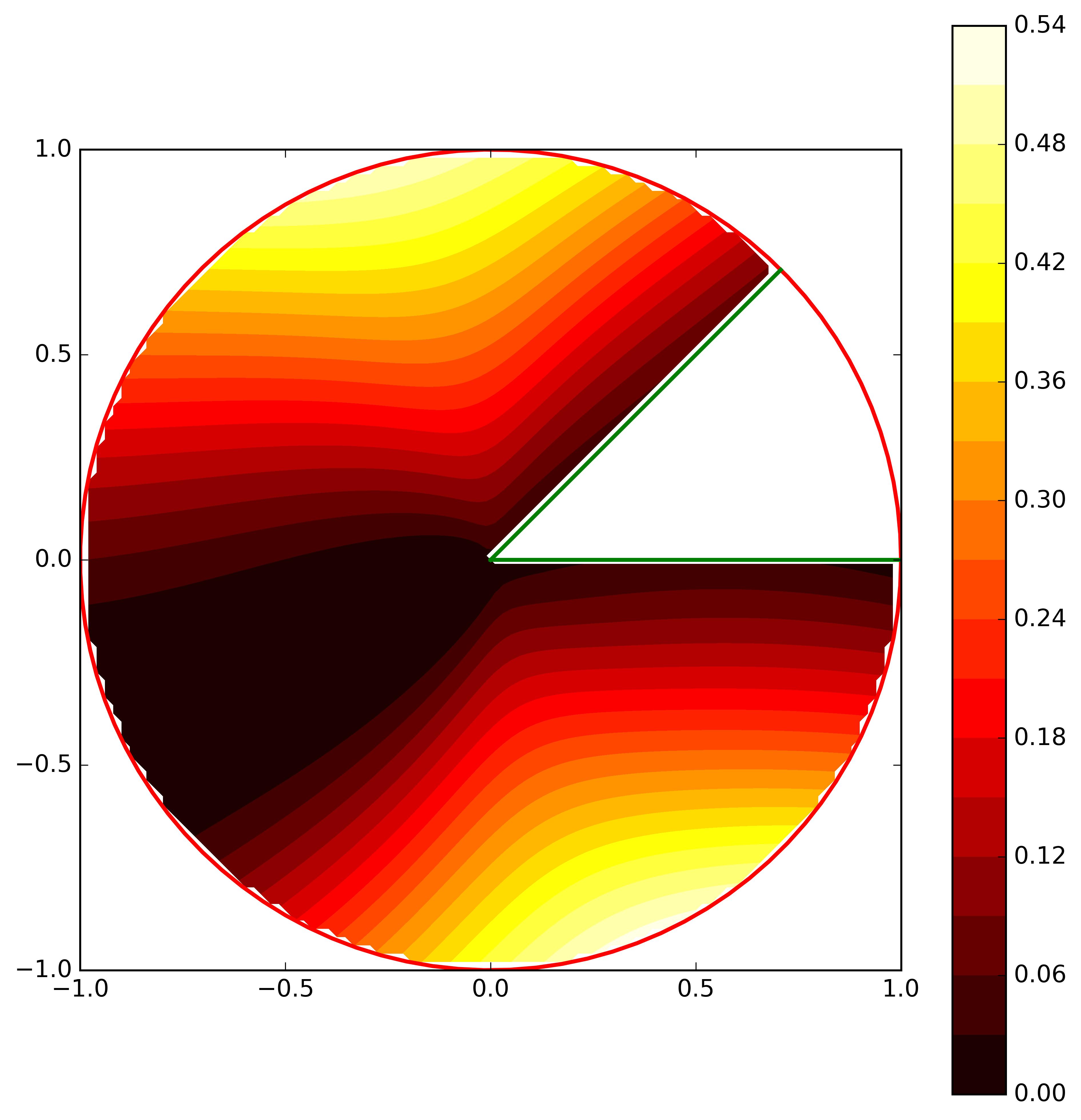}} \hfill
\subfloat[Displacement error, level 2]
{\includegraphics[width=0.32\textwidth]{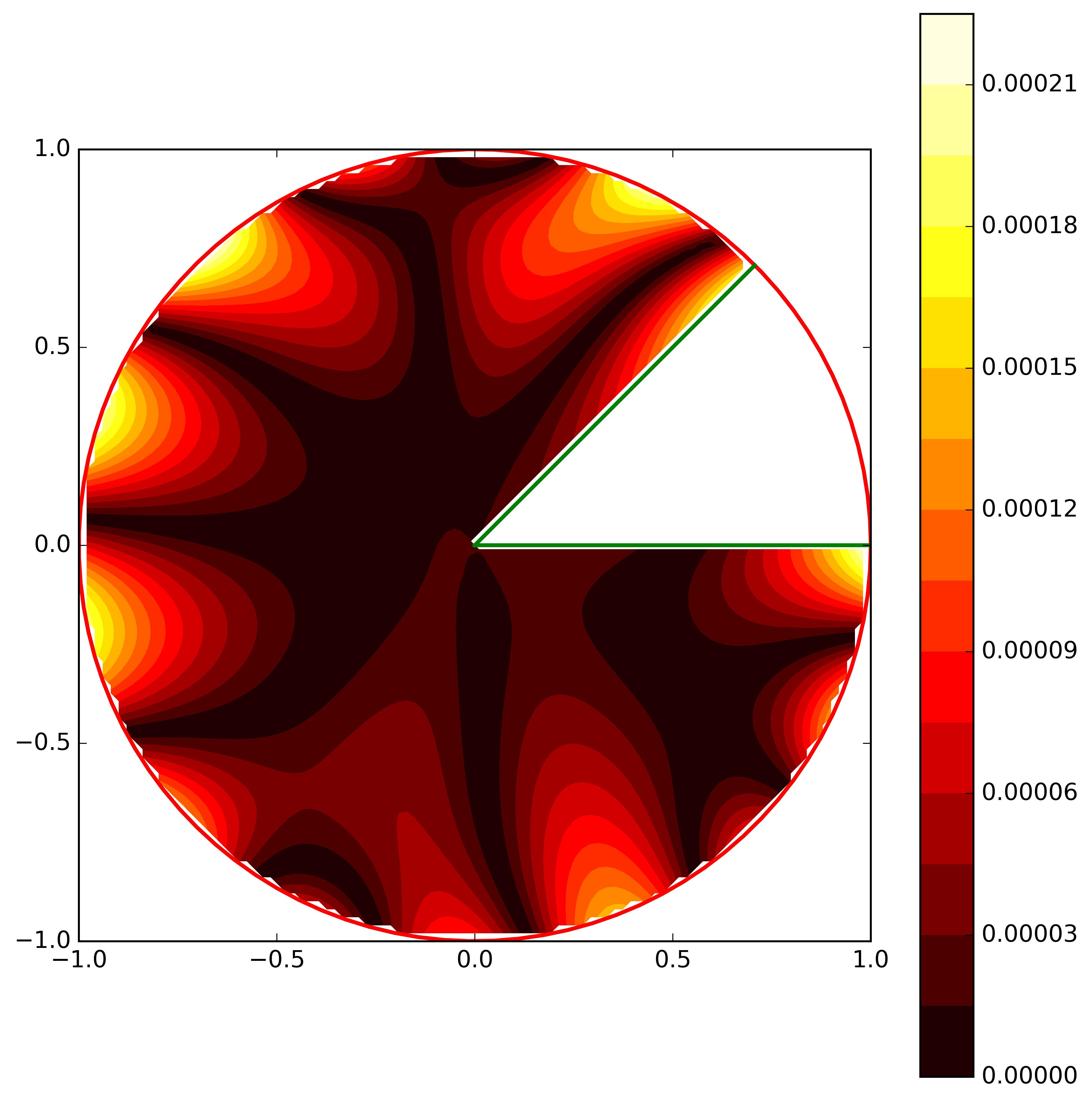}} \hfill
\subfloat[Displacement error, level 3]
{\includegraphics[width=0.32\textwidth]{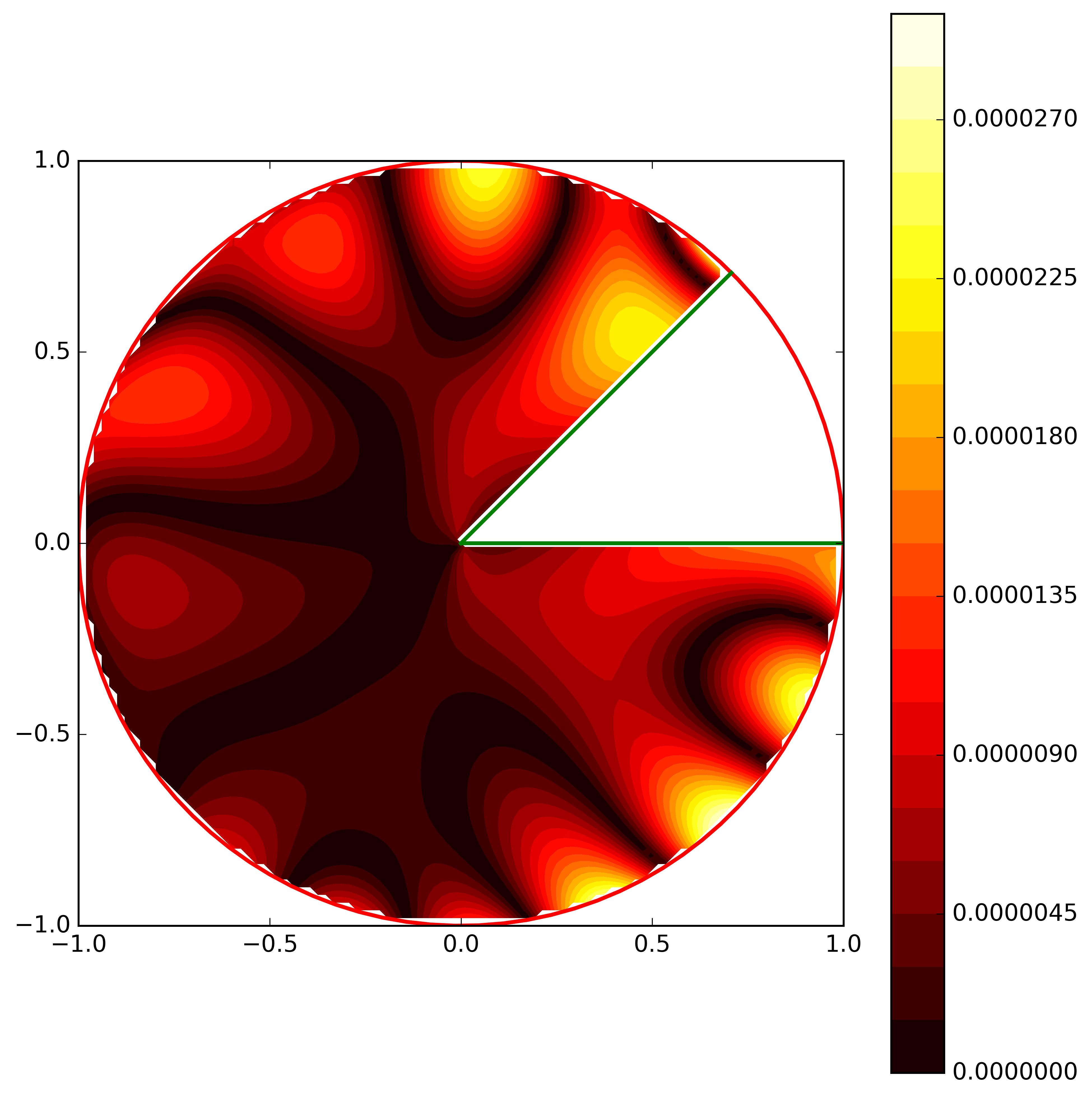}}\\
\subfloat[$\sigma_{xx}$ error, level 1]
{\includegraphics[width=0.32\textwidth]{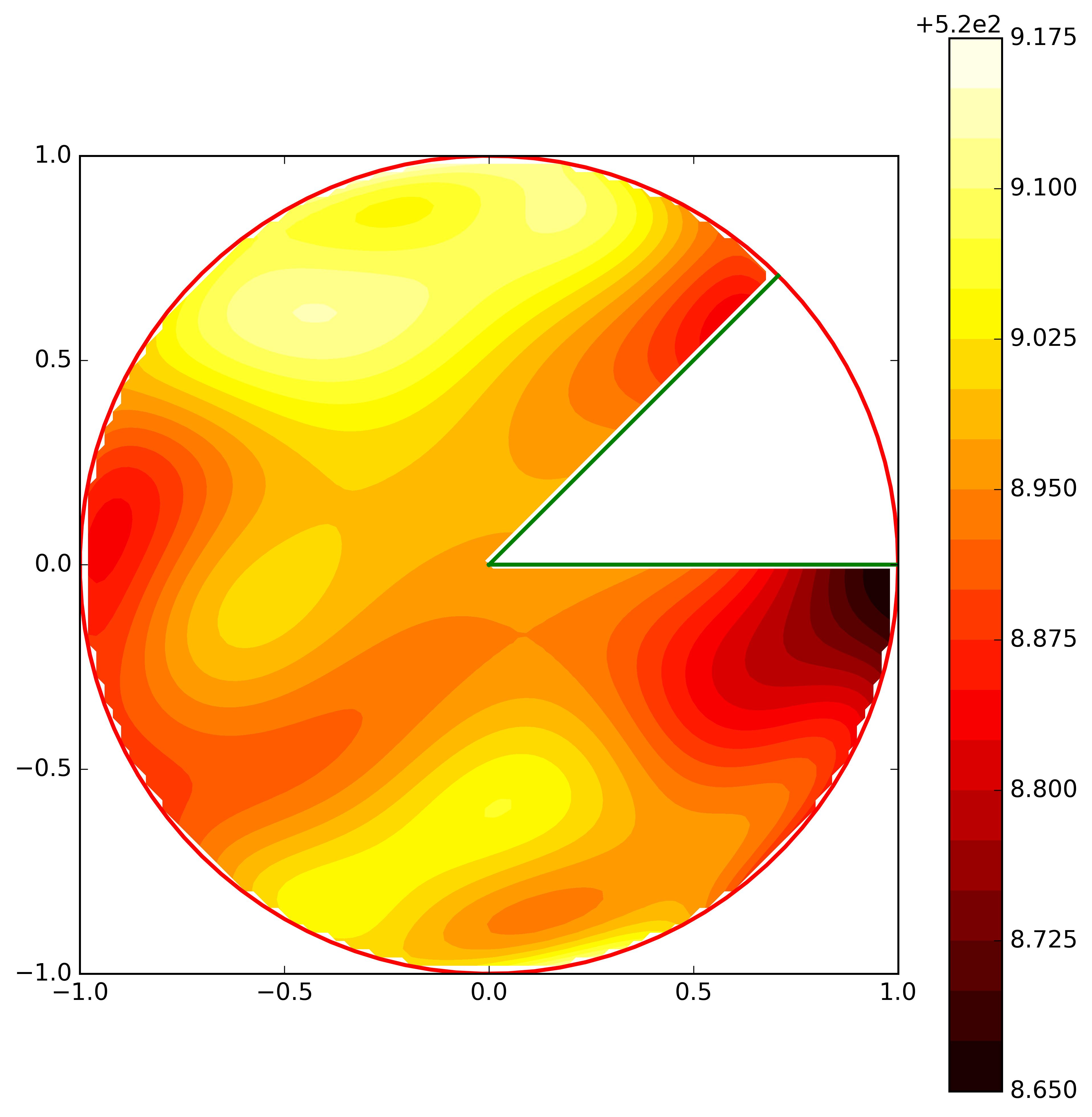}} \hfill
\subfloat[$\sigma_{xx}$ error, level 2]
{\includegraphics[width=0.32\textwidth]{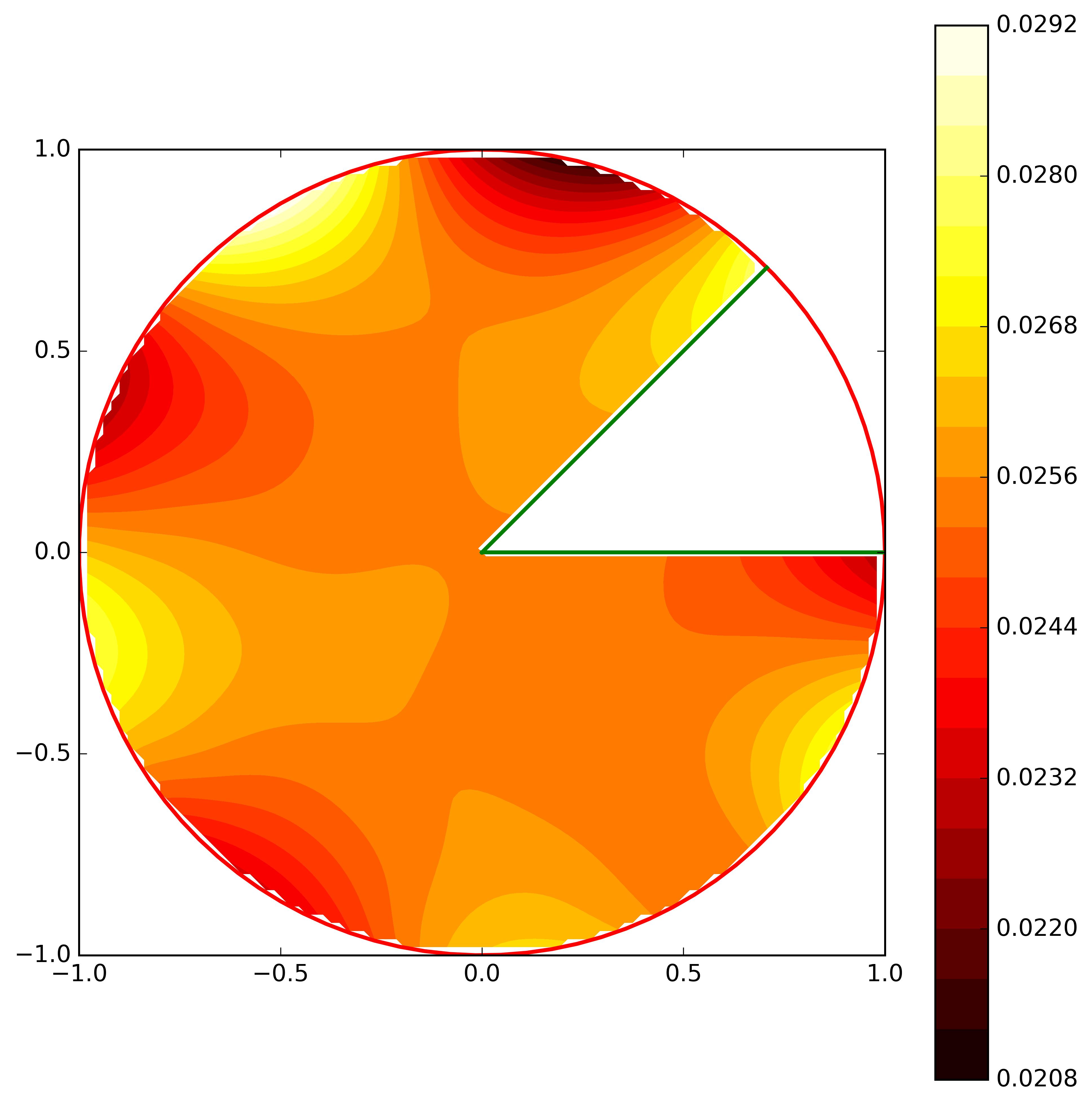}} \hfill
\subfloat[$\sigma_{xx}$ error, level 3]
{\includegraphics[width=0.32\textwidth]{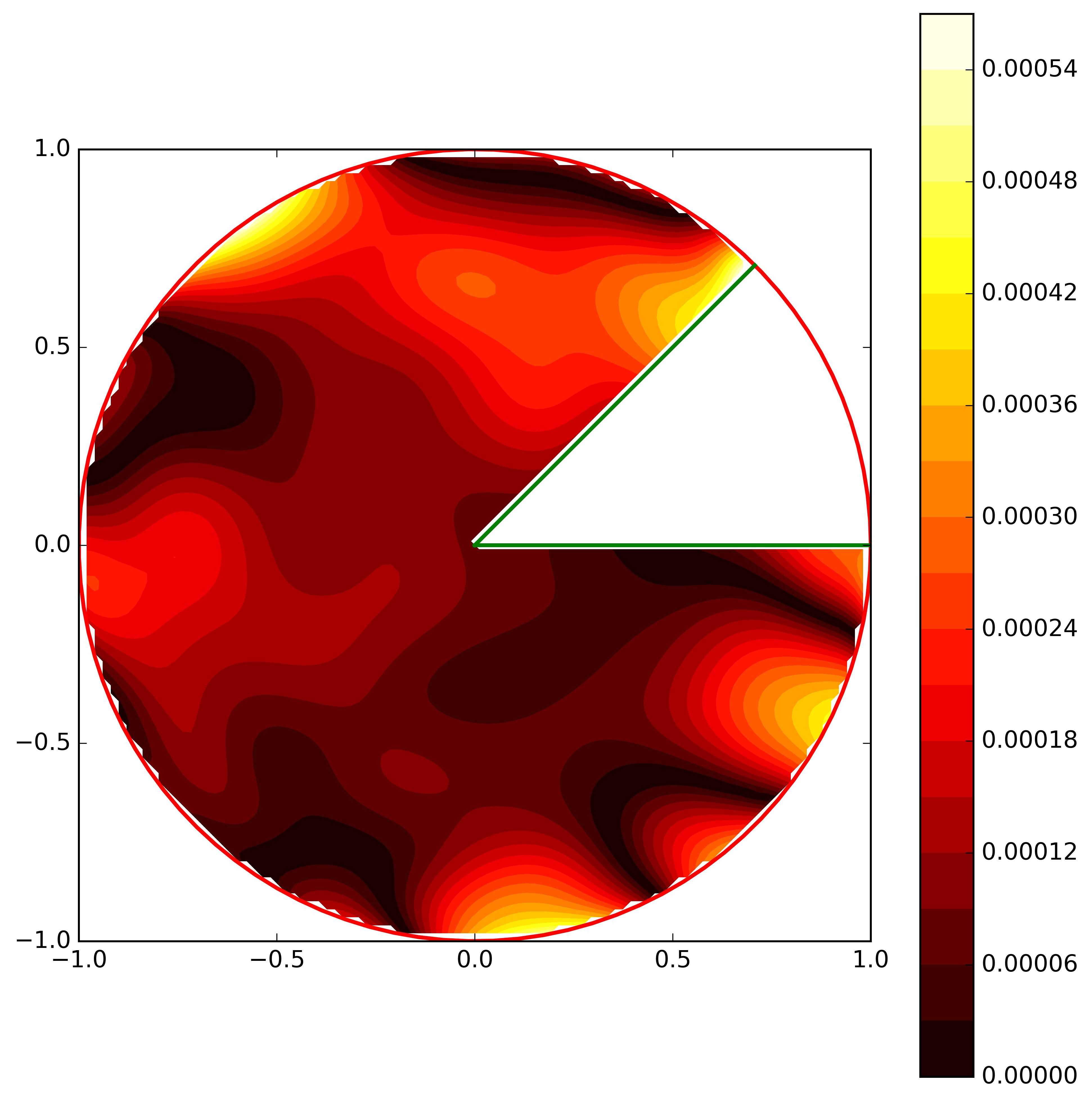}}
\caption{Linear elasticity: sampling points at levels 1--3 (top); absolute error of the predicted displacement field at levels 1--3 (middle); absolute error of stress $\sigma_{xx}$ at levels 1--3 (bottom).}
\label{fig:2dlinearelasticity_points}
\end{figure}

\subsection{Viscous Prandtl equation}
\label{sec:ViscousPrandtl_2D}
We consider the two-dimensional unsteady viscous Prandtl system with zero
pressure gradient,
\begin{equation}\label{eq:ViscousPrandtl_2D}
\begin{aligned}
u_t+u u_x+v u_y&=\nu u_{yy},
& u_x+v_y&=0,
&&0\le x\le1,\quad y\ge0,\\
u(x,0,t)&=v(x,0,t)=0,
& u_y(x,y,t)&\to1
&&\text{as }y\to\infty,\\
u(x,y,0)&=y+\sin(2\pi x),
\end{aligned}
\end{equation}
with $\nu=0.1$. Following \cite{hong2003singularity}, the problem is
computed on the truncated space--time domain
\[
Q=[0,1]\times[0,2]\times[0,0.16].
\]
This example is characterized by the rapid growth of $v$ near the nominal
breaking time $t_b=(2\pi)^{-1}\approx0.159$. Multilevel correction is
applied only to $v$, while $u$ is recomputed at each level using the
baseline training procedure. We use $10^5$ interior collocation points,
$4\times10^4$ boundary points for the prescribed $x$-boundary conditions,
and $5\times10^3$ initial points. The condition $v(x,0,t)=0$ is imposed
exactly.

Figures~\ref{fig:ViscousPrandtl2D_Surface}--\ref{fig:ViscousPrandtl2D_line2} show the adaptive sampling patterns and numerical solutions across successive levels. The results agree well with those in \cite[Fig.~4.3]{hong2003singularity} and accurately resolve the sharp peak in \(v\) near the breaking time.

\begin{figure}[htbp]
\centering
\subfloat[Points distribution, level 1]
{\includegraphics[width=0.32\textwidth]{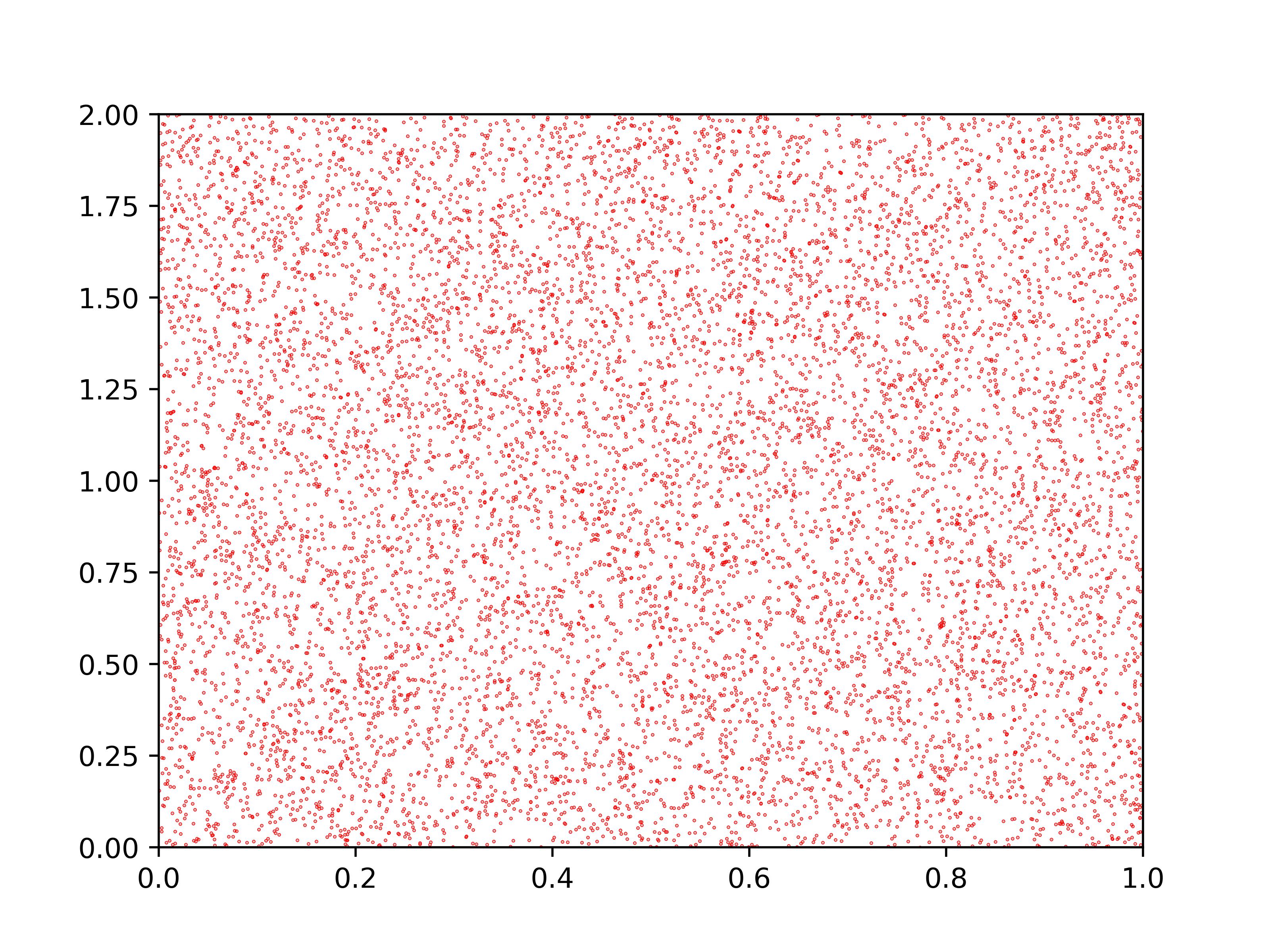}} \hfill
\subfloat[Points distribution, level 2]
{\includegraphics[width=0.32\textwidth]{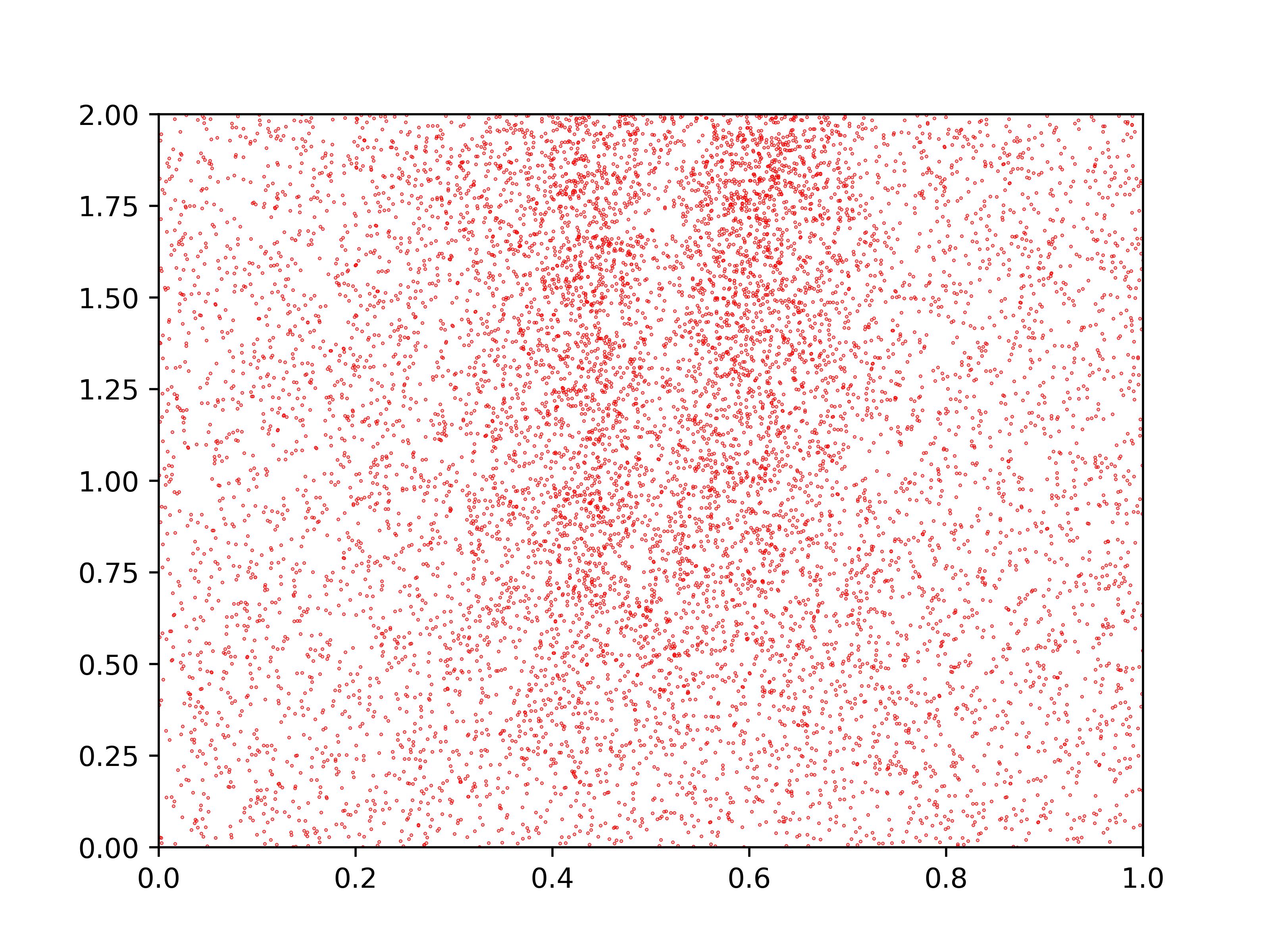}} \hfill
\subfloat[Points distribution, level 3]
{\includegraphics[width=0.32\textwidth]{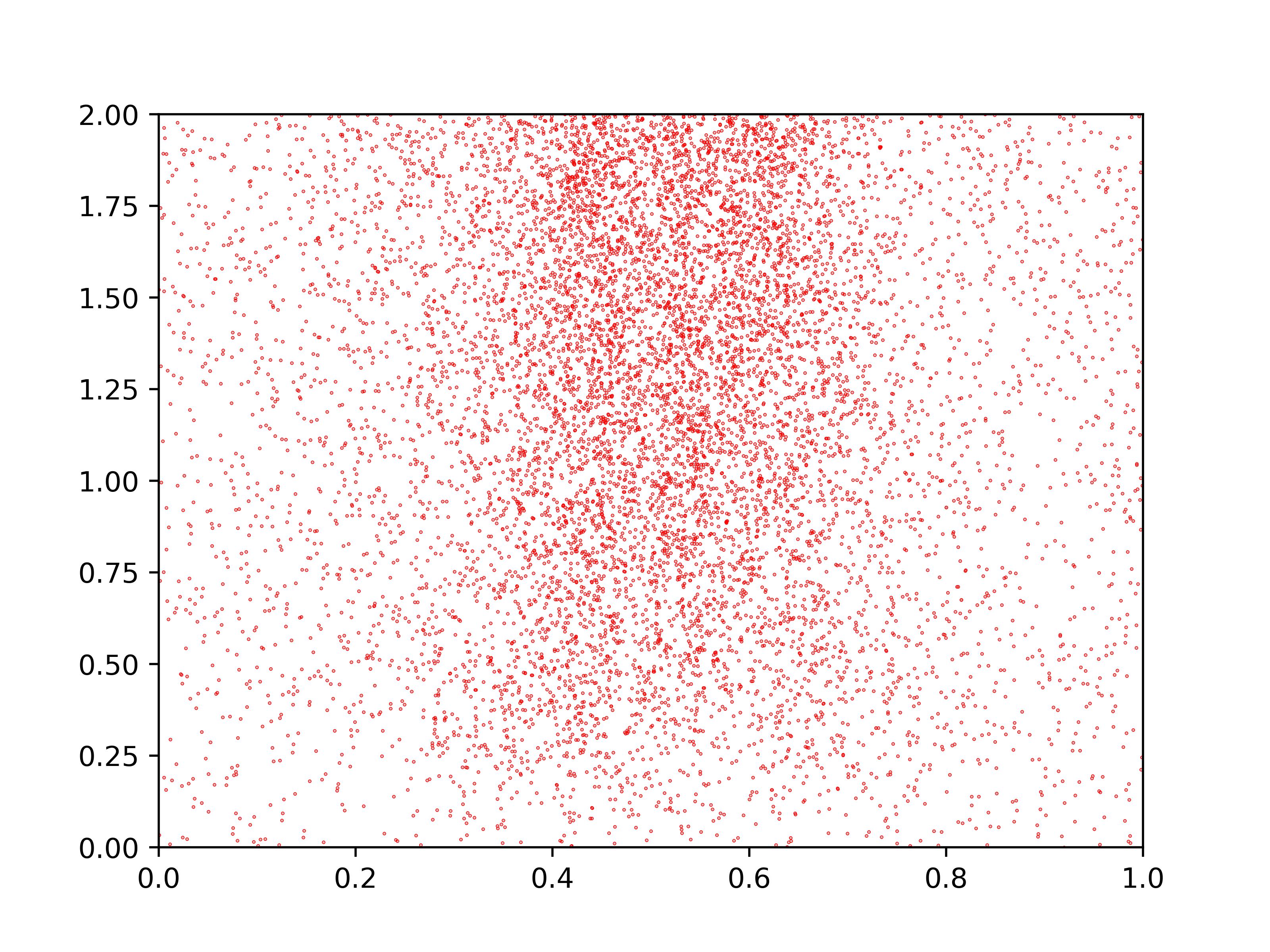}} \hfill
\subfloat[Approximation of $u$, level 1]
{\includegraphics[width=0.32\textwidth]{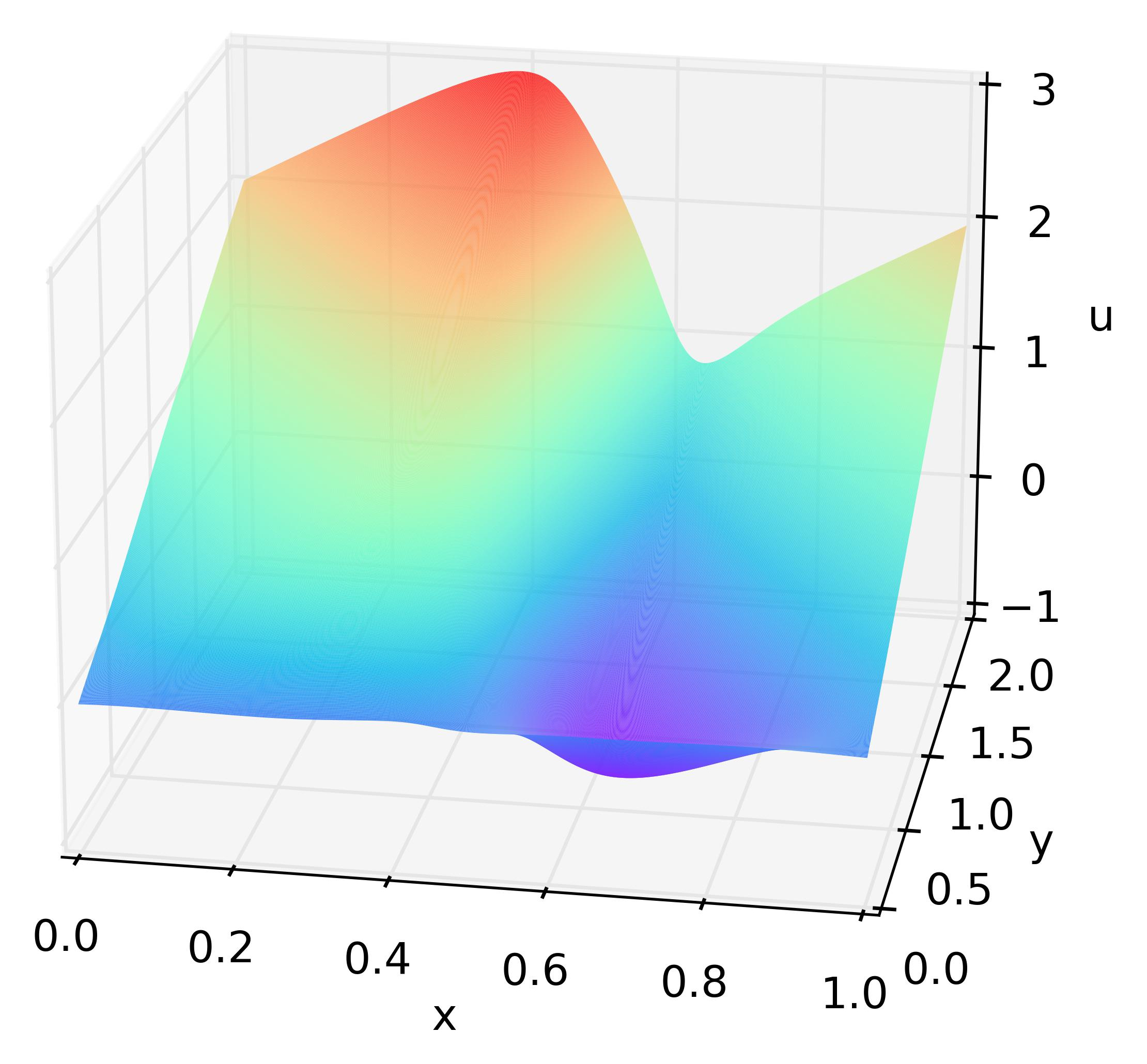}}\hfill  
\subfloat[Approximation of $u$, level 2]
{\includegraphics[width=0.32\textwidth]{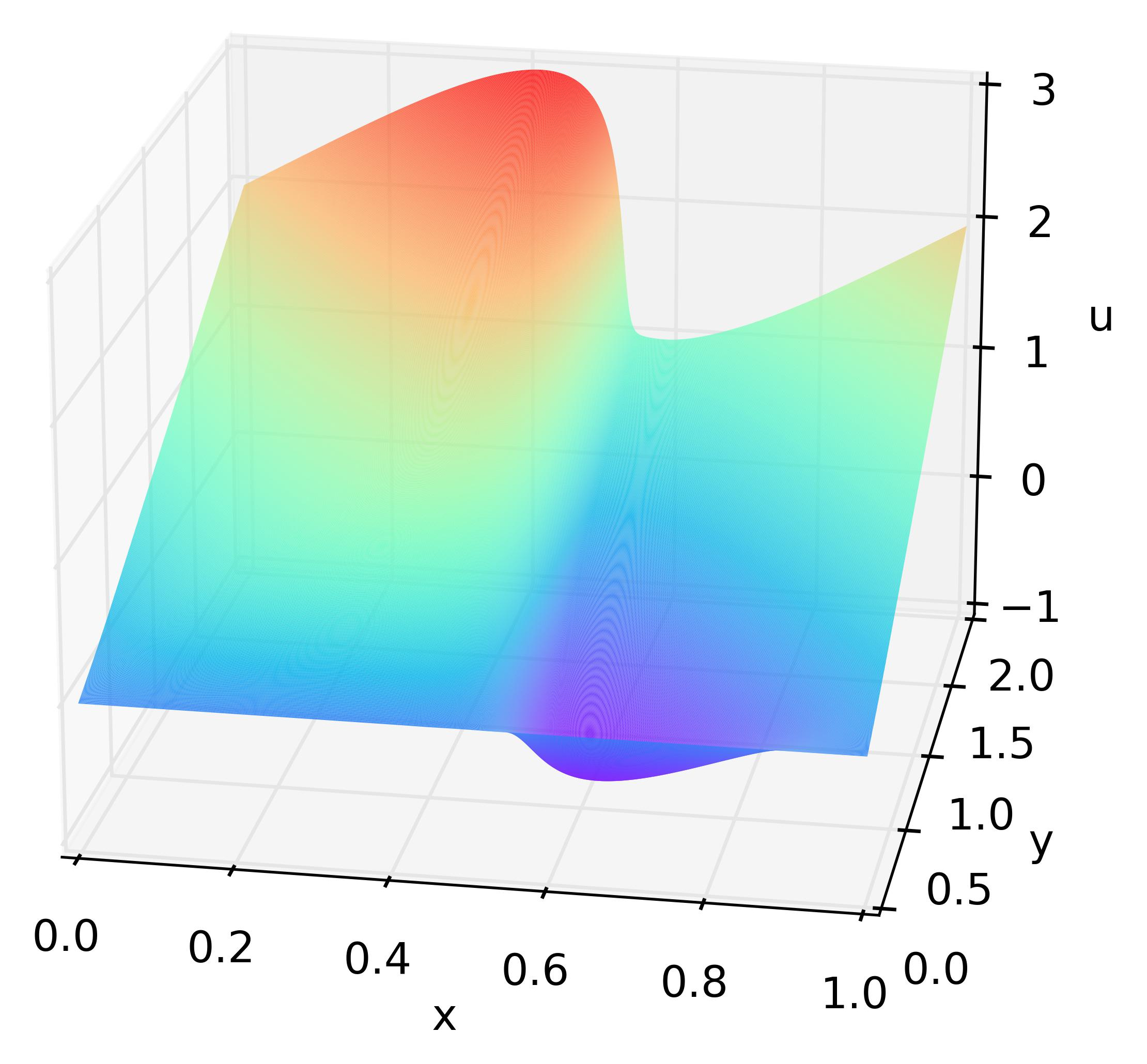}} \hfill
\subfloat[Approximation of $u$, level 3]
{\includegraphics[width=0.32\textwidth]{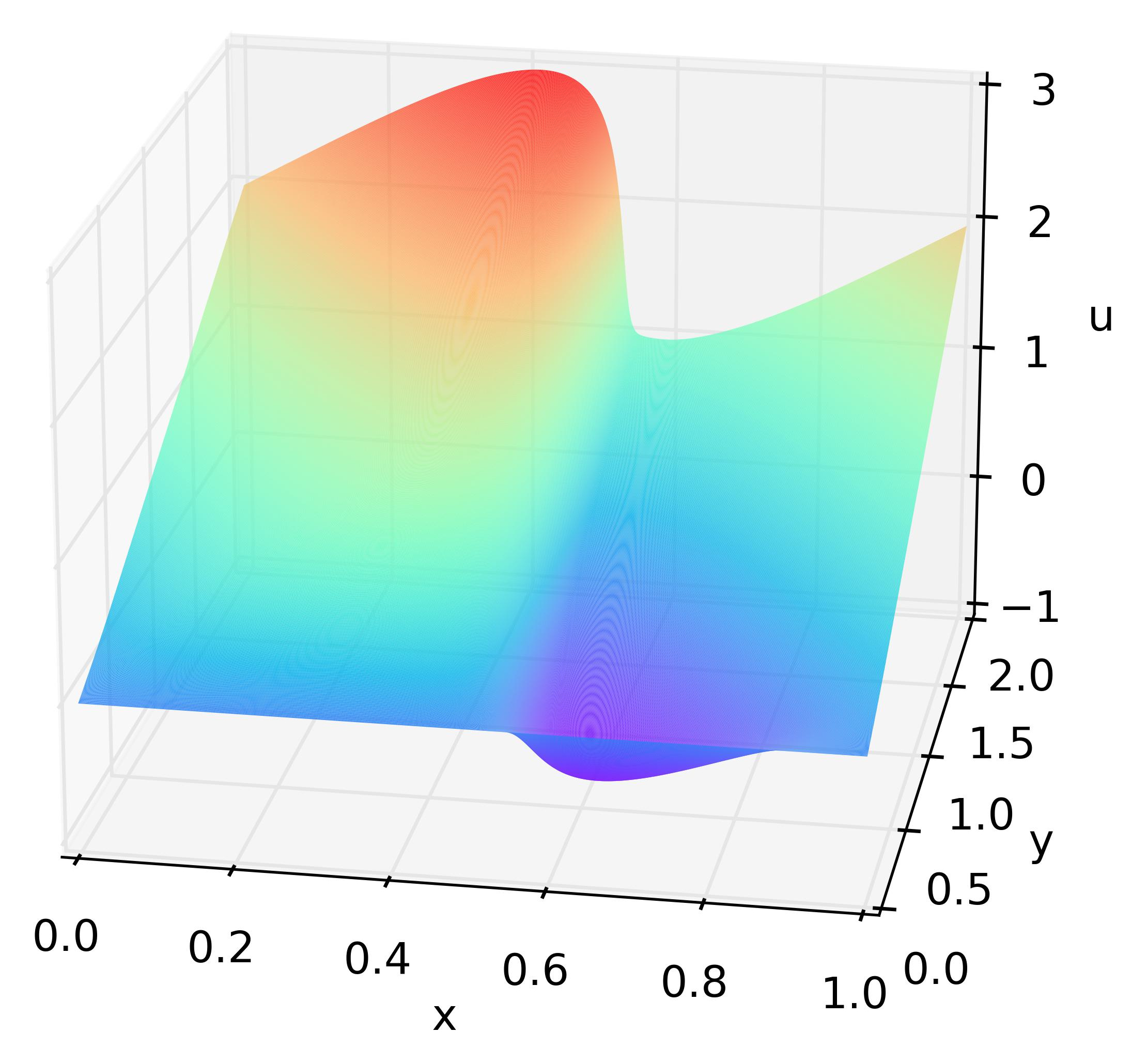}} \hfill
\subfloat[Approximation of $v$, level 1]
{\includegraphics[width=0.32\textwidth]{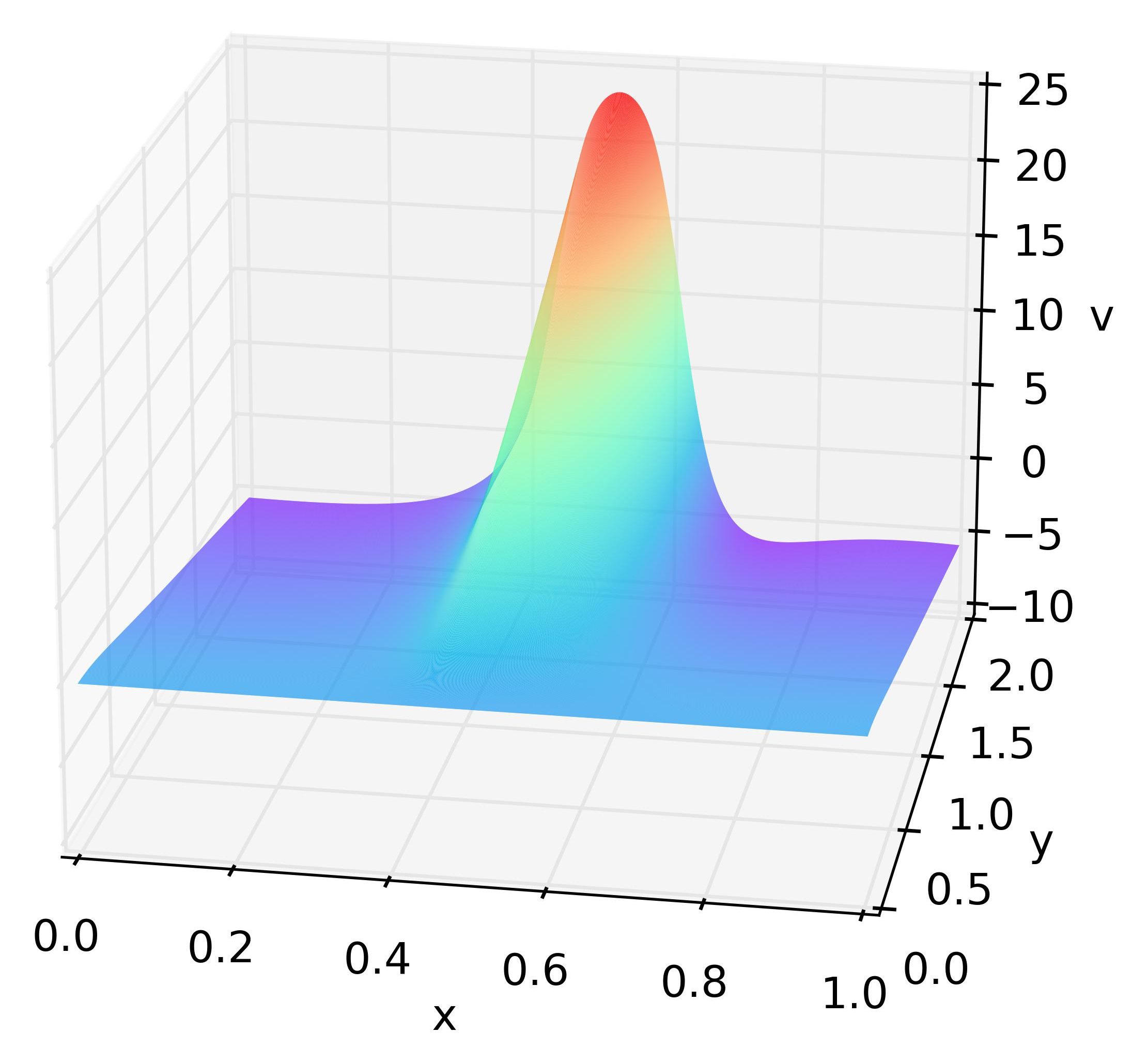}} \hfill
\subfloat[Approximation of $v$, level 2]
{\includegraphics[width=0.32\textwidth]{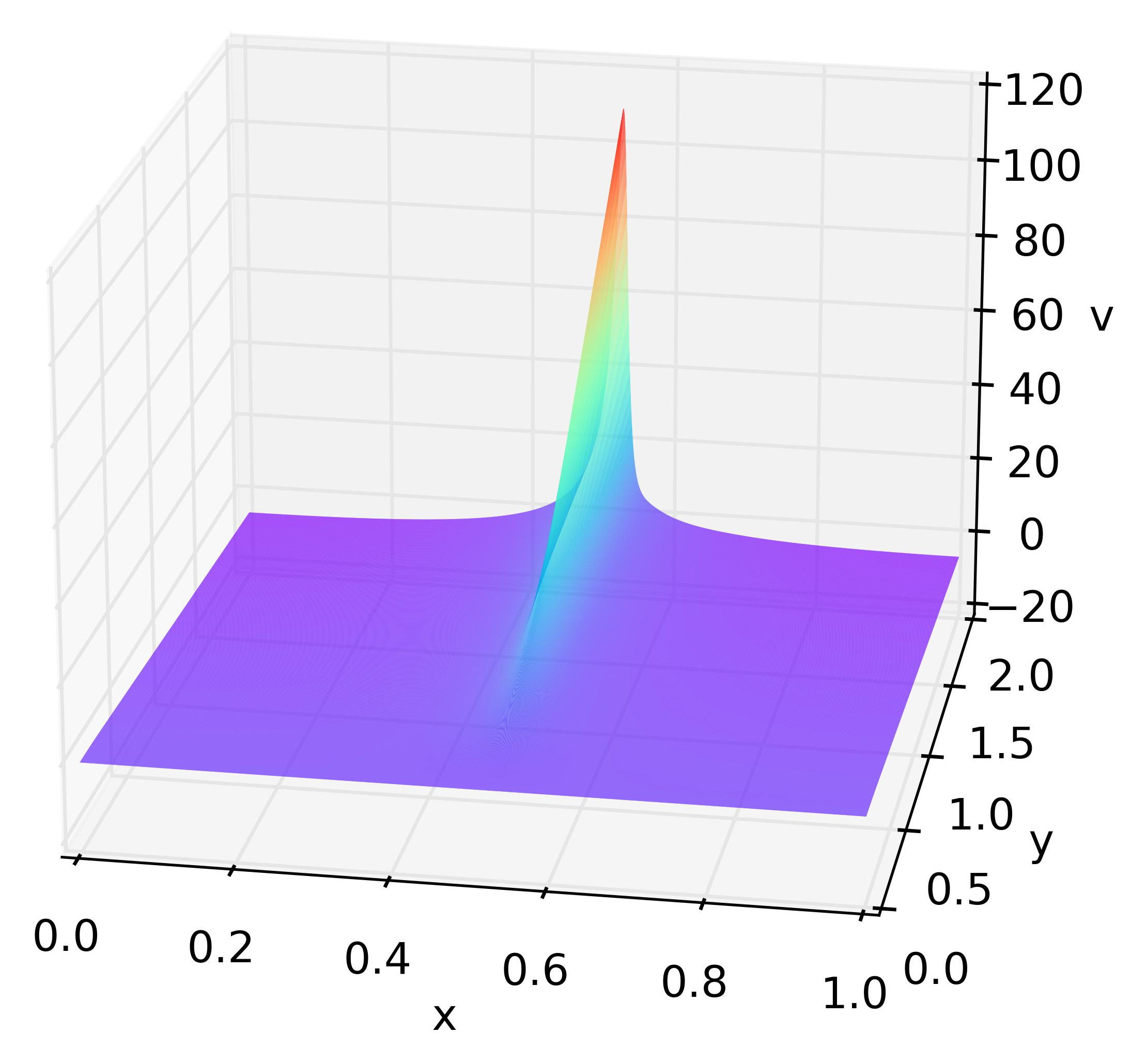}} \hfill
\subfloat[Approximation of $v$, level 3]
{\includegraphics[width=0.32\textwidth]{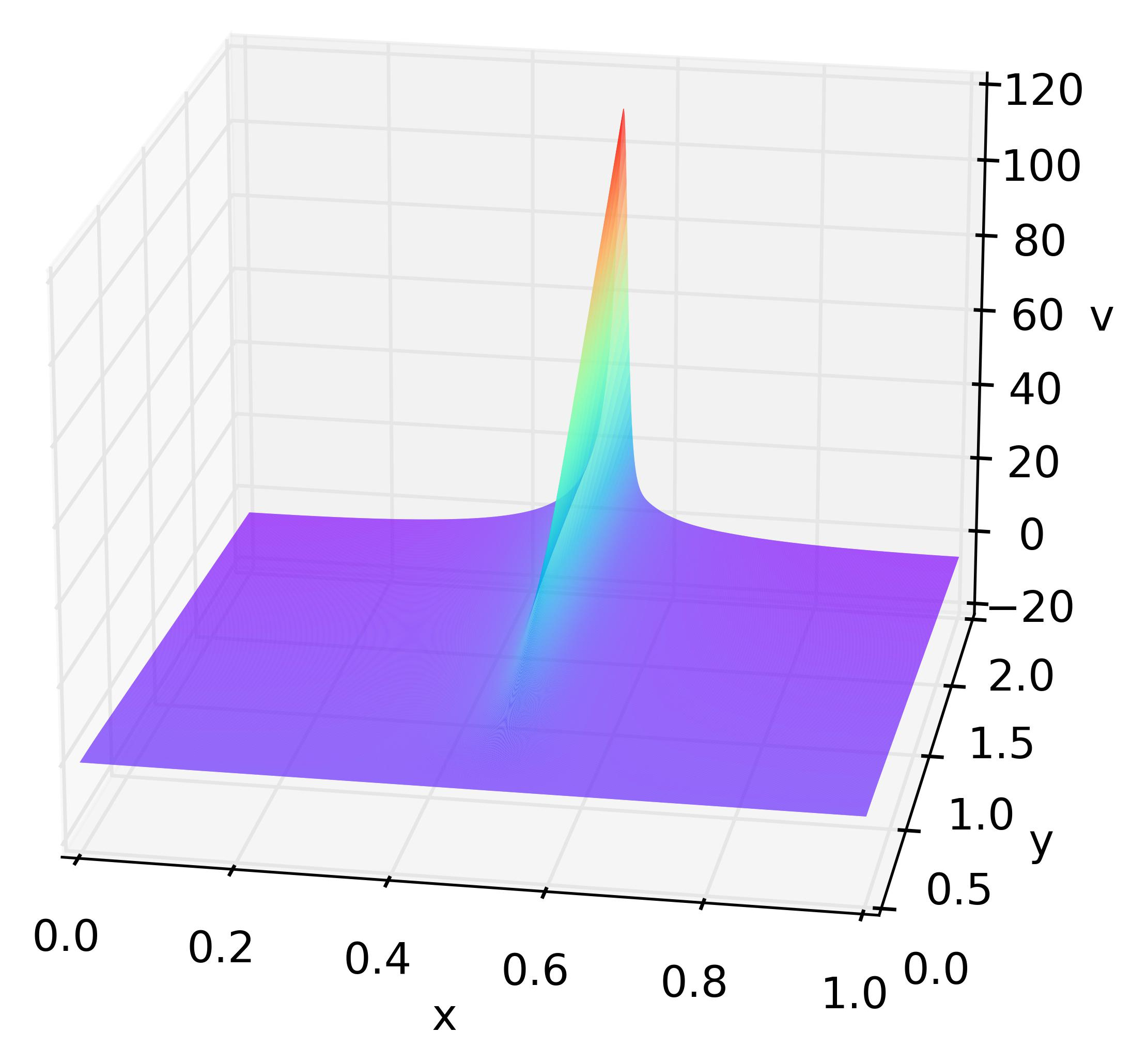}}
\caption{Viscous Prandtl problem at \(t=0.16\): \(10{,}000\) displayed collocation points (top), \(u(x,y,0.16)\) (middle), and \(v(x,y,0.16)\) (bottom).}
\label{fig:ViscousPrandtl2D_Surface}
\end{figure}

\begin{figure}[htbp]
\centering
\subfloat[Approximation of $u$, level 1]
{\includegraphics[width=0.32\textwidth]{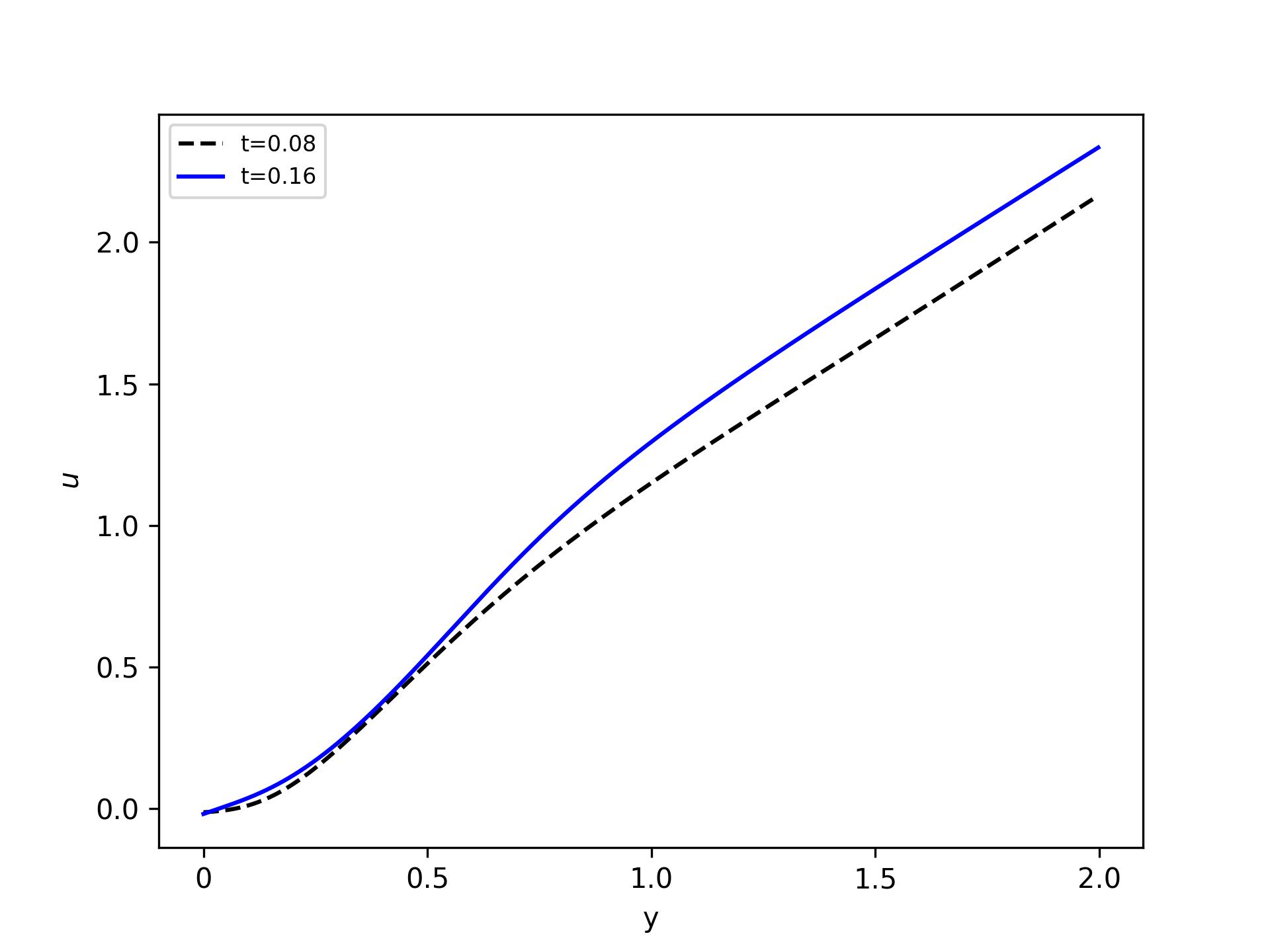}} \hfill
\subfloat[Approximation of $u$, level 2]
{\includegraphics[width=0.32\textwidth]{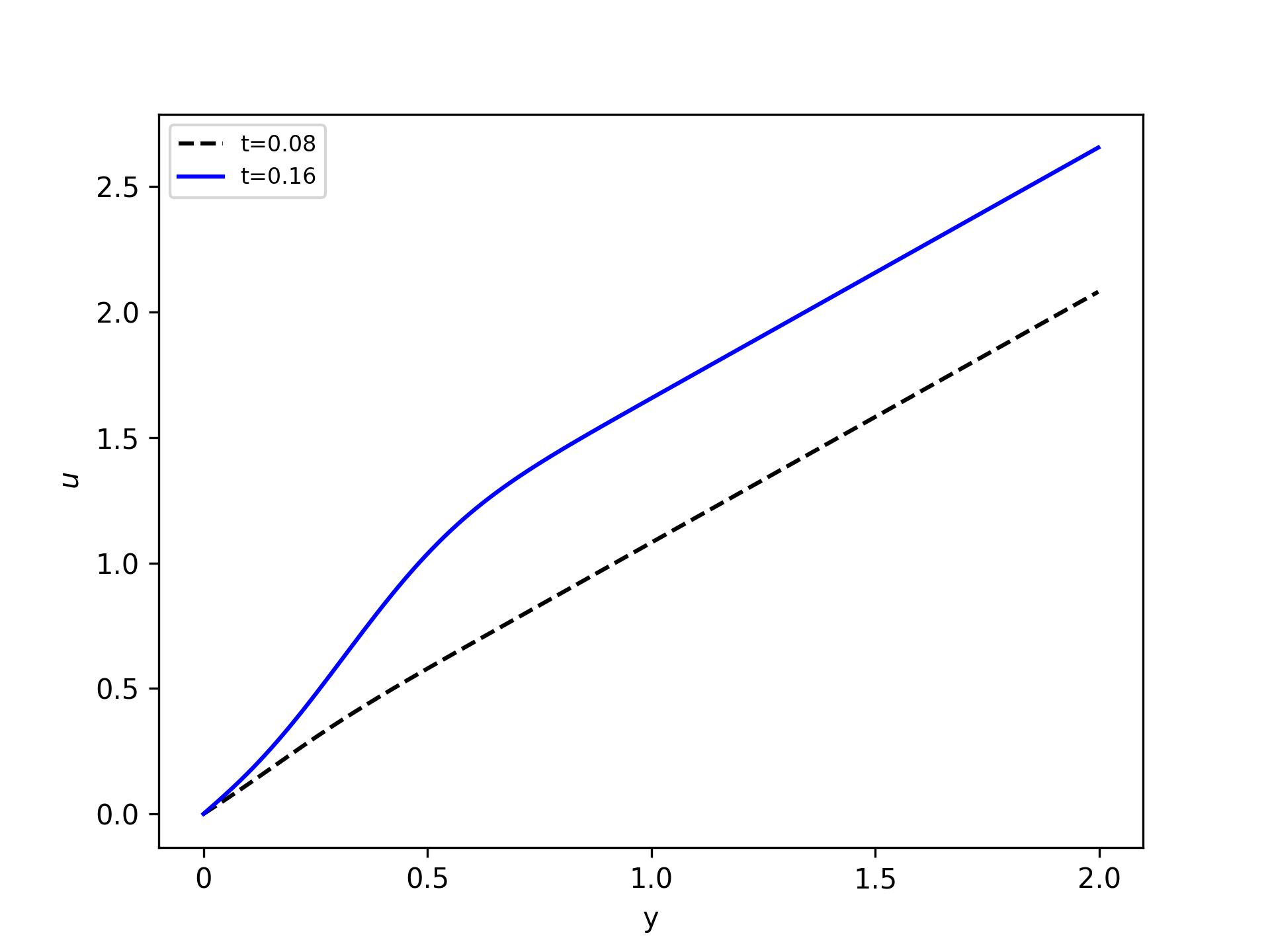}} \hfill
\subfloat[Approximation of $u$, level 3]
{\includegraphics[width=0.32\textwidth]{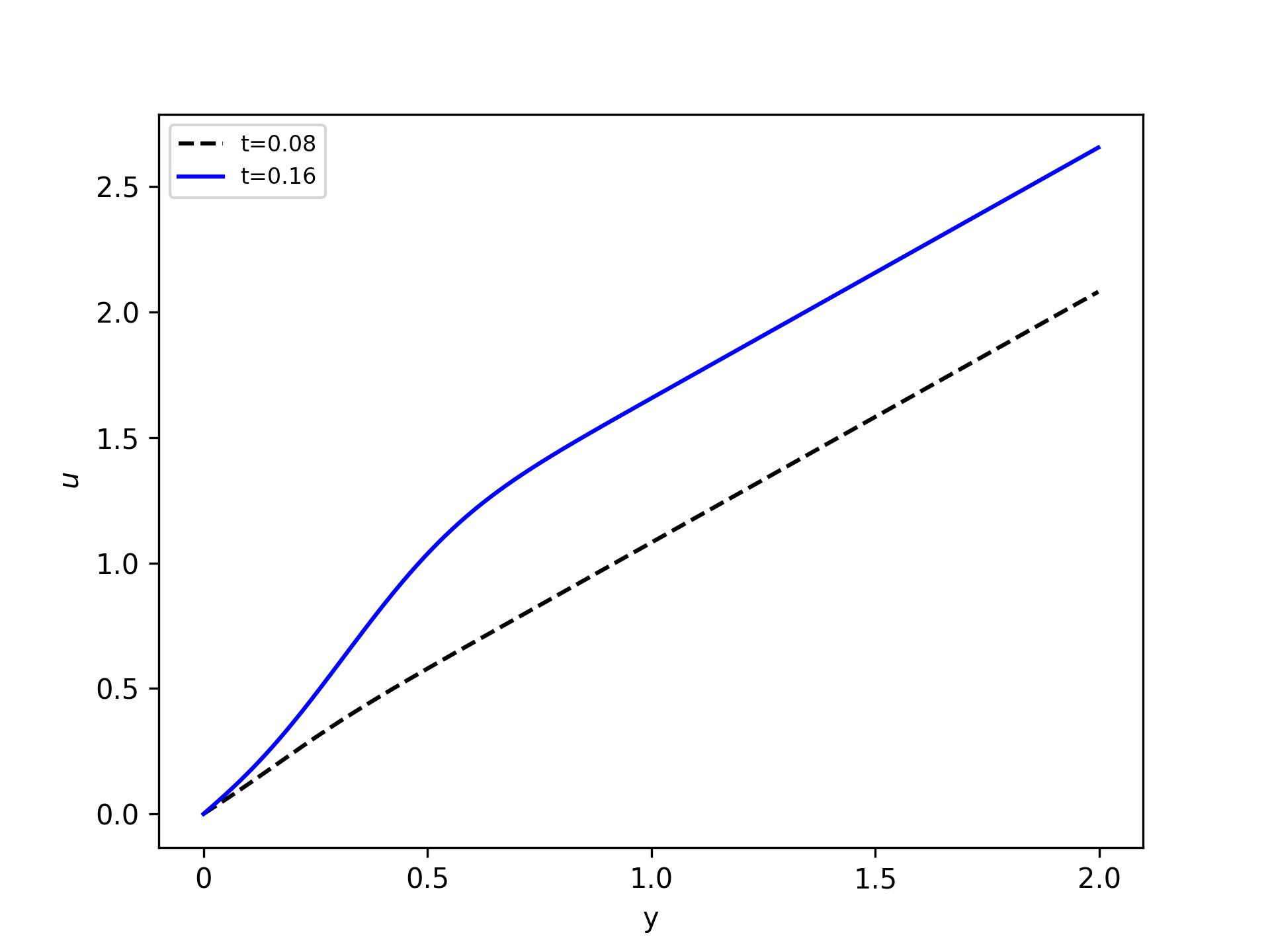}} \hfill
\subfloat[Approximation of $v$, level 1]
{\includegraphics[width=0.32\textwidth]{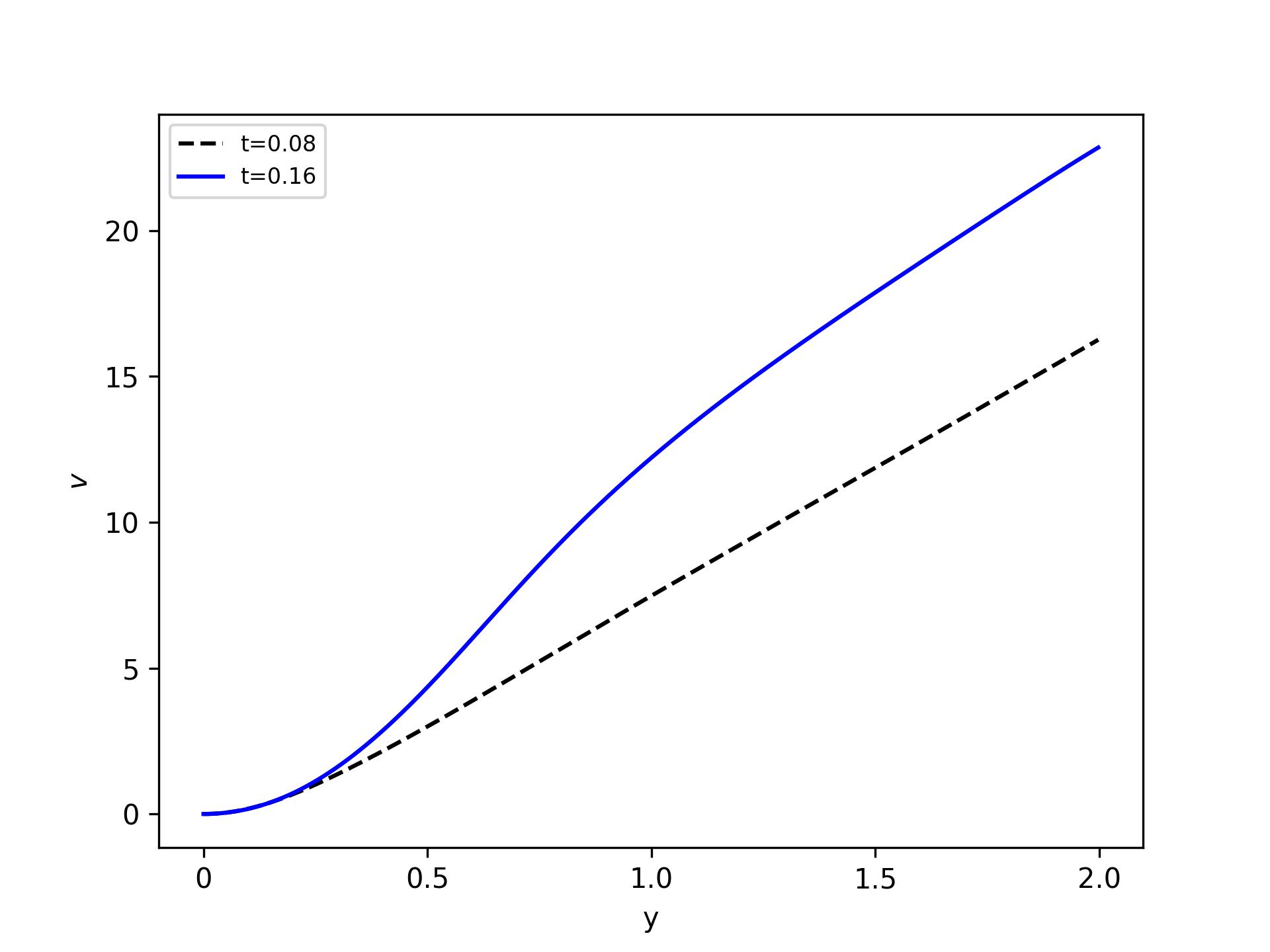}} \hfill
\subfloat[Approximation of $v$, level 2]
{\includegraphics[width=0.32\textwidth]{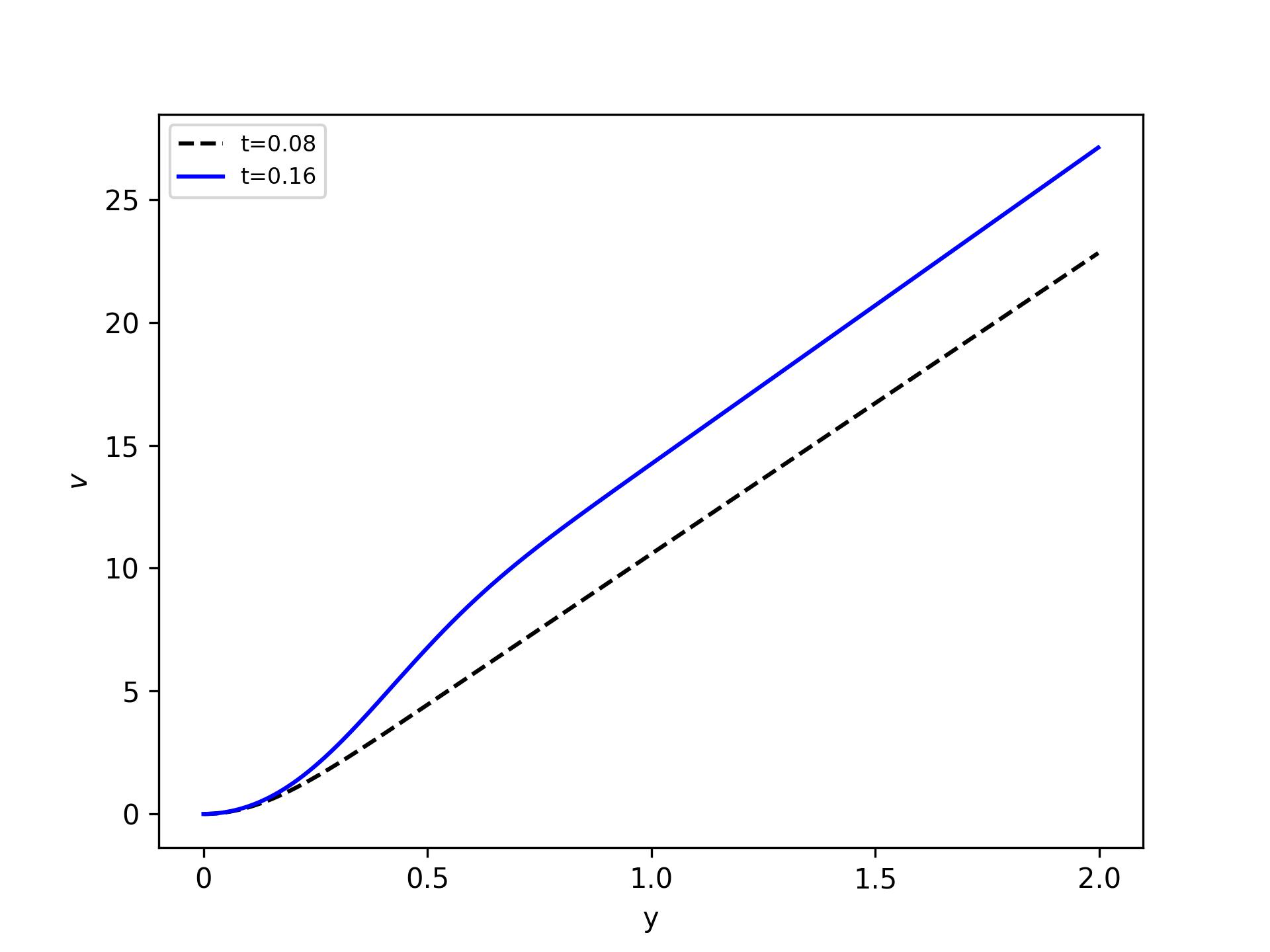}} \hfill
\subfloat[Approximation of $v$, level 3]
{\includegraphics[width=0.32\textwidth]{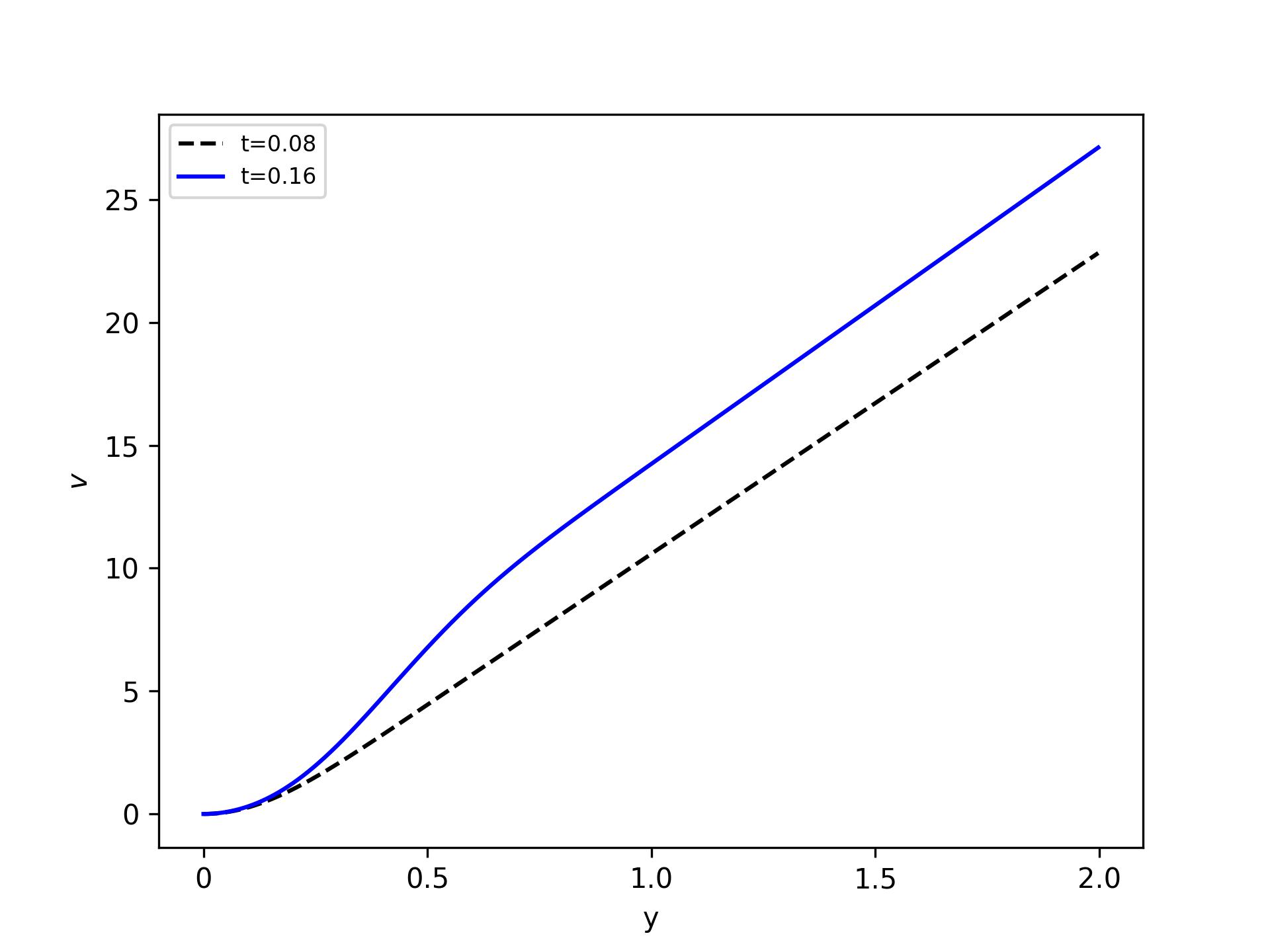}}
\caption{Viscous Prandtl solution at \(x=0.5\): cross-sections of \(u(0.5,y,t)\) (top) and \(v(0.5,y,t)\) (bottom).}
\label{fig:ViscousPrandtl2D_line}
\end{figure}

\begin{figure}[htbp]
\centering
\subfloat[Approximation of $u$, level 1]
{\includegraphics[width=0.32\textwidth]{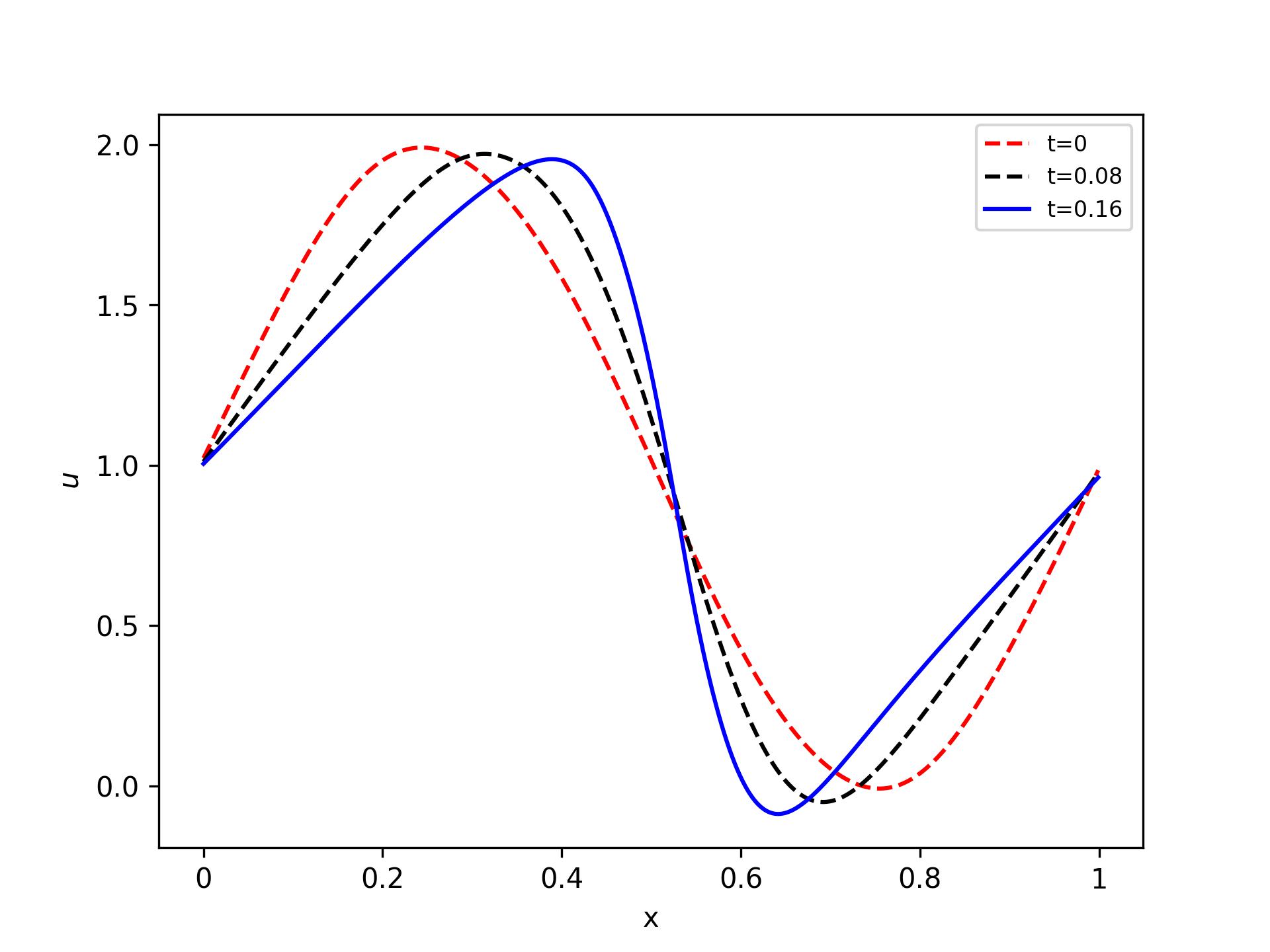}} \hfill
\subfloat[Approximation of $u$, level 2]
{\includegraphics[width=0.32\textwidth]{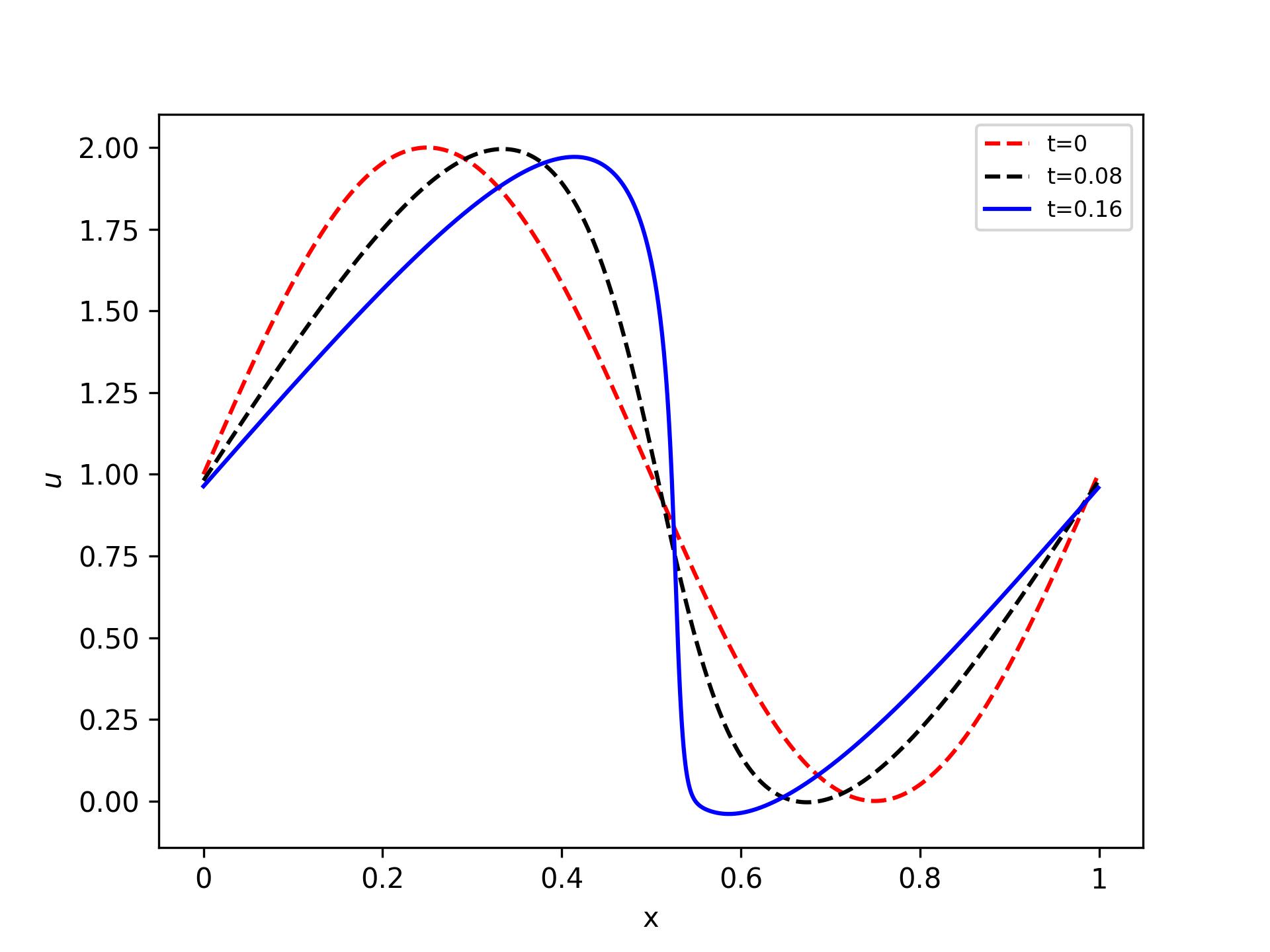}} \hfill 
\subfloat[Approximation of $u$, level 3]
{\includegraphics[width=0.32\textwidth]{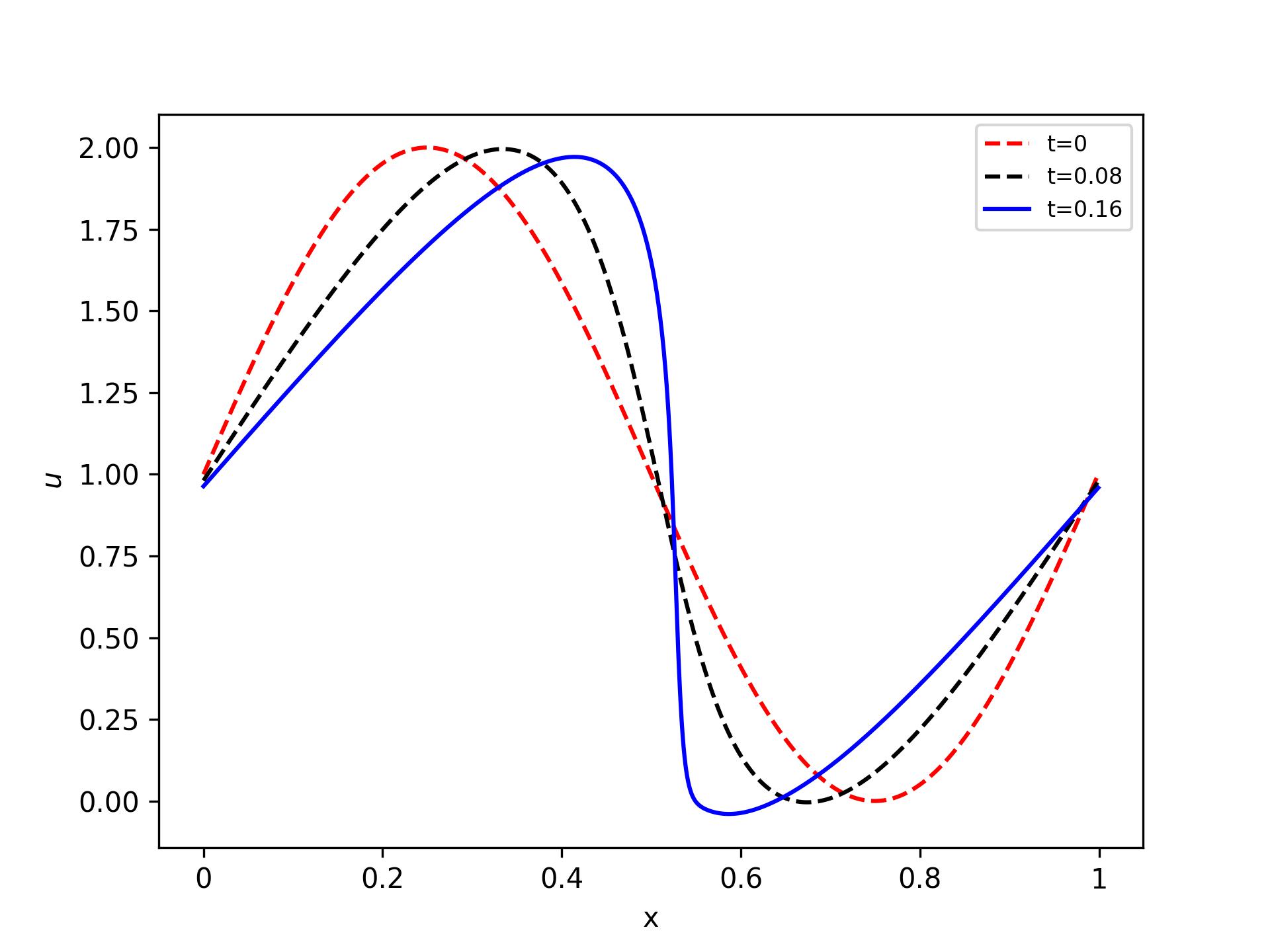}} \hfill
\subfloat[Approximation of $v$, level 1]
{\includegraphics[width=0.32\textwidth]{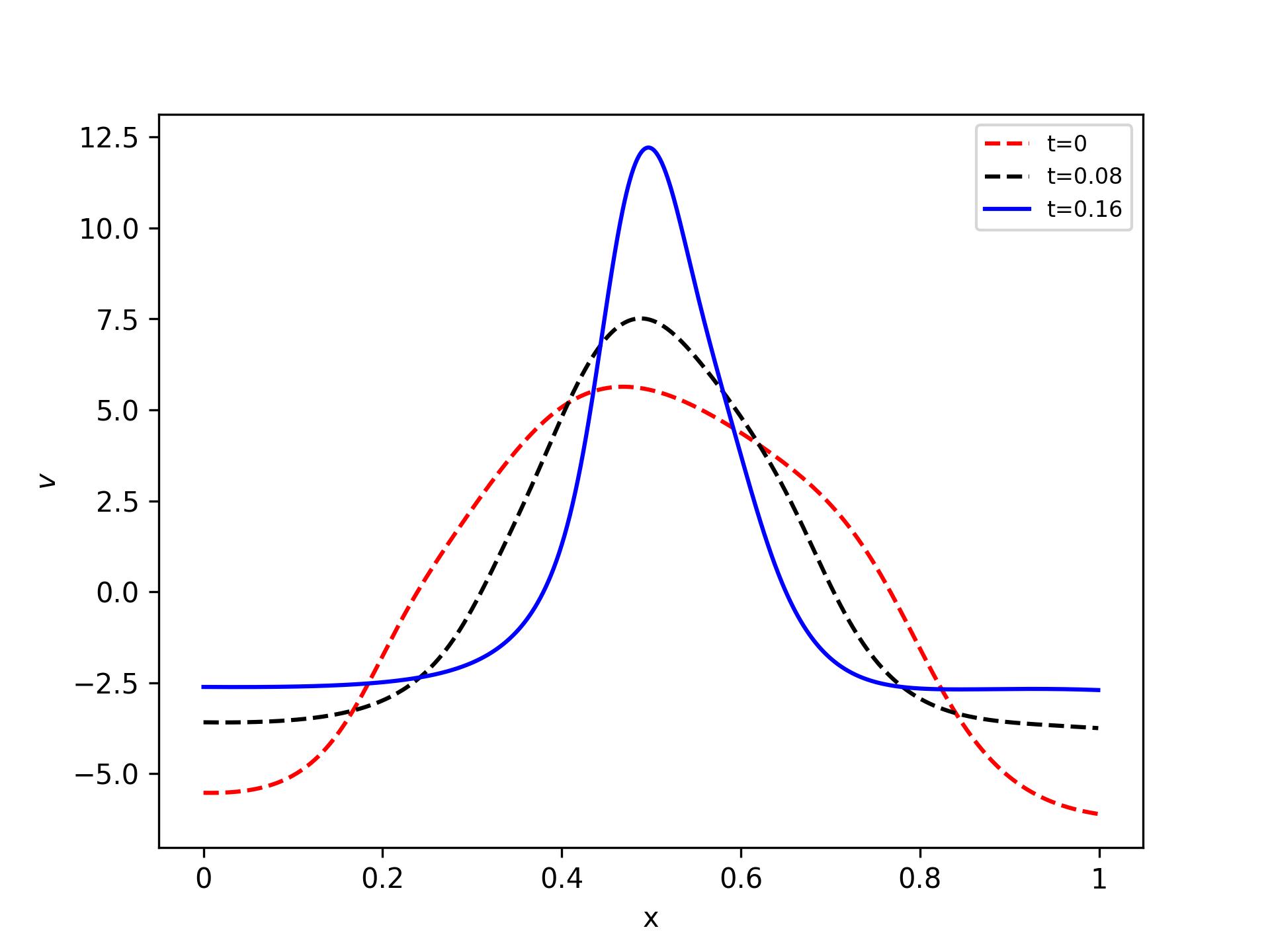}} \hfill
\subfloat[Approximation of $v$, level 2]
{\includegraphics[width=0.32\textwidth]{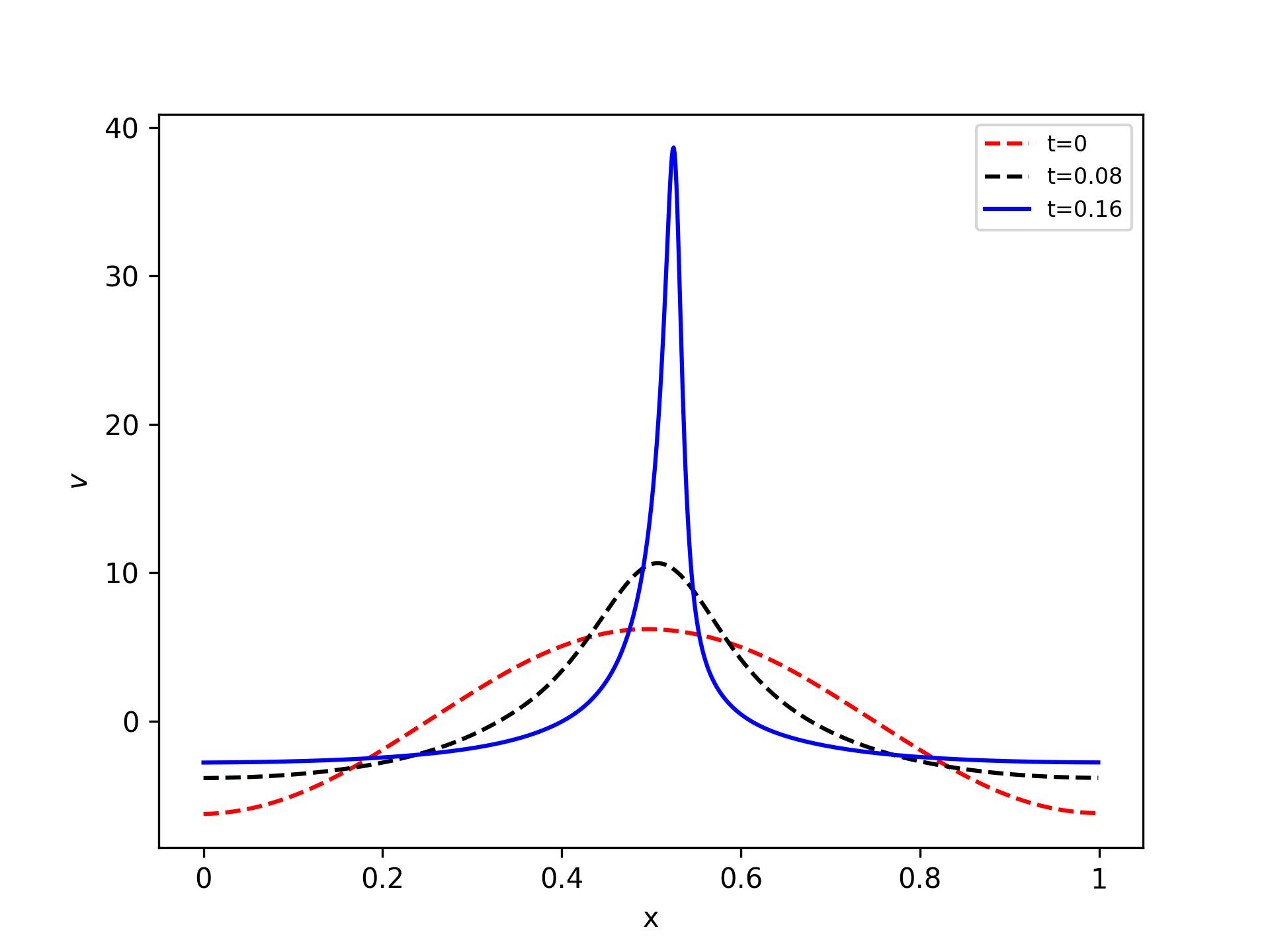}} \hfill
\subfloat[Approximation of $v$, level 3]
{\includegraphics[width=0.32\textwidth]{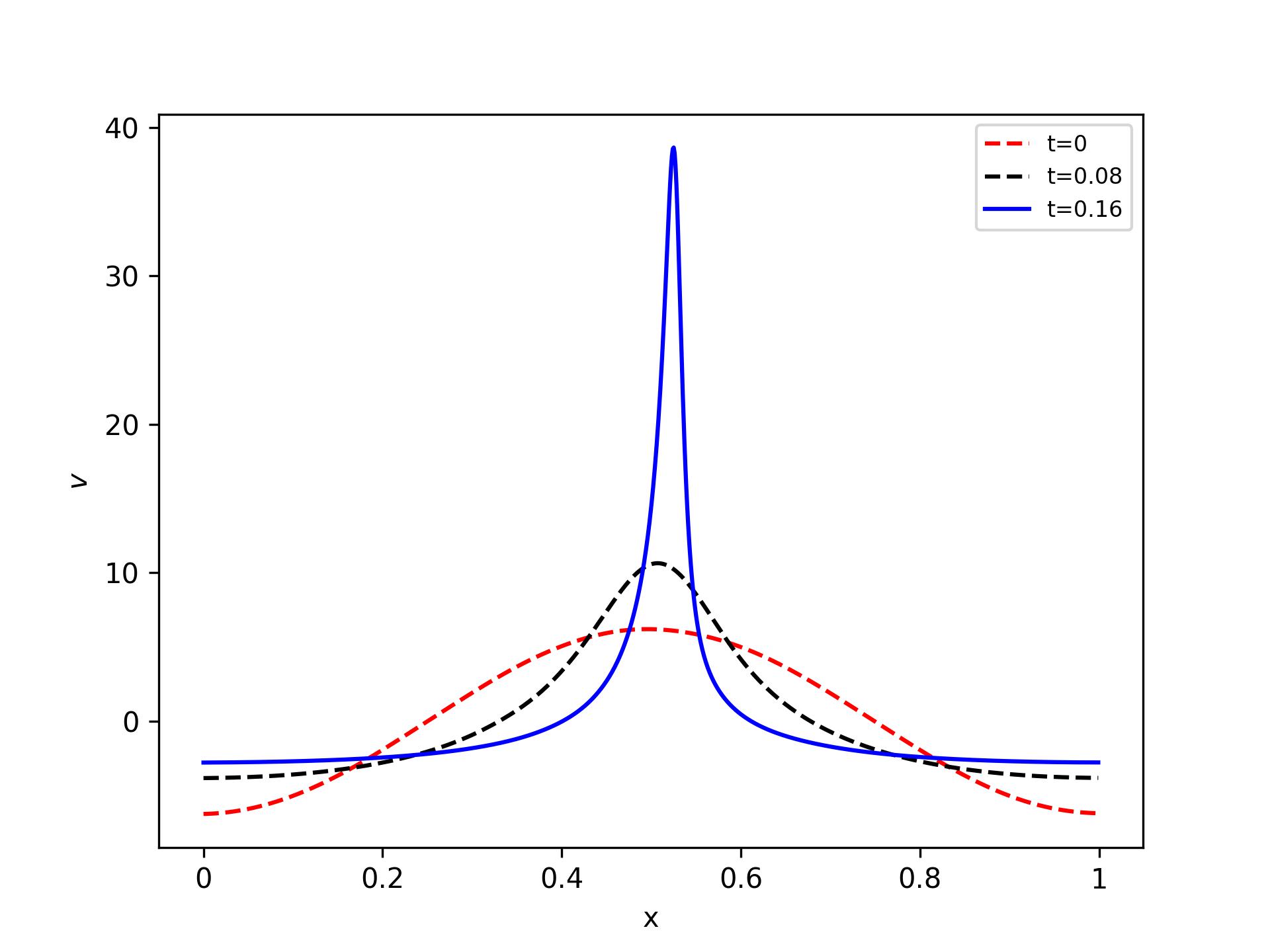}}
\caption{Viscous Prandtl solution at \(y=1\): cross-sections of \(u(x,1,t)\) (top) and \(v(x,1,t)\) (bottom).}
\label{fig:ViscousPrandtl2D_line2}
\end{figure}

\begin{sidewaystable}[p]
\centering

\caption{Hyperparameter settings and relative errors. $N_r$, $N_b$, and $N_t$ denote the numbers of interior residual, boundary/initial, and test points, respectively.}
\label{tab:all_hyperparameters}

\scriptsize
\setlength{\tabcolsep}{2pt}
\renewcommand{\arraystretch}{1.12}

\begin{tabularx}{\textheight}{
  >{\raggedright\arraybackslash}p{2.65cm}
  >{\centering\arraybackslash}p{1.0cm}
  >{\centering\arraybackslash}p{1.25cm}
  >{\centering\arraybackslash}p{1.15cm}
  >{\centering\arraybackslash}p{1.05cm}
  >{\centering\arraybackslash}p{0.85cm}
  >{\centering\arraybackslash}p{3.95cm}
  >{\centering\arraybackslash}p{3.35cm}
  >{\centering\arraybackslash}p{2.45cm}
  >{\centering\arraybackslash}p{1.65cm}
}
\toprule
Problem &
$N_r$ &
$N_b$ &
$N_t$ &
Network &
Levels &
Training schedule &
$(\gamma,s,c,\mu,\eta)$ &
$e_{\infty}(u)$ &
$e_2(u)$
\\
\midrule

2D Poisson equation on the L-shaped domain
&
$7500$
&
$1000$
&
$119600$
&
$40\times4$
&
3
&
\makecell[l]{
L1: SOAP $20\,000$\\
L2: SOAP $10\,000$ + SSB $5\,000$\\
L3: SOAP $2\,500$ + SSB $12\,500$
}
&
$\left(10^{-3},0.5,0.1,4,1\times10^{-3}\right)$
&
$2.158\times10^{-3}$
&
$2.798\times10^{-5}$
\\

2D Poisson equation with a sharp interior layer
&
$10000$
&
$1000$
&
$400\times400$
&
$32\times8$
&
3
&
\makecell[l]{
L1: SOAP $20\,000$\\
L2: SOAP $20\,000$ + SSB $20\,000$\\
L3: SSB $20\,000$
}
&
$\left(1,0.5,1,4,1\times10^{-3}\right)$
&
$4.400\times10^{-7}$
&
$8.197\times10^{-8}$
\\

3D Poisson equation with a sharp interior layer
&
$50000$
&
\makecell{$500$\\per face}
&
$160000$
&
$32\times8$
&
3
&
\makecell[l]{
L1: SOAP $20\,000$ + SSB $10\,000$\\
L2: SOAP $20\,000$ + SSB $10\,000$\\
L3: SSB $20\,000$
}
&
$\left(1,0.5,1,4,1\times10^{-5}\right)$
&
$1.746\times10^{-7}$
&
$7.931\times10^{-8}$
\\


\\

2D Poisson equation with two peaks
&
$10000$
&
$1000$
&
$400\times400$
&
$40\times4$
&
3
&
\makecell[l]{
L1: SOAP $20\,000$ + SSB $10\,000$\\
L2: SOAP $20\,000$ + SSB $20\,000$\\
L3: SSB $10\,000$\\
}
&
$\left(1,0.5,1,4,1\times10^{-3}\right)$
&
$1.524\times10^{-5}$
&
$1.071\times10^{-4}$
\\
3D shock-cube Poisson problem
&
$50000$
&
\makecell{$2000$\\per face}
&
$160000$
&
$40\times4$
&
3
&
\makecell[l]{
L1: SOAP $20\,000$ + SSB $5\,000$\\
L2: SOAP $10\,000$ + SSB $5\,000$\\
L3: SSB $15\,000$
}
&
$\left(1,0.5,1,4,1\times10^{-5}\right)$
&
$1.471\times10^{-7}$
&
$1.027\times10^{-7}$
\\
2D biharmonic equation
&
$10000$
&
$256$
&
$400\times400$
&
$40\times4$
&
3
&
\makecell[l]{
L1: SOAP $20\,000$\\
L2: SOAP $5\,000$ + SSB $10\,000$\\
L3: SSB $1200$
}
&
$\left(0,1,1,4,1\times10^{-3}\right)$
&
$1.651\times10^{-4}$
&
$4.774\times10^{-5}$

\\
2D Helmholtz equation
&
$10000$
&
$1000$
&
$400\times400$
&
$40\times4$
&
3
&
\makecell[l]{
L1: SOAP $20\,000$\\
L2: SOAP $20\,000$ + SSB $20\,000$\\
L3: SSB $15\,000$
}
&
$\left(1,0.5,1,4,5\times10^{-3}\right)$
&
$9.469\times10^{-6}$
&
$7.359\times10^{-5}$
\\

2D linear elasticity equation
&
$10000$
&
$1000$
&
$160000$
&
$40\times4$
&
3
&
\makecell[l]{
L1: SOAP $20\,000$\\
L2: SOAP $10\,000$ + SSB $5\,000$\\
L3: SSB $10\,000$
}
&

$\left(1,5,0.1,4,1\times10^{-3}\right)$
&
\makecell[c]{
$u_x:\ 2.013\times10^{-5}$\\
$u_y:\ 9.142\times10^{-5}$\\
$\sigma_{xx}:\ 2.675\times10^{-7}$
}
&
\makecell[c]{
$1.275\times10^{-5}$\\
$2.292\times10^{-5}$\\
$8.834\times10^{-8}$
}

\\

2D unsteady viscous Prandtl equation
&
$100000$
&

 ----\textsuperscript{a}

&
----
&
$40\times4$
&
3
&
\makecell[l]{
L1: SOAP $20\,000$\\
L2: SOAP $20\,000$ + SSB $30\,000$\\
L3: SSB $50\,000$
}
&
$\left(1,0.5,1,4,1\times10^{-3}\right)$
&
----\textsuperscript{b}
&
----

\\

\bottomrule
\end{tabularx}

\vspace{2mm}

\begin{minipage}{\textheight}
\footnotesize
\textit{Notes.}
SOAP and SSB denote the SOAP and SSBroyden optimizers, respectively.
For the linear elasticity problem, the gradient term
$\lvert\nabla u\rvert$ in the monitor function is replaced by
$\sigma_{xx}$.

\textsuperscript{a} For detailed settings, please refer to Section \ref{sec:ViscousPrandtl_2D}.

\textsuperscript{b} This example demonstrates its advantages by comparing the results with those from Reference \cite{hong2003singularity}.

\end{minipage}

\end{sidewaystable}

\section{Conclusions}
\label{sec:conclusion}
This work develops a multi-level framework that combines history-dependent residual-gradient sampling with additive neural corrections. Coupled with SOAP and SSB optimization, the proposed approach consistently improves the accuracy of the tested PINN solvers for problems involving steep gradients, localized singularities, and reduced regularity. Under the stated assumptions, we establish approximation properties and sampling-dependent error estimates for the multi-level scheme. Numerical experiments on a range of elliptic, elasticity, and time-dependent problems demonstrate the effectiveness of the adaptive sampling and successive correction mechanisms. Future work will focus on extending the framework to broader classes of neural PDE solvers and on developing more flexible and problem-adaptive sampling strategies.

\section*{Acknowledgment}
This research is partially sponsored by  National Natural Science Foundation of China (No.12371434, No. U25A20200), National Key R \& D Program of China (No.2022YFE03040002), Sichuan University Research Project on Teaching Reform in Graduate Education (No. SCYZTD2025003). 

\section*{Data Availability Statement}
The code is publicly available at \url{https://github.com/YangYuSCU/MLF}.

\section*{Conflict of Interest}
The authors have no conflicts to disclose.

\vskip 1.0cm

\section*{Appendix A}
\label{sec:AppendixA}

\begin{proof}[\textbf{Proof of Lemma~\ref{lem:1}}]
Let
\[
\bu(\bx)=W_{\rm out}\bh(\bx)+\bb_{\rm out}
\]
be the affine output representation of $\bu\in \mathcal D_1^{NN}$. Setting
$W_{\rm out}=\bb_{\rm out}=0$ gives the zero function. For $\xi\in[0,1]$,
replacing $(W_{\rm out},\bb_{\rm out})$ by
$(\xi W_{\rm out},\xi \bb_{\rm out})$ yields $\xi \bu$ while preserving the prescribed
parameter bounds. Hence the assertion follows.
\end{proof}

\begin{proof}[\textbf{Proof of Theorem~\ref{thm:2}}]
Set $\bU_1:=\bu_1$. For each $2\leq i\leq k$, let
$\bU_i:=\widehat{\bU_i}\in\mathcal D_i^{\mathrm{NN}}$ be the best approximation given by Assumption~\ref{asu:u_opt}, and define
\[
\bu_i:=\bU_i-\bU_{i-1}\in\mathcal D_1^{\mathrm{NN}}.
\]
Membership follows from the definition of $\mathcal D_i^{NN}$. Since $\mathcal D_1^{\mathrm{NN}}$ is star-shaped with respect to the origin,
Lemma~\ref{lem:1} implies that, for any $\lambda\in[0,1]$,
$$\bU_i-\lambda\bu_i=\bU_{i-1}+(1-\lambda)\bu_i\in\mathcal D_i^{NN}.$$ 

By the best-approximation property of $\bU_i$
\begin{equation}\label{eq:thm2_9}
\begin{aligned}
0&\leq\|\bU_i-\lambda\bu_i-\bu^*\|_{\mathbb H}^2-\|\bU_i-\bu^*\|_{\mathbb H}^2=-2\lambda\langle\bU_i-\bu^*,\bu_i\rangle_{\mathbb H}+\lambda^2\|\bu_i\|_{\mathbb H}^2.
\end{aligned}
\end{equation}
Dividing by $\lambda>0$ and letting $\lambda\to0^+$ gives $\langle\bU_i-\bu^*,-\bu_i\rangle\geq0$. Consequently,
\begin{equation}\label{eq:thm2_6}
\begin{aligned}
\|\bU_{i-1}-\bu^*\|_{\mathbb H}^2
&=\|\bU_i-\bu^*\|_{\mathbb H}^2+\|\bu_i\|_{\mathbb H}^2+2\langle\bU_i-\bu^*,-\bu_i\rangle_{\mathbb H}\\
&>\|\bU_i-\bu^*\|_{\mathbb H}^2,
\end{aligned}
\end{equation}
whenever $\bu_i\ne0$, which proves the result. 
\end{proof}

\begin{proof}[\textbf{Proof of Theorem~\ref{thm:3}}]
Since $\mathcal A$ and $\mathcal B$ are linear and $\bu^\ast$ satisfies \eqref{eq:gen_PDE}, Assumption~\ref{asu:Normrelations} gives
\[
C_1\|\bU_k-\bu^*\|_{L^2(\Omega)}\leq\|r\|_{L^2_{\rho_{k-1}^{\mathrm{MLS}}}(\Omega)}+\|b\|_{L^2(\partial\Omega)}.\]
Hence,
\begin{equation}
\label{eq:stability-thm3}
\|\bU_k-\bu^\ast\|_{L^2(\Omega)}
\leq
\frac{\sqrt{2}}{C_1}
\left(
\|r\|_{L^2_{\rho_{k-1}^{\mathrm{MLS}}}(\Omega)}^2
+\|b\|_{L^2(\partial\Omega)}^2
\right)^{1/2}.
\end{equation}
By Assumption~\ref{asu:r-boundary}, the functions in
$\mathcal G_{k,r}^{\mathrm{NN}}$ and
$\mathcal G_{k,b}^{\mathrm{NN}}$ are uniformly bounded by
$B_{k,r}^2$ and $B_{k,b}^2$, respectively. 
Applying Lemma~\ref{lem:Uniform laws of large numbers-Rademacher} 
to the two classes with tolerance $\delta/2$ yields, uniformly over $\bu\in\mathcal D_k^{\mathrm{NN}}$,
\begin{align}
\|r\|_{L^2_{\rho_{k-1}^{\mathrm{MLS}}}(\Omega)}^2&\leq\|r\|_{2,N_r}^2+2\mathcal R_{N_r}(\mathcal G_{k,r}^{\mathrm{NN}})+\delta/2, \label{eq:residual-generalization-r}\\
\|b\|_{L^2(\partial\Omega)}^2&\leq\|b\|_{2,N_b}^2+2\mathcal R_{N_b}(\mathcal G_{k,b}^{\mathrm{NN}})+\delta/2. \label{eq:residual-generalization-b}
\end{align}
The corresponding events have probabilities at least
\[
1-\exp\left(
-\frac{N_a\delta^2}{8B_{k,a}^4}
\right),
\qquad a\in\{r,b\}.
\]
Since the interior and boundary sample sets are independent, the two events hold simultaneously with probability at least the product of these bounds. Substituting \eqref{eq:residual-generalization-r}--\eqref{eq:residual-generalization-b} into \eqref{eq:stability-thm3}
gives \eqref{eq:errorestimate}.
\end{proof}

\begin{proof}[\textbf{Proof of Theorem~\ref{thm:4}}]
For each fixed \(k\), the uniform deviation estimates underlying Theorem~3.6, together with the assumed decay of the Rademacher complexities, yield
\[
\sup_{\bu\in\mathcal D_k^{\mathrm{NN}}}
\left|
\mathcal L_{N_r,N_b}(\bu)-\mathcal L_k(\bu)
\right|
\xrightarrow[]{\mathbb P}0,
\]
where
\[
\mathcal L_k(\bu)
:=
\|\mathcal A\bu-f\|_{L^2_{\rho_{k-1}^{\mathrm{MLS}}}(\Omega)}^2
+\|\mathcal B\bu-g\|_{L^2(\partial\Omega)}^2.
\]
Since $\bU_k^*$ minimizes the empirical loss $\mathcal L_{N_r,N_b}$,
\[
\mathcal L_{k}(\bU_k^*)
\leq
\mathcal L_k(\widetilde{\bU_k})
+o_{\mathbb P}(1).\]
By the stability estimate in Assumption~\ref{asu:Normrelations},
\[
\|\bU_k^*-\bu^\ast\|_{L^2(\Omega)}
\leq
\frac{\sqrt{2}}{C_1}
\mathcal L_k(\bU_k^*)^{1/2}.
\]
Since $\bu^\ast$ solves \eqref{eq:gen_PDE}, linearity and the continuity
estimate yield
\[
\mathcal L_k(\widetilde{\bU_k})^{1/2}
\leq
\|\mathcal A(\widetilde{\bU_k}-\bu^\ast)\|_
{L^2_{\rho_{k-1}^{\mathrm{MLS}}}(\Omega)}
+
\|\mathcal B(\widetilde{\bU_k}-\bu^\ast)\|_{L^2(\partial\Omega)}
\leq
C_2\|\widetilde{\bU_k}-\bu^\ast\|_X.
\]
Therefore, for each fixed $k$,
\begin{equation}\label{eq:fixed-k-asymptotic}
\|\bU_k^*-\bu^\ast\|_{L^{2}(\Omega)}
\leq
\frac{\sqrt{2}C_2}{C_1}
\|\widetilde{\bU_k}-\bu^\ast\|_X
+o_{\mathbb P}(1).
\end{equation}
The conclusion follows from Assumption~\ref{asu:opt-*}, namely,
\[
\|\widetilde{\bU_k}-\bu^\ast\|_X\to\delta_d
\qquad\text{as }k\to\infty.
\]
That is, for every $\varepsilon>0$,
\begin{equation}
\lim_{k\to\infty}
\limsup_{N_r,N_b\to\infty}
\mathbb P\left(
\|{\bU}^{*}_{k}-\bu^\ast\|_{L^{2}(\Omega)}
>
\frac{\sqrt{2}C_2}{C_1}\delta_d+\varepsilon
\right)
=0.
\end{equation}
\end{proof}

\bibliographystyle{cas-model2-names}

\bibliography{MLF}






\end{document}